\documentclass[12pt]{amsart}
\pdfoutput=1
\usepackage{amsthm,amssymb,amsmath,array,url,microtype,booktabs}
\usepackage[margin=1in,a4paper]{geometry}
\usepackage[dvips]{graphicx}
\usepackage{mathrsfs}
\usepackage[all]{xy}
\usepackage[colorlinks,bookmarks=false,ocgcolorlinks]{hyperref}
\hypersetup{citecolor=blue,linkcolor=red,urlcolor=[rgb]{0,0.5,0}}
\def\arXiv#1{arXiv:\href{http://arXiv.org/abs/#1}{#1}}
\usepackage{multirow} 
\usepackage[clockwise]{rotating}
\usepackage{caption}
\usepackage[parfill]{parskip}    

\theoremstyle{plain}
\newtheorem{theorem}{Theorem}
\newtheorem{corollary}[theorem]{Corollary}
\newtheorem{lemma}[theorem]{Lemma}
\newtheorem{proposition}[theorem]{Proposition}

\newtheorem{question}[theorem]{Question}
\numberwithin{theorem}{section}

\theoremstyle{definition}
\newtheorem{definition}[theorem]{Definition}
\newtheorem{example}[theorem]{Example}

\theoremstyle{remark}
\newtheorem{remark}[theorem]{Remark}

\numberwithin{equation}{section}

\makeatletter
\def\theenumi{\@roman\c@enumi}
\makeatother

\newcommand{\Z}{\mathbb{Z}}
\newcommand{\Q}{{\mathbb Q}}
\newcommand{\C}{{\mathbb C}}
\newcommand{\A}{{\mathbb A}}

\newcommand{\F}{{\mathbb F}}

\renewcommand{\H}{{\mathbb H}}
\newcommand{\Qbar}{{\kern.1ex\overline{\kern-.1ex\Q\kern-.1ex}\kern.1ex}}
\newcommand{\sO}{{\mathcal O}}
\newcommand{\E}{{\mathcal E}}
\newcommand{\Proj}{{\mathbb P}}
\let\P\Proj

\newcommand{\Gal}{\operatorname{Gal}}
\newcommand{\Hom}{\operatorname{Hom}}

\newcommand{\Pic}{\operatorname{Pic}}

\newcommand{\NS}{\operatorname{NS}}

\newcommand{\Km}{\mathit{Km}}
\newcommand{\Ino}{\mathit{Ino}}
\newcommand{\rank}{\operatorname{rank}}
\newcommand{\II}{\mathrm{II}}
\newcommand{\I}{\mathrm{I}}
\newcommand{\IV}{\mathrm{IV}}
\newcommand{\kbar}{\overline{k}}
\newcommand{\MoW}{Mordell-\mbox{\kern-.12em}Weil}
\DeclareMathOperator{\GL}{GL}

\DeclareMathOperator{\SL}{SL}
\DeclareMathOperator{\PSL}{PSL}
\DeclareMathOperator{\MW}{MW}

\let\isom=\simeq  
\newcommand{\<}{\langle} 
\renewcommand{\>}{\rangle}
\newcommand{\ph}{\varphi}



\title[Elliptic K3 surfaces associated with the product of two
  elliptic curves]{Elliptic K3 surfaces associated with the product of
  two elliptic curves: Mordell-Weil lattices and their fields of
  definition}

\author{Abhinav Kumar} 
\address{Department of Mathematics, Massachusetts Institute of
  Technology, Cambridge, MA 02139, USA}
\curraddr{Department of Mathematics, Stony Brook University, Stony
  Brook, NY 11794, USA}
\email{thenav@gmail.com}

\author{Masato Kuwata}
\address{Faculty of Economics, Chuo University, 742-1 Higashinakano, 
Hachioji-shi, Tokyo 192-0393, Japan}
\email{kuwata@tamacc.chuo-u.ac.jp}

\date{June 30, 2017}

\makeatletter
\def\subsection{\@startsection{subsection}{2}%
  \z@{0\linespacing\@plus0\linespacing}{.1\linespacing}%
  {\normalfont\selectfont\bfseries\itshape}}
\makeatother

\begin{document}

\begin{abstract}
To a pair of elliptic curves, one can naturally attach two K3
surfaces: the Kummer surface of their product and a double cover of
it, called the Inose surface. They have prominently featured in many
interesting constructions in algebraic geometry and number
theory. There are several more associated elliptic K3 surfaces,
obtained through base change of the Inose surface; these have been
previously studied by Kuwata. We give an explicit description of the
geometric Mordell-Weil groups of each of these elliptic surfaces in
the generic case (when the elliptic curves are non-isogenous). In the
non-generic case, we describe a method to calculate explicitly a
finite index subgroup of the Mordell-Weil group, which may be
saturated to give the full group. Our methods rely on several
interesting group actions, the use of rational elliptic surfaces, as
well as connections to the geometry of low degree curves on cubic and
quartic surfaces. We apply our techniques to compute the full
Mordell-Weil group in several examples of arithmetic interest, arising
from isogenous elliptic curves with complex multiplication, for which
these K3 surfaces are singular.
\end{abstract}

\maketitle

\section{Introduction}

Elliptic K3 surfaces play an important role in the study of the
geometry, arithmetic and moduli of K3 surfaces \cite{P-SS,
  Shioda-Inose, Artin-Swinnerton-Dyer, Morrison, Bogomolov-Tschinkel}.

An elliptic surface $\E$ fibered over $\Proj^1$ with section, over a
field $k$, may be described by a Weierstrass equation of the form
\[
y^2 + a_1(t) xy + a_3(t) y = x^3 + a_2(t) x^2 + a_4(t) x + a_6(t),
\]
where the $a_i(t)$ are rational functions (or even polynomials). Let
us assume that the elliptic fibration has at least one singular
fiber. The following is a fundamental question:
\begin{question}
Find generators for the (finitely generated) Mordell-Weil group
$\E(\Proj^1)$.
\end{question}
Usually, one is interested in the {\em geometric} Mordell-Weil group
$\E(\kbar(t))$, as well as its field of definition and the Galois
action of $\Gal(\kbar/k)$.

A theorem of Shioda and Tate connects the Mordell-Weil group with the
Picard group or the N\'eron-Severi group of $\E$ (note that linear
equivalence and algebraic equivalence coincide). Namely, there is an
intersection pairing on $\NS(\E)$, making it into a Lorentzian
lattice. The class of the zero section $O$ and the fiber $F$
contribute a unimodular sublattice of signature $(1,1)$, which is
therefore either the hyperbolic plane $U$ or the odd lattice $\I_{1,1}
= \langle 1 \rangle \oplus \langle -1 \rangle$, depending on the Euler
characteristic $\chi(\sO_\E)$. Furthermore, every reducible fiber over
a point $v \in \Proj^1(\kbar)$ contributes the negative of a root
lattice $T_v$ to $\NS(\E)$. Let the {\em trivial lattice} $T$ be
defined as $(\Z O + \Z F) \oplus (\bigoplus T_v)$. The theorem says
that the Mordell-Weil group $\E(\Proj^1)$ is isomorphic to
$NS(\E)/T$. In addition, the natural isomorphism induces an isometry
of lattices, once we mod out by torsion.

It guarantees that determination of the Mordell-Weil group is
equivalent to finding the Picard group or the N\'eron-Severi lattice
of the K3 surface. The theory of Mordell-Weil lattices has found
numerous applications in recent years, from construction of
record-breaking dense lattices to finding high rank elliptic curves to
the inverse Galois problem.

Recently, algorithms have been outlined for the basic question above
(see \cite{Poonen-et-al}, or for the case of elliptic K3 surfaces
\cite{Charles}); however, these algorithms require point counting over
large finite fields, and therefore are not practicable in most cases.

In this paper, we solve this question for several families of K3
surfaces of arithmetic and geometric interest. Namely, let $E_1$ and
$E_2$ be two elliptic curves, and form the Kummer surface of their
product $\Km(E_1 \times E_2)$. This K3 surface carries a lot of the
arithmetic information of the product abelian surface. It has several
different elliptic fibrations \cite{Oguiso, Kuwata-Shioda}, but one in
particular has been the focus of a lot of attention in arithmetic
algebraic geometry. This elliptic fibration (to be described below)
has two reducible fibers of type $\IV^*$ if $E_1$ and $E_2$ are
non-isomorphic. By taking a base change along an appropriate double
cover $\Proj^1 \to \Proj^1$ (namely $t \mapsto t^2$, where $t$ is the
elliptic parameter, chosen to place the $\IV^*$ fibers at $t = 0$ and
$t = \infty$), a natural double cover of the Kummer surface can be
formed, which is also an elliptic K3 surface; it is called the Inose
surface and has been useful in several contexts \cite{Shioda-Inose,
  Inose:singular-K3, Kumar, Elkies, Shioda:Kummer-sandwich}. In
\cite{Kuwata:MW-rank}, Kuwata defined elliptic K3 surfaces $F^{(1)}$
through $F^{(6)}$ through a base change of the Inose surface, and used
them to produce elliptic K3 surfaces over $\Q$ of every geometric rank
between $1$ and $18$, except for $15$.  (The rank $15$ case was dealt
with several years later by Kloosterman \cite{Kloosterman-rank15}. See
also \cite{Top-de Zeeuw} for an extension, and the note \cite{rank15}
which provides a construction starting from a Kummer surface.) In
particular, $F^{(1)}$ is the Inose surface, and $F^{(2)}$ is the
Kummer surface.

The main purpose of this article is to describe completely explicitly
the Mordell-Weil lattices of these elliptic K3 surfaces $F^{(n)}_{E_1
  \times E_2}$ in the ``generic'' case\footnote{This is a slight abuse
  of notation: in the geometric moduli space $\mathbb{A}_{j_1} \times
  \mathbb{A}_{j_2}$, pairs of isogenous elliptic curves cut out a
  countable union of curves. So we are really describing a ``very
  general'' situation.}, i.e., when $E_1$ and $E_2$ are not
isogenous. As a result, we also recover by specialization a finite
index sublattice of the full Mordell-Weil lattice in the non-generic
case. We will show that the splitting field in the generic case is a
subfield of the compositum $\kappa_n := k(E_1[n], E_2[n])$ of the
$n$-torsion subfields of the two elliptic curves, where $k$ is the
base field, and use natural group actions of $\GL_2(\Z/n\Z)^2$ on the
universal family of $F^{(n)}$ over pairs of elliptic curves with level
$n$ structure to give relatively concise descriptions of explicit
bases for the Mordell-Weil lattices.

The last part of this paper is arithmetic and describes the
Mordell-Weil lattices of $F^{(n)}_{E_1\times E_2}$, for those pairs of
elliptic curves such that the Inose surface $F^{(1)}$ is defined over
$\Q$ and is singular, i.e. has the maximal Picard number $20$. In this
case, $E_1$ and $E_2$ must be non-isomorphic isogenous curves with
complex multiplication. In fact, this situation is connected to a
beautiful theorem of Shioda and Inose \cite{Shioda-Inose}, relating
singular K3 surfaces over $\C$ up to isomorphism with classes of
positive definite even quadratic forms. They deduced this theorem from
the work of Piatetski-Shapiro and Shafarevich \cite{P-SS}, which
connected singular Kummer surfaces to doubly even forms, by the use of
a double cover which is nowadays called a Shioda-Inose structure. The
upshot is that the map $X \to T_X$ which just takes the transcendental
lattice of a K3 surface, establishes a bijective correspondence
between Inose surfaces $F^{(1)}_{E_1, E_2}$ and even positive definite
quadratic forms. Shioda and Inose also determined the zeta functions
of these singular K3 surfaces. In our work, we look at the most
arithmetically interesting of these K3 surfaces: namely those which
can be defined over $\Q$. In addition we impose the condition that the
Inose fibration have the maximum possible rank $18$. This requires
that $E_1$ and $E_2$ be isogenous but non-isomorphic. A few of the
examples arise from $E_1$ and $E_2$ being defined over $\Q$, but most
of them arise from $\Q$-curves \cite{Gross}. Our methods can be used
to determine the full Mordell-Weil group and N\'eron-Severi lattice in
each case; we give several illustrative examples.

We note some prior work toward computation of the Mordell-Weil groups
of the surfaces $F^{(n)}$ studied in this paper. In
\cite{Kuwata:MW-rank}, Kuwata used rational quotients and twists to
decribe a method to compute the Mordell-Weil group of $F^{(3)}$. This
was made somewhat more explicit and extended to the other $F^{(n)}$ by
Kloosterman in \cite{Kloosterman-explicit}, who computed polynomials
(the most complicated one, for $F^{(5)}$, having degree $240$) whose
solution would yield generators for the Mordell-Weil groups in the
generic case, and lead to a finite index subgroup in other
cases. Still, it was not clear what the systematic solution of these
polynomial equations should ``look like''. In our work, we make use of
two key insights to elucidate the structure of the Mordell-Weil groups
of these surfaces. The first is that the splitting field of the
N\'eron-Severi group, or of the Mordell-Weil group, of $F^{(n)}$
associated with curves $E_1$ and $E_2$ should be related to the
$n$-torsion fields of these
two elliptic curves. This is a natural leap of faith from the
situation of the Kummer surface, which is relatively well studied. The
second is that the action of $\SL_n(\Z/n\Z)$ on the moduli space
$X(n)$ of elliptic curves with full level $n$ structure gives rise to
an action of $\SL_n(\Z/n\Z)^2$ on the universal family of $F^{(n)}$
over $X(n) \times X(n)$, and therefore on the family of Mordell-Weil
groups. This action allows us to propagate a single section to
essentially obtain a basis of the Mordell-Weil lattice. In addition to
these two observations, we use the technique of studying associated
rational elliptic surfaces, for which we have better control of the
Mordell-Weil group, to complete the description of the Mordell-Weil
lattices in the generic case.  In the non-generic case,
\cite{Shioda:correspondence}, Shioda related the Mordell-Weil group of
$F^{(1)}$ to isogenies between the two elliptic curves. We make
Shioda's construction completely explicit, even carrying out the
transformation from isogenies to sections in many examples.

\subsection{Outline} 
In Section \ref{sec:defns}, we define the elliptic K3 surfaces $F^{(1)}$
through $F^{(6)}$ that we shall study in this paper, and recall
relevant results from the literature. Section \ref{sec:F1andF2}
describes the explicit connection between the Mordell-Weil group of
$F^{(1)}$ and isogenies between the two elliptic curves. It also
describes the Mordell-Weil group of the Kummer surface $F^{(2)}$ in
the generic case (i.e., when the elliptic curves are not isogenous).
Section \ref{sec:F3} computes the Mordell-Weil group of $F^{(3)}$ in
the generic case by introducing two of the key methods in this paper:
the study of associated rational elliptic surface (for which the
determination of the Mordell-Weil group is easier), and the use of a
large group of symmetries acting on the K3 surface. In section
\ref{sec:F4}, we compute the Mordell-Weil group of $F^{(4)}$ in the
generic case by two methods: first by using the associated rational
elliptic surfaces, and second by analyzing curves of low degree on a
quartic model of this K3 surface. Section \ref{sec:F5} describes the
Mordell-Weil group of $F^{(5)}$ in the generic case. In section
\ref{sec:F6}, we compute the Mordell-Weil group of $F^{(6)}$ again by
two methods: first by analyzing rational elliptic surfaces, and second
by transference from $F^{(3)}$ and its twist, a cubic surface. In
Section \ref{sec:singular}, we recall the correspondence between even
binary quadratic forms and singular K3 surfaces, and describe several
Inose surfaces which can be defined over $\Q$. Finally, in section
\ref{sec:examples}, we apply our methods to give an explicit
description of the Mordell-Weil groups of $F^{(6)}$ obtained from some
of these singular Inose surfaces.

\subsection{Computer files}
Auxiliary files containing computer code to verify the calculations in
this paper, as well as some formulas omitted for lack of space, are
at \url{http://arxiv.org/e-print/1409.2931}. The file at
this URL is a \texttt{tar} archive, which can be extracted to produce
not only the {\LaTeX} file for this paper, but also the computer
code. The text file \texttt{README.txt} briefly describes the various
auxiliary files.

\subsection{Acknowledgements}
Kumar was supported in part by NSF CAREER grant DMS-0952486, and by a
grant from the Solomon Buchsbaum Research Fund. Kuwata was partially
supported by JSPS Grant-in-Aid for Scientific Research (C) Grant
Number 23540028, and by the Chuo University Grant for Special
Research. We thank Tetsuji Shioda for helpful discussions, and Noam
Elkies, Remke Kloosterman and Matthias Sch\"utt for useful comments on
an earlier draft of the paper. The computer algebra systems
\texttt{Magma}, \texttt{sage}, \texttt{gp/PARI}, \texttt{Maxima} and
\texttt{Maple} were used in the calculations for this paper.

\section{Elliptic surfaces associated with the product of elliptic
  curves}

Throughout this paper the base field $k$ is assumed to be a number
field.

\subsection{Kummer surfaces of product type, the Inose fibration and the Inose surface} \label{sec:defns}

Let $E_1$ and $E_2$ be two elliptic curves over $k$.  Later in this
paper, we will be concerned with fields of definition of the
Mordell-Weil groups of various elliptic fibrations. Here, we give a
summary of Kummer surfaces and related constructions, being careful
about the field of definition.

Let $\Km(E_{1}\times E_{2})$ be the Kummer surface associated with the
product abelian surface $E_1\times E_2$, namely the minimal
desingularization of the quotient surface $E_1\times E_2/\{\pm1\}$.  If
$E_1$ and $E_2$ are defined by the equations
\begin{equation}\label{ell}
\begin{aligned}
\null &E_1 : y^2=x^3+ax+b,\\
&E_2 : y^2=x^3+cx+d,
\end{aligned}
\end{equation}
an affine singular model of $\Km(E_{1}\times E_{2})$ may be given as
the hypersurface in $\A^3$ defined by the equation
\begin{equation}\label{Kummer}
x_{2}^3+cx_{2}+d=t_{2}^2(x_{1}^3+ax_{1}+b).
\end{equation}
Then the map $\Km(E_{1}\times E_{2})\to \P^1$ induced by
$(x_{1},x_{2},t_{2})\mapsto t_{2}$ is an elliptic fibration, which is
sometimes called the Kummer pencil.  This elliptic fibration has
obvious geometric sections (i.e., sections defined over $\bar k$), but
they are defined only over the extension $k(E_{1}[2],E_{2}[2])/k$
obtained by adjoining the coordinates of points of order $2$.

Take a parameter $t_{6}$ such that $t_{2}=t_{6}^3$, and consider
\eqref{Kummer} as a family of cubic curves in $\P^2$ over the field
$k(t_{6})$.  Then, this family has a rational point $(1:t_{6}^2:0)$
(cf. Mestre \cite{Mestre:given-j-invariant} and Kuwata-Wang
\cite{Kuwata-Wang}).  Using this point, we convert \eqref{Kummer} to
the Weierstrass form:
\begin{equation}\label{eq:F6}
Y^2 = 
X^3 - 3\,ac\,X  
+ \frac{1}{64}\Bigl(\Delta_{E_{1}}t_{6}^{6}+864\,bd+\frac{\Delta_{E_{2}}}{t_{6}^{6}}\Bigr),
\end{equation}
where $\Delta_{E_{1}}$ and $\Delta_{E_{2}}$ are the discriminants of
$E_{1}$ and $E_{2}$ respectively:
\[
\Delta_{E_{1}}=-16(4a^{3}+27b^{2}), \quad \Delta_{E_{2}} = -16(4c^{3}+27d^{2}).
\]
The change of coordinates between \eqref{Kummer} and \eqref{eq:F6} are
given by
\begin{equation}\label{change-of-var}
\left\{
\begin{aligned}
&X=\frac{-t_{6}^2(2at_{6}^4-c)x_{1} - 3(bt_{6}^6-d) - (at_{6}^4-2c)x_{2}}{t_{6}^{2}(t_{6}^2x_{1}-x_{2})},
\\
&Y=\frac{6(at_{6}^4-c)(bt_{6}^6-d)+6(at_{6}^4-c)(at_{6}^6x_{1}-cx_{2})
-9(bt_{6}^6-d)(t_{6}^4x_{1}^2-x_{2}^2)}
{2t_{6}^{3}(t_{6}^2x_{1}-x_{2})^2}.
\end{aligned}
\right.
\end{equation}
Note that if we choose other models of $E_1$ and $E_2$, we still
obtain an isomorphic equation. Indeed, if we replace the equations of
$E_1$ and $E_2$ by
\begin{align*}
&E_1 : y^2=x^3+(k^4a)x+(k^6b),\\
&E_2 : y^2=x^3+(l^4c)x+(l^6d),
\end{align*}
then replacing $(X,Y,t_{6})$ by $(l^4X,l^6Y,(l/k)t_{6})$, we recover
equation \eqref{eq:F6}.

It is easy to see that equation \eqref{eq:F6} is invariant under the two
automorphisms of $t_{6}$-line:
\begin{equation}\label{D6-action}
\begin{alignedat}{2}
&\sigma:t_{6} \mapsto \zeta_{6}t_{6}, \quad &&\text{where $\zeta_{6}$ is a primitive sixth root of unity,} \\
&\tau: t_{6} \mapsto \delta/t_{6},
&&\text{where $\delta$ is a chosen sixth root of  $\Delta_{2}/\Delta_{1}$.}
\end{alignedat}
\end{equation}
Taking the quotient by the action of $\sigma$, or equivalently,
setting $t_{1}=t_{6}^{6}$, we obtain an elliptic curve over the field
$k(t_{1})$, which we denote by $F^{(1)}_{E_{1},E_{2}}$:
\begin{equation}\label{eq:F1}
F^{(1)}_{E_{1},E_{2}}: Y^2 = 
X^3 - 3\,ac\,X  
+ \frac{1}{64}\Bigl(\Delta_{E_{1}}t_{1}+864\,bd+\frac{\Delta_{E_{2}}}{t_{1}}\Bigr).
\end{equation}

\begin{definition} 
  The Kodaira-N\'eron model of the elliptic curve
  $F^{(1)}_{E_{1},E_{2}}$ over $k(t_{1})$ defined by \eqref{eq:F1} is
  called the Inose surface associated with $E_{1}$ and $E_{2}$, and it
  is denoted by $\Ino(E_{1},E_{2})$.
\end{definition}

\begin{remark}
In \cite{Shioda-Inose}, what we call the Inose surface in this article
was originally constructed as a double cover of $\Km(E_{1}\times
E_{2})$.  Shioda and Inose then showed that the following diagram of
rational maps, called a Shioda-Inose structure, induces an isomorphism
of integral Hodge structures on the transcendental lattices of
$E_{1}\times E_{2}$ and $\Ino(E_{1},E_{2})$.
\[
\xymatrix@H=3.6ex@C=-2em{
E_{1}\times E_{2} \ar@{-->}[dr]_{\pi_{0}}
& & \ar@{-->}[dl]^{\pi_{1}} \Ino(E_{1},E_{2})\\
& \Km(E_{1}\times E_{2}) & }
\]
Since the Kodaira-N\'eron model of $F^{(2)}_{E_{1}\times E_{2}}$ is
isomorphic over~$\kbar$ to $\Km(E_{1}\times E_{2})$ (with $t_2$ being
the elliptic parameter of the Inose fibration
\cite{Inose:singular-K3}), we have another quotient map from
$\Km(E_{1}\times E_{2})$ to $\Ino(E_{1},E_{2})$.  Thus, we have a
``Kummer sandwich'' diagram:
\[
\xymatrix{
\Km(E_{1}\times E_{2}) \ar@{-->}[r]^{\pi_{2}} & \Ino(E_{1},E_{2})
 \ar@{-->}[r]^{\pi_{1}} & \Km(E_{1}\times E_{2})}
\]  
(cf. Shioda \cite{Shioda:Kummer-sandwich}).  However, with our
definition of $\Ino(E_{1},E_{2})$, the quotient map $\pi_{1}$ may not
be defined over the base field $k$ itself, but rather only over
$k(E_{1}[2],E_{2}[2])$ (or an extension of $k$ including some of the
$2$-torsion of $E_1$ and $E_2$).
\end{remark}

\begin{definition}
  For $n=1,\dots,6$, let $t_{n}$ be a parameter satisfying
  $t_{n}^{n}=t_{1}$.  Define the elliptic curve
  $F^{(n)}_{E_{1},E_{2}}$ over $k(t_{n})$ by
  \[
  F^{(n)}_{E_{1},E_{2}}: Y^2 = X^3 - 3\,ac\,X +
  \frac{1}{64}\Bigl(\Delta_{E_{1}}t_{n}^{n} +
  864\,bd+\frac{\Delta_{E_{2}}}{t_{n}^{n}}\Bigr).
  \]
\end{definition}

When $E_{1}$ and $E_{2}$ are understood, we write $F^{(n)}$ for short.
\begin{remark}
The Kodaira-N\'eron model of $F^{(n)}_{E_{1},E_{2}}$ is a K3 surface
for $n=1,\dots,6$, but not for $n\ge 7$.
\end{remark}

By Inose's theorem (\cite[Cor.~1.2]{Inose:quartic}), the Picard number
of the K3 surface $F^{(n)}$ does not depend on~$n$, and equals the
Picard number of $\Km(E_1\times E_2)$. It is therefore at least~$18$.
These surfaces are clearly of geometric and arithmetic interest, being
closely related to abelian surfaces which are the product of two
elliptic curves. We now summarize what is known about the geometric
Picard and \MoW\ groups of these elliptic K3 surfaces.

Define $R(t)$ and $S(t)$ by letting the Inose surface as in equation
\eqref{eq:F1} be $Y^2 = X^3 + R(t)X + S(t)$, and let $h = \rank
\Hom(E_1, E_2)$, so that $0 \leq h \leq 4$. The table below lists the
minimal Weierstrass equations, the configuration of singular fibers,
and the Mordell-Weil rank in the ``generic'' case $j(E_1) \neq j(E_2)$
and $j(E_i) \neq 0$. In the other cases, which will not be relevant to
this paper, we refer the reader to \cite[Th.~4.1]{Kuwata:MW-rank} for
the analogous data.

\begin{table}[h]
\caption*{The surfaces $F^{(n)}$ in the case $j(E_i)$ are nonzero and
  unequal.}
\begin{center}
\begin{tabular}{clcc}
\toprule
$n$ & \hfil Minimal equation & Singular fibers & Rank \\
\midrule 
1 & $Y^2= X^3 +t^4R(t) X + t^5 S(t)$  & $2 \II^*, 4 \I_1$ & $h$   \\
2 & $Y^2= X^3 +t^4R(t) X + t^4 S(t^2)$ & $2 \IV^*, 8 \I_1$ & $4+h$ \\
3 & $Y^2= X^3 +t^4R(t) X + t^3 S(t^3)$ & $2 \I_0^*, 12 \I_1$ & $8+h$ \\
4 & $Y^2= X^3 +t^4R(t) X + t^2 S(t^4)$ & $2 \IV, 16 \I_1$ & $12+h$ \\
5 & $Y^2= X^3 +t^4R(t) X + t S(t^5)$ & $2 \II, 20 \I_1$ & $16+h$ \\
6 & $Y^2= X^3 +t^4R(t) X + S(t^6)$ & $24 \I_1$ & $16+h$ \\
\bottomrule
\end{tabular}
\end{center}
 
\end{table}

The N\'eron-Severi and transcendental lattices were further analyzed
by Shioda \cite{Shioda:2000, Shioda:correspondence,
  Shioda:sphere-packing}, culminating in the following theorems, which
are stated in the geometric situation $k = \C$.  In this case we may
scale $x,y,t$ to work with a simpler equation of $F^{(n)}$, as in
\cite{Inose:quartic, Shioda:2000}:
\[
Y^2=X^3-3\root 3\of {J_{1}J_{2}}\,X
+t^{n}+\frac{1}{t^{n}}-2\sqrt{(1-J_{1})(1-J_{2})},
\]
where $J_{i}=j(E_{i})/1728$.  

\begin{theorem}[Shioda \cite{Shioda:correspondence}] \label{structure-theorem1-shioda}
There is a lattice isomorphism $T(F^{(n)}) \cong T(F^{(1)}) \langle n
\rangle$. In particular, $\det T(F^{(n)}) = \det T(F^{(1)}) \cdot
n^{\lambda}$, where $\lambda = 4-h$. The \MoW\ group $\MW(F^{(n)})$ is
torsion-free, except when $j(E_{1}) = j(E_{2}) = 0$ and $n = 2,4,6$,
or $j(E_{1}) =j(E_{2}) = 1728$ and $n = 3,6$.
\end{theorem}

\begin{remark}
The notation $\<n\>$ means that the pairing of the lattice is
multiplied by $n$.
\end{remark}

\begin{theorem}[Shioda \cite{Shioda:sphere-packing}] \label{structure-theorem2-shioda}
There is a natural isomorphism of lattices
\[
\Hom(E_1, E_2) \cong F^{(1)}(k(t)).
\]
\end{theorem}

In particular, we can compute the \MoW\ rank as follows:

\begin{proposition}\label{prop:Fn-MWrank} \cite{Shioda:2000}
For elliptic curves $E_{1}$ and $E_{2}$, and for $1\le n\le 6$, we
have
\[
\rank F^{(n)}_{E_{1},E_{2}}(\bar k(t_{n}))
= h + \min\bigl(4(n-1),16\bigr) 
- \begin{cases}
0 & \text{if $j(E_{1})\neq j(E_{2})$,}\\
n & \text{if $j(E_{1}) = j(E_{2}) \neq 0, 1728$,}\\
2n & \text{if $j(E_{1}) = j(E_{2}) =0$ or $1728$,}
\end{cases}
\]
where $h=\rank \Hom(E_{1},E_{2})$.  
\end{proposition}
In particular, the largest possible Mordell-Weil rank is $18$, and we
have the following.

\begin{proposition}
Let $E_{1}$ and $E_{2}$ be two elliptic curves over $k$ satisfying the
following two conditions.
\begin{enumerate}
\item $E_{1}$ and $E_{2}$ are isogenous but not isomorphic over $\kbar$.
\item $E_{1}$ and $E_{2}$ have complex multiplication.
\end{enumerate}
Then, the Mordell-Weil groups $F^{(5)}(\kbar(t_{5}))$ and
$F^{(6)}(\kbar(t_{6}))$ have rank~$18$.
\end{proposition}

Shioda further analyzed the surface $F^{(5)}$ for the CM elliptic
curves $y^2 = x^3 - 1$ and $y^2 = x^3 - 15x + 22$, which are
$2$-isogenous to each other, and determined its Mordell-Weil group
\cite{Shioda:2007}. For the same pair of elliptic curves, $F^{(6)}$
was studied in \cite{Chahal-Meijer-Top} and generators for its
Mordell-Weil group were computed.

In this article, we will generalize these results further, to obtain
explicit descriptions of the Mordell-Weil lattices of the surfaces
$F^{(n)}$. Our main results are the following.

\begin{theorem}
 Suppose the two elliptic curves $E_1$ and $E_2$ are not isogenous
 (over $\kbar$). 
\begin{enumerate} 
\item The field of definition of the Mordell-Weil group of $F^{(n)}$
  (i.e. the smallest field over which all the sections are defined) is
  contained in $k(E_1[n], E_2[n])$, the compositum of the $n$-torsion
  fields of $E_1$ and $E_2$.
\item An explicit basis for $\MW(F^{n})$ is described by the
  corresponding results: Proposition \ref{prop:F^2}, Theorem
  \ref{th:F^(3)}, Theorem \ref{th:F4-RES} and Theorem
  \ref{th:F4-quartic}, Theorem \ref{th:F^5}, Theorem \ref{th:F^6-RES}
  and Theorem \ref{th:F^6-cubic}.
\end{enumerate}
\end{theorem}

\begin{theorem}
 In the general case when $E_1$ and $E_2$ are allowed to be isogenous,
 there is a finite index sublattice $\MW(F^{n})$ for which all the
 sections can be defined over the compositum of $k(E_1[n], E_2[n])$
 and the field of definition of $\Hom(E_1, E_2)$.
\end{theorem}

\begin{remark}
It is possible that the field of definition of $\MW(F^{n})$ in the
general case coincides with the above compositum. However, we have not
generated sufficient numerical evidence to formally state this as a
conjecture.
\end{remark}

\subsection{Galois correspondence of sublattices}
The Mordell-Weil lattice of the surface $F^{(6)}$ has a particularly
rich structure, with sublattices induced from the Mordell-Weil
lattices of several quotients which are elliptic rational or K3
surfaces. As we saw in \eqref{D6-action}, a dihedral group $D_{6}$
generated by $\sigma$ and $\tau$ acts on $F^{(6)}$.  We define
\begin{equation}\label{sni}
s_{n,i}=t_{n}+\frac{(\zeta_{6}^{-i}\delta)^{6/n}}{t_{n}}
\quad  n=1,2,3,6, \ i=0,1,\dots,n-1.
\end{equation}
Recall that $t_n^{n}=t_{1} = t_6^{6}$. Then $s_{n,i}$ is invariant
under $\sigma^{n}$ and $\tau\sigma^{i}$.  Write
$S=s_{1,0}=t_{1}+(\Delta_{E_{2}}/\Delta_{E_{1}})t_{1}^{-1}$ for
simplicity.  Then, the extension $k(t_{6})/k(S)$ is a Galois
extension, and its Galois group is $D_{6}=\<\sigma,\tau\>$.  Our basic
idea is to consider the elliptic surface
\[
F_{S}:Y^{2} = X^3 - \frac{1}{3} A\,X  + \frac{1}{64} (\Delta_{E_{1}} S + C),
\]
where $A$ and $C$ are as in \eqref{ABCD}, and view $F^{(6)}(\bar
k(t_{6}))$ as the Mordell-Weil group of $F_{S}$ over the extension
$\bar k(t_{6})/\bar k(S)$.  In other words, we regard $F^{(6)}(\bar
k(t_{6}))=F_{S}(\bar k(t_{6}))$.

Write $M(t_{n})=F_{S}(\bar k(t_{n}))$ and $M(s_{n,i})=F_{S}(\bar
k(s_{n,i}))$.  Between $k(t_{6})$ and $k(S)$ there are fourteen
intermediate fields. Corresponding to these we have a relation among
the Mordell-Weil groups $M(t_{n})$ and $M(s_{n,i})$.
\[
\includegraphics[scale=1.05]{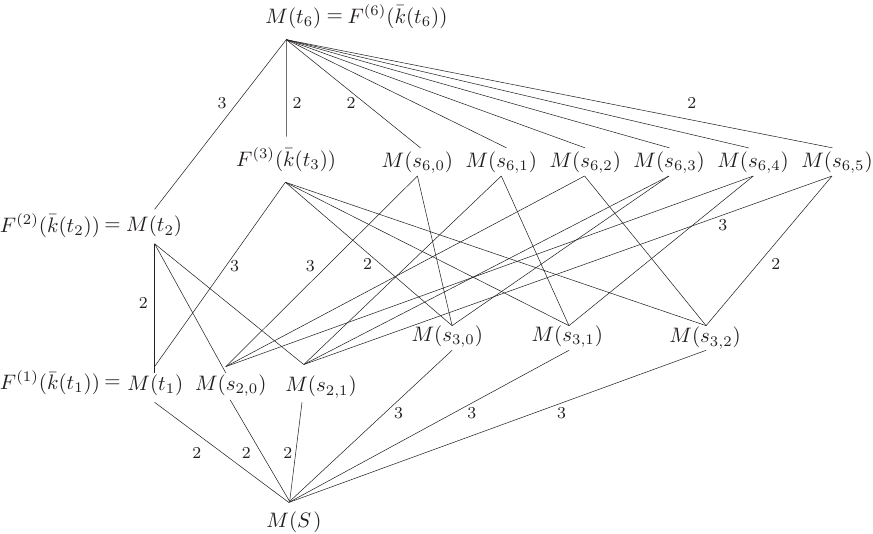}
\]

For later use, let us write a more general formula for $F^{(n)}$.  If
$E_{1}$ and $E_{2}$ are given by
\begin{align*}
&E_1 : y^2=x^3+a_{2}x^{2}+a_{4}x+a_{6},\\
&E_2 : y^2=x^3+a'_{2}x^{2}+a'_{4}x+a'_{6},
\end{align*}
then the equation of $F^{(n)}_{E_{1},E_{2}}$ is given by
\[
F^{(n)}_{E_{1},E_{2}} : Y^2 = X^3 - \frac{1}{3} A\,X 
+\frac{1}{64} \left( B\,t_{n}^{n} + C + \frac{D}{t_{n}^{n}} \right),
\]
where
\begin{equation}\label{ABCD}
\left\{
\begin{aligned}
A &= (a_{2}^2-3a_{4})({a'_{2}}^2-3a'_{4}), \\
B &= 16
(a_{2}^2a_{4}^2-4a_{2}^3a_{6}+18a_{2}a_{4}a_{6}-4a_{4}^3-27a_{6}^2)
=\Delta_{E_{1}}, \\
C &= \frac{32}{27}(2a_{2}^3-9a_{2}a_{4}+27a_{6})(2{a'_{2}}^3-9a'_{2}a'_{4}+27a'_{6}), \\
D &=16({a'_{2}}^2{a'_{4}}^2-4{a'_{2}}^3a'_{6}+18a'_{2}a'_{4}a'_{6}
-4{a'_{4}}^3-27{a'_{6}}^2)
=\Delta_{E_{2}}.
\end{aligned}
\right.
\end{equation}

\section{The Mordell-Weil groups of $F^{(1)}$ and $F^{(2)}$} 
\label{sec:F1andF2}

In this section we summarize the description of the Mordell-Weil
lattices $F^{(1)}(\kbar(t_{1}))$ and $F^{(2)}(\kbar(t_{2}))$ for
completeness.

\subsection{$F^{(1)}$ for isogenous case}
The Mordell-Weil lattice of $F^{(1)}(\kbar (t_{1}))$ for generic
$E_{1}$ and $E_{2}$ is trivial, and we have nothing to do.  If $E_{1}$
and $E_{2}$ are isogenous but not isomorphic over $\kbar$, we have the
following interpretation of the Mordell-Weil lattice by
$\Hom(E_{1},E_{2})$.

\begin{proposition}[Shioda~\cite{Shioda:correspondence}]
Let $E_{1}$ and $E_{2}$ be two elliptic curves not isomorphic to each
other over~$\kbar$. Then, the Mordell-Weil lattice of
$F^{(1)}(\kbar (t_{1}))$ is isomorphic to the lattice
$\Hom(E_{1},E_{2})\<2\>$, where the pairing of $\Hom(E_{1},E_{2})$ is
given by
\[
(\ph,\psi)=\frac{1}{2}\bigl(\deg(\ph+\psi)-\deg\ph-\deg\psi\bigr)
\quad \ph, \psi\in\Hom(E_{1},E_{2}).
\]
\end{proposition}

For a given $\ph\in \Hom(E_{1},E_{2})$, we would like to compute the
section corresponding to $\ph$ explicitly.  To do so, we consider the
inclusion
\[
\Hom(E_{1},E_{2})\<2\>
\isom F^{(1)}_{E_{1},E_{2}}(\bar k(t_{1}))
\hookrightarrow 
\Hom(E_{1},E_{2})\<12\> \subset F^{(6)}_{E_{1},E_{2}}(\bar k(t_{6}))
\]
induced by $t_{1}\mapsto t_{6}^{6}$, and we look for a section in
$F^{(6)}_{E_{1},E_{2}}(\bar k(t_{6}))$. 

Suppose $E_{1}$ and $E_{2}$ are given in the form of \eqref{ell}.  By
replacing $t_{2}$ by $t_{6}^{3}$ in the equation of Kummer surface
\eqref{Kummer}, we regard it as a cubic curve over $k(t_{6})$.  More
precisely, we consider the cubic curve in the projective plane over
$k(t_{6})$ with coordinates $(x_{1}:x_{2}:z)$ defined by
\begin{equation}\label{eq:C_t6}
C_{t_{6}}:x_{2}^{3}+cx_{2}z^{2}+dz^{3}
=t_{6}^{6}\bigl(x_{1}^{3}+ax_{1}z^2+bz^{3}\bigr).
\end{equation}
Suppose $\ph$ is an isogeny of degree~$d$.  Then, $\ph$ can be written
in the form
\[
\ph:(x_{1},y_{1})\mapsto (x_{2},y_{2}) =\bigl(\ph_{x}(x_{1}), \ph_{y}(x_{1})y_{1}\bigr).
\]
Consider the curve of degree $d$ given by $x_{2}=\ph_{x}(x_{1})$.  The
intersection of these two curves
\begin{equation}\label{intersection}
\left\{
\begin{array}{l}
x_{2}^{3} + c x_{2} + d = t_{6}^{6}(x_{1}^{3} + a x_{1} + b), \\
x_{2} = \ph_{x}(x_{1}),
\end{array}
\right.
\end{equation}
gives a divisor of degree $3d$ in $C_{t_{6}}$. 
Since we have 
\(
\ph_{x}(x_{1})^{3} + c \ph_{x}(x_{1}) + d = \ph_{y}(x_{1})^{2}y_{1}^{2}= \ph_{y}(x_{1})^{2}(x_{1}^{3} + a x_{1} + b),
\)
the first equation reduces to
\begin{equation}\label{eq:deg=3d}
\bigl(\ph_{y}(x_{1})-t_{6}^{3}\bigr)\bigl(\ph_{y}(x_{1})+t_{6}^{3}\bigr)(x_{1}^{3} + a x_{1} + b)=0.
\end{equation}

\begin{proposition} \label{prop:sectionsfromisogeny}
Let $\ph:E_{1}\to E_{2}$ be an isogeny of degree~$d$ defined over~$k$.
Let $D_{\ph}^{+}$ \textup{(}resp. $D_{\ph}^{-}$\textup{)} be the
divisor on the cubic curve \eqref{eq:C_t6} defined by the equation
$\ph_{y}(x_{1})=t_{6}^{3}$ \textup{(}resp. $\ph_{y}(x_{1})=
-t_{6}^{3}$\textup{)}.
\begin{enumerate}
\item The divisor $D_{\ph}^{+}$
  \textup{(}resp. $D_{\ph}^{-}$\textup{)} determines a
  $k(t_{6})$-rational point $P^{+}_{\ph}$
  \textup{(}resp. $P^{-}_{\ph}$\textup{)} in $F^{(6)}(k(t_{6}))$.
\item $P^{+}_{\ph}-P^{-}_{\ph}$ is in the image of
  $F^{(1)}(k(t_{1}))\to F^{(6)}(k(t_{6}))$.  The height of its
  pre-image in $F^{(1)}(k(t_{1}))$ is $2d$.
\end{enumerate}
\end{proposition}

\begin{proof}
(i) If $d$ is odd, the denominator of $\ph_{y}(x_{1})$ and $x_{1}^{3}
  + a x_{1} + b$ are relatively prime.  So, the degree of
  $\ph_{y}(x_{1})=\pm t_{6}^{3}$ in $x_{1}$ equals $(3d-3)/2$.  If $d$
  is even, a cancellation occurs between the denominator of
  $\ph_{y}(x_{1})$ and $x_{1}^{3} + a x_{1} + b$ at the $x_{1}$
  coordinate of one of the $2$-torsion points of $E_{1}$.  So, the
  degree of $\ph_{y}(x_{1})=\pm t_{6}^{3}$ equals $(3d-2)/2$.  In any
  case, let $r$ be the degree of $\ph_{y}(x_{1})=\pm t_{6}^{3}$.

Write $D^{+}_{\ph}=Q^{+}_{1}+\cdots+Q^{+}_{r}$.  (Note that the
$Q_i^+$ are defined over an algebraic closure of $\bar k(t_6)$.)
Recall that we chose $(x_{1}:x_{2}:z)=(1:t_{6}^{2}:0)$ as the origin
$O$ of the group law on $C_{t_{6}}$.  We identify
$F_{E_{1},E_{2}}^{(6)}(\overline{k(t_{6})})$ with the divisor class
group $\Pic^{0}_{\overline{k(t_{6})}}(C_{t_{6}})$ by the usual map
which associates a section with its generic fiber minus $O$.  Let
$Q_{\ph}^{+}$ be the point in $C_{t_{6}}$ such that
$D^{+}_{\ph}-rO\sim Q_{\ph}^{+}-O$.  Since $D_{\ph}^{+}$ and $O$ are
defined over $k(t_{6})$, so is $Q_{\ph}^{+}$.  Thus, we have a point
$P_{\ph}^{+}=[Q_{\ph}^{+}-O]\in
\Pic^{0}_{k(t_{6})}(C_{t_{6}})=F_{E_{1},E_{2}}^{(6)}(k(t_{6}))$.
Similarly, we obtain $P_{\ph}^{-}\in F_{E_{1},E_{2}}^{(6)}(k(t_{6}))$
from $D^{-}_{\ph}$.

(ii) By definition, the surface $F_{E_{1},E_{2}}^{(1)}$ is obtained as
the quotient of $F_{E_{1},E_{2}}^{(6)}$ by the action
$(X,Y,t_{6})\mapsto (X,Y,\zeta_{6} t_{6})$ on \eqref{eq:F6}, where
$\zeta_{6}$ is a sixth root of unity.  However, the action
$\sigma:\bigl((x_{1}:x_{2}:z),t_{6}\bigr)\mapsto
\bigl((x_{1}:x_{2}:z),\zeta_{6} t_{6}\bigr)$ on $C_{t_{6}}$ does not
correspond to $(X,Y,t_{6})\mapsto (X,Y,\zeta_{6} t_{6})$ since the
quotient of the former gives a rational surface.  As a matter of fact,
calculations show that the action $\sigma$ corresponds to the action
$(X,Y,t_{6})\mapsto (X,-Y,\zeta_{6} t_{6})$.

By construction, the involution $\sigma^{3}$ interchanges between the
points $Q_{\ph}^{+}$ and $Q_{\ph}^{-}$ in $C_{t_{6}}$. Thus, the
corresponding involution $(X,Y,t_{6})\mapsto (X,-Y,-t_{6})$ on
$F_{E_{1},E_{2}}^{(6)}$ sends $P_{\varphi}^{+}$ to $P_{\varphi}^{-}$.
This implies that $P_{\varphi}^{+}-P_{\varphi}^{-}$ is invariant under
the involution $(X,Y,t_{6})\mapsto (X,Y,-t_{6})$.  Moreover, since
$Q_{\ph}^{\pm}$ are both invariant under the automorphism $\sigma^{2}$
by construction, $P_{\varphi}^{\pm} $ are also invariant under the
corresponding action. We thus conclude that
$P_{\varphi}^{+}-P_{\varphi}^{-}$ is invariant under
$(X,Y,t_{6})\mapsto (X,Y,\zeta_{6} t_{6})$, and it belongs to the
image of $F^{(1)}_{E_{1}, E_{2}}(k(t_{6}))$ under the map
$t_{1}\mapsto t_{6}^{6}$.

It remains to calculate the height of this point, but the calculation
is essentially the same as in
\cite[Proposition~3.1]{Shioda:correspondence}.
\end{proof}

To compute $P_{\ph}^{\pm}$ explicitly, we need to find a curve in the
plane passing through the points in the divisor $D^{\pm}_{\ph}$, and
this is in principle just an exercise in linear algebra.  We shall
illustrate it using a concrete example (see Example~\ref{eg:disc-28}).

\subsection{$F^{(2)}$ for generic case}
Suppose two elliptic cuves are given by 
\begin{equation}\label{legendre}
\begin{aligned}
E_{\lambda}: y^{2} = x(x - 1)(x - \lambda), \\
E_{\mu}: y^{2} = x(x - 1)(x - \mu)
\end{aligned}
\end{equation}
for $\lambda, \mu \in \kbar$. Then, $F^{(2)}_{E_{\lambda},E_{\mu}}$ is
given by the Weierstrass equation
\begin{equation}\label{eq:F^2}
\begin{split}
F^{(2)}_{E_{\lambda},E_{\mu}} : Y^{2} & = X^{3} 
-\frac{1}{3} (\mu^2-\mu+1)(\lambda^2-\lambda+1)X \\
& \quad +\frac{1}{4}\lambda^2(\lambda-1)^2t^2  +\frac{\mu^2(\mu-1)^2}{4t^{2}} \\
& \quad +\frac{1}{54}(2\mu-1)(\mu-2)(\mu+1)(2\lambda-1)(\lambda-2)(\lambda+1).
\end{split}
\end{equation}
Here, we wrote $t_{2}=t$ for simplicity.  Note that in this case the
equation of Kummer surface $x_{2}(x_{2} - 1)(x_{2} - \mu)=t_{2}^2
x_{1}(x_{1} - 1)(x_{1} - \lambda)$ can be converted over $k(\lambda,
\mu)$ to the Weierstrass form $F^{(2)}$ using the point
$(x_{1},x_{2})=(0,0)$.  Then we can obtain sections from $2$-torsion
points of $E_{\lambda}$ and $E_{\mu}$.  In the following proposition,
$P_{1},\dots,P_{4}$ are obtained from
$(x_{1},x_{2})=(1,1),(\lambda,\mu),(1,\mu),(\lambda,1)$, respectively.

\begin{proposition}\label{prop:F^2}
Suppose $E_{\lambda}$ and $E_{\mu}$ are not isogenous over $\bar k$.
Then the following sections form a basis of the Mordell-Weil group
$F^{(2)}_{E_{\lambda},E_{\mu}}(\bar k(t))$
\begin{alignat*}{3}
P_{1} &= \bigg( \frac{1}{3}(-2\lambda\mu+\lambda+\mu+1), &\phantom{=} \frac{1}{2}\lambda(\lambda-1)t+\frac{\mu(\mu-1)}{2t} \bigg)  \\
P_{2} &= \bigg( \frac{1}{3}(\lambda\mu+\lambda+\mu-2), &\phantom{=} \frac{1}{2}\lambda(\lambda-1)t+\frac{\mu(\mu-1)}{2t} \bigg) \\
P_{3} &= \bigg( \frac{1}{3}(\lambda\mu-2\lambda+\mu+1), &\phantom{=} \frac{1}{2}\lambda(\lambda-1)t-\frac{\mu(\mu-1)}{2t} \bigg) \\
P_{4} &= \bigg( \frac{1}{3}(\lambda\mu+\lambda-2\mu+1), &\phantom{=} \frac{1}{2}\lambda(\lambda-1)t-\frac{\mu(\mu-1)}{2t} \bigg)
\end{alignat*}

Moreover, the height pairing of these sections is given by
\[
\frac{2}{3}\left(\begin{array}{*4r}
2 & -1 & 0 & 0 \\
-1 & 2 & 0 & 0 \\
0 & 0 & 2 & -1 \\
0 & 0 & -1 & 2
\end{array}\right).
\]
As a lattice $F^{(2)}_{E_{\lambda},E_{\mu}}(\bar k(t))\isom
A_{2}^{*}\<2\>\oplus A_{2}^{*}\<2\>$.
\end{proposition}

\begin{proof}
This follows from standard and straightforward calculations.
\end{proof}

\begin{corollary}
Let $E_{1}$ and $E_{2}$ be elliptic curves over $k$.  Suppose $E_{1}$
and $E_{2}$ are not isogenous.  Then, the Mordell-Weil lattice
$F^{(2)}_{E_{1},E_{2}}(\bar k(t_{2}))$ is defined over
$k(E_{1}[2],E_{2}[2])$, the field over which all the $2$-torsion
points of $E_{1}$ and $E_{2}$ are defined.
\end{corollary}

If $E_{1}$ and $E_{2}$ are isogenous but not isomorphic, then
$\Hom(E_{1},E_{2})\<4\>\oplus A_{2}^{*}\<2\>\oplus A_{2}^{*}\<2\>$ is
a sublattice of the Mordell-Weil lattice of index $2^h$, where $h=\rank \Hom(E_{1},E_{2})$ (see
\cite[Theorem~1.2]{Shioda:correspondence}).

\section{The Mordell-Weil group of $F^{(3)}$} \label{sec:F3}

In the case of $F^{(2)}$, it is evident that if $E_{1}$ and $E_{2}$
are given in the Legendre normal form, the sections are defined over
the base field.  This is because all the $2$-torsion points are
defined over the base field.  In the case of $F^{(3)}$, it is not so
evident that $3$-torsion points of $E_{1}$ and $E_{2}$ have something
to do with the field of definition of the sections.  However, it turns
out that this is the case, as we will show in this section.

\subsection{Elliptic modular surface associated with $\Gamma(3)$}
Let us begin with the Hesse cubic $x^{3}+y^{3}+z^{3}=3\mu xyz$.  Using
the point $(-1:1:0)$, we convert it to the Weierstrass form
(see \cite{Rubin-Silverberg}):
\begin{equation}\label{X(3)}
y^{2} = x^3-27 \mu(\mu^{3}+8) x+ 54(\mu^6-20\mu^3-8).
\end{equation} 
This elliptic curve has nine $3$-torsion points that are defined over
$k(\omega)$, where $\omega$ is a primitive cube root of unity.

The group of $3$-torsion points is generated by
\[
T_{1}=\bigl(3(\mu+2)^2, 36(\mu^2+\mu+1)\bigr) \quad \textrm{ and } \quad
T_{2}=\bigl(-9\mu^2, 12(2\omega+1)(\mu^2+\mu+1)\bigr).
\]
Let 
\[
\theta = \begin{pmatrix} 1 & 1 \\ 0 & 1 \end{pmatrix} \textrm{ and } \rho = \begin{pmatrix} 0 & 1 \\ -1 & 0 \end{pmatrix}.
\]
be generators for $G = \SL_2(\F_3)$. Consider the following
representation $\pi$ of $G$ on the projective line $\Proj^1_\mu$ by
fractional linear transformations, which factors through
$\PSL_2(\F_3) \cong A_4$.  
\[
\pi(\theta):\mu \mapsto \omega \mu ,\quad \pi(\rho):\mu \mapsto \frac{\mu+2}{\mu-1}.
\]
The $j$-invariant of \eqref{X(3)} is given by
\[
j=\frac{27\mu^3(\mu+8)^3}{(\mu^{3}-1)^3},
\]
and is invariant under the action of $\pi(\SL_2(\F_{3}))$.

Now we regard \eqref{X(3)} as an elliptic surface.  In other words,
consider the elliptic modular surface $\mathcal{E}_{\Gamma(3)}\to
X(3)\isom\P^{1}_{\mu}$ whose generic fiber at $\mu$ is given by
\eqref{X(3)}.

\begin{lemma}\label{lemma:G-action-on-E(3)}
The action of $G=\SL_2(\F_3)$ on the base $\Proj^1_\mu$ extends to a
compatible faithful action $\widetilde{\pi}$ on the surface
$\mathcal{E}_{\Gamma(3)}$.  It is given by the formulas
\begin{equation}\label{sl(2,3)-action}
\begin{gathered}
\widetilde{\pi}(\theta):(x,y,\mu) 
\mapsto (\omega^{2} x, y, \omega \mu), \qquad 
\widetilde{\pi}(\rho):(x,y,\mu) 
\mapsto \left(\frac{(2\omega+1)^{2}}{(\mu-1)^2}x,
\frac{(2\omega+1)^{3}}{(\mu-1)^3}y ,\frac{\mu+2}{\mu-1} \right).
\end{gathered} 
\end{equation}
The action $\widetilde{\pi}$ in turn induces an action $\Pi$ of $G$ on
the group of sections $\mathcal{E}_{\Gamma(3)}(k(\mu))$ of the
elliptic surface as follows.  For a section $s:\P^{1}_{\mu}\to
\mathcal{E}_{\Gamma(3)}$ given by rational functions $\mu\mapsto
P(\mu) = (x(\mu),y(\mu))$, $\Pi(\gamma)(s)$ for $\gamma \in G$ is
defined to be $\mu \mapsto P'(\mu)$ where
\[
P'(\mu) = \widetilde{\pi}(\gamma) \bigl(  P( \pi(\gamma^{-1})(\mu) \bigr).
\]
The action of $G$ by $\Pi$ on the subgroup of $3$-torsion sections
$\mathcal{E}_{\Gamma(3)}(k(\mu))[3]$ is equivalent to the usual linear
action of $\SL_2(\F_3)$ on $\F_3^2$ (and identifies $G$ as the
automorphism group of the $3$-torsion subgroup equipped with the Weil
pairing).  More precisely, we have
\[
\Pi(\theta):\left\{\begin{array}{ccl}
T_{1} & \mapsto & T_{1} - T_{2}\\
T_{2} & \mapsto & T_{2}
\end{array}\right.,
\qquad
\Pi(\rho):\left\{\begin{array}{ccr}
T_{1} & \mapsto & T_{2}\\
T_{2} & \mapsto & -T_{1}
\end{array}\right..
\]
\end{lemma}

\begin{remark}
We have $\tilde{\pi}(\rho)^{2}=[-1]$, the multiplication by $(-1)$
map, and thus the action of $G=\SL_2(\F_3)$ does not factor through
the quotient $\PSL_2(\F_3)$.
\end{remark}

\subsection{$F^{(3)}$ for universal families}
Now, take two copies of \eqref{X(3)},
\begin{align*}
E_{u} &: y_{1}^{2} = x_{1}^3-27 u(u^{3}+8) x_{1}+ 54(u^6-20u^3-8), \\
E_{v} &: y_{2}^{2} = x_{2}^3-27 v(v^{3}+8) x_{2}+ 54(v^6-20v^3-8),
\end{align*}
and construct $\Ino(E_{u},E_{v})$, and in turn, $F^{(3)}_{E_{u},E_{v}}$:
\begin{multline}\label{F^(3)}
F^{(3)}_{E_{u},E_{v}}: Y^{2} = X^{3} -27uv(u^3+8)(v^3+8)X  
\\
+ 1728 (u^3-1)^3 t_{3}^{3}
+ 54(u^6-20u^3-8)(v^6-20v^3-8) 
+ \frac{1728(v^3-1)^3}{t_{3}^{3}}.
\end{multline}
Here, since $\Delta_{E_{1}}=2^{12}\cdot 3^{9}(\mu^{3}-1)^3$, and so
forth, we scaled $X$ and $Y$ differently from \eqref{ABCD}; the
difference is a Weierstrass transformation over $\Q$.

Let $\mathcal{E}_{u}\to \P^{1}_{u}$ and $\mathcal{E}_{v}\to
\P^{1}_{v}$ be the elliptic modular surfaces associated with $E_{u}$
and $E_{v}$ respectively.  Also let $G_u = \< \theta_u, \rho_u \>$ and
$G_v = \< \theta_v, \rho_v \>$ be groups of automorphisms of
$\mathcal{E}_{u}$ and $\mathcal{E}_{v}$ described above, respectively.
We consider \eqref{F^(3)} as the family of elliptic surfaces
$\mathcal{F}^{(3)}_{E_{u},E_{v}}\to \P^{1}_{u}\times \P^{1}_{v}$
parametrized by $u$ and $v$. The total space is a fourfold.

\begin{proposition} \label{action:fourfold}
  The actions of $G_{u}$ and $G_{v}$ induce the action
  $\widetilde{\Pi}$ on the fourfold $\mathcal{F}^{(3)}_{E_{u},E_{v}}$
  given by the following formulas.
\begin{gather*}
\widetilde{\Pi}(\theta_{u}):(X,Y,t_{3},u,v)
\mapsto (\omega^{2}X,Y,\omega t_{3},\omega u,v)
\\
\widetilde{\Pi}(\theta_{v}):(X,Y,t_{3},u,v) 
\mapsto (\omega^{2}X,Y,\omega^2 t_{3},u,\omega v),
\\
\widetilde{\Pi}(\rho_{u}):(X,Y,t_{3},u,v)
\mapsto \left(\frac{(2\omega+1)^2}{(u-1)^2}X,
\frac{(2\omega+1)^3}{(u-1)^3}Y,\frac{(u-1)^2}{(2\omega+1)^2}t_{3},
\frac{u+2}{u-1},v\right),
\\
\widetilde{\Pi}(\rho_{v}):(X,Y,t_{3},u,v)
\mapsto \left(\frac{(2\omega+1)^2}{(v-1)^2}X,
\frac{(2\omega+1)^3}{(v-1)^3}Y,
\frac{(2\omega+1)^2}{(v-1)^2}t_{3},u,\frac{v+2}{v-1}\right).
\end{gather*}
\end{proposition}

\begin{proof}
  Recall that $F^{(6)}$ is, by definition, birationally equivalent to
  $x_{2}^{3}+cx_{2}+d=t_{6}^{6}(x_{1}^{3}+ax_{1}+b)$, where $a,b,c,d$
  are appropriate functions of $u,v$.  (Note also that $F^{(3)}$ is
  \emph{not} birationally equivalent to
  $x_{2}^3+cx_{2}+d=t_{3}^{3}(x_{1}^{3}+ax_{1}+b)$, the latter being a
  rational surface.)  The action on $\mathcal{E}_{u}$ or
  $\mathcal{E}_{v}$ induces the action on $(X,Y,t_{6})$ through the
  change of variables \eqref{change-of-var}. The change of variables
  also depends on the choice of the rational point $(1:t_{6}^{2}:0)$.

  Since $\theta_{u}$ acts as $(x_{1},y_{1},u)\mapsto
  (\omega^{2}x_{1},y_{1},\omega u)$, it stabilizes the equation
  $x_{2}^3+cx_{2}+d=t_{6}^{6}(x_{1}^{3}+ax_{1}+b)$, but it moves the
  rational point $(1:t_{6}^{2}:0)$ to $(\omega^{2}:t_{6}^{2}:0) =
  \big(1: (\omega^2 t_6)^2 : 0\big)$. Therefore, $t_6 \mapsto \omega^2
  t_6$ under this transformation.  Using \eqref{change-of-var}, we see
  that $X$ is mapped to $\omega^{2}X$ and $Y$ remains invariant. Thus,
  $\theta_u$ acts as $(X,Y,t_{6},u,v)\mapsto (\omega^{2}X,Y,\omega^{2}
  t_{6},\omega u,v)$ on $\mathcal{F}^{(6)}$.  This action descends to
  the above expression for $\widetilde{\Pi}(\theta_{u})$ on
  $\mathcal{F}^{(3)}_{E_{u},E_{v}}$.

  As for $\rho_{u}$, the action 
\[(x_{1},x_{2},t_{6},u,v)\mapsto
  \left(\frac{(2\omega+1)^{2}}{(u-1)^2} x_{1},x_{2},\frac{(u-1)}{(2\omega+1)}
  t_{6},\frac{(u+2)}{(u-1)}v\right)
\]
 stabilizes the equation, and by following
  the change of coordinates \eqref{change-of-var}, we obtain the
  desired formula. Similarly, we obtain the expressions for
  $\theta_{v}$ and $\rho_{v}$.
\end{proof}
For clarity of notation, we will henceforth use just $\gamma$ instead
of $\pi(\gamma), \widetilde{\pi}(\gamma), \Pi(\gamma),
\widetilde{\Pi}(\gamma)$, with the action being understood.

\begin{corollary}
  The group $G_{u}\times G_{v}$ acts on the group of sections
  $F^{(3)}_{E_{u},E_{v}}\bigl(\bar k(u,v)(t_{3})\bigr)$.  More
  precisely, if $\gamma\in G_{u}\times G_{v}$ and $s\in
  F^{(3)}_{E_{u},E_{v}}\big(\bar k(u,v)(t_{3})\big)$ is given by
  $X(t_{3},u,v)$ and $Y(t_{3},u,v)$, then $\gamma \cdot s$ is
  given by
\[
\gamma X\big(\gamma^{-1}t_{3},\gamma^{-1}u,\gamma^{-1}v \big) \textrm{ and }
\gamma Y\big(\gamma^{-1}t_{3},\gamma^{-1}u,\gamma^{-1}v \big).
\]
\end{corollary}

From the action of $D_{6}$ on $F^{(6)}$ (see \eqref{D6-action}), we
see that $\<\sigma^{2},\tau\>\simeq D_{3}\simeq S_{3}$ acts on
$F^{(3)}$.  We summarize the action on a section
$(X(t_{3},u,v),Y(t_{3},u,v))$ as follows.
\begin{align*}
& \sigma^{2}\big(X(t_{3},u,v),Y(t_{3},u,v)\big)=\big(X(\omega^2 t_{3},u,v),Y(\omega^2 t_{3},u,v)\big)\\
& \tau\big(X(t_{3},u,v),Y(t_{3},u,v)\big)=\big(X(\delta^2/t_{3},u,v),Y(\delta^2/t_{3},u,v)\big)
\\
& \theta_{u}\big(X(t_{3},u,v),Y(t_{3},u,v)\big)
=\big(\omega^{2}X(\omega^2 t_{3},\omega^{2} u,v),Y(\omega^2 t_{3},\omega^{2} u,v)\big)
\\
& \theta_{v}\big(X(t_{3},u,v),Y(t_{3},u,v)\big)
=\big(\omega^{2}X(\omega t_{3},u,\omega^2 v),Y(\omega t_{3},u,\omega^2 v)\big)
\\
& \rho_{u}\big(X(t_{3},u,v),Y(t_{3},u,v)\big)
\\
&\qquad =\left(
\frac{(u-1)^{2}}{(2\omega+1)^{2}}
X\Bigl(\frac{(u-1)^{2}}{(2\omega+1)^{2}}t_{3},\frac{u+2}{u-1},v\Bigr),
-\frac{(u-1)^{3}}{(2\omega+1)^{3}}
Y\Bigl(\frac{(u-1)^{2}}{(2\omega+1)^{2}}t_{3},\frac{u+2}{u-1},v\Bigr)
\right)
\\
& \rho_{v}\big(X(t_{3},u,v),Y(t_{3},u,v)\big)
\\
&\qquad =\left(
\frac{(v-1)^{2}}{(2\omega+1)^{2}}
X\Bigl(\frac{(2\omega+1)^{2}}{(v-1)^{2}}t_{3},u,\frac{v+2}{v-1}\Bigr),
-\frac{(v-1)^{3}}{(2\omega+1)^{3}}
Y\Bigl(\frac{(2\omega+1)^{2}}{(v-1)^{2}}t_{3},u,\frac{v+2}{v-1}\Bigr)
\right) 
\end{align*}
We may expect that the orbit of a section by these automorphisms
generates the whole Mordell-Weil group, and it turns out that it is
the case.  To show this we must find one section, and we will do this
in the following subsections.

\subsection{Computation of $M(s_{3,0})$}
As we modified the equation of $F^{(3)}$ from the one in \S2, we also
change the definition of $s_{3,i}$ to
\[
s_{3,i}=4\Bigl((u^3-1)t_{3}+\frac{v^3-1}{\omega^{i}t_{3}}\Bigr).
\]
In what follows we denote $s_{3,0}$ simply by $s$.  Then,
\eqref{F^(3)} becomes
\begin{multline}\label{E_s}
R_s^{(3)}:Y^{2} = X^3 - 27 u v (u^3+8) (v^3+8) X \\
+ 27 \bigl(s^3 -48 (u^3-1) (v^3-1) s + 2 (u^6-20 u^3-8) (v^6-20 v^3-8)\bigr).
\end{multline}
$R_s^{(3)}$ is a rational elliptic surface over $k(u,v)$ that has a
singular fiber of type I${}^{*}_{0}$ at $s=\infty$.  As a consequence,
its Mordell-Weil lattice is of type $D_{4}^{*} $ (type
n${}^{\textup{o}}$ 9 in the Oguiso-Shioda \cite{Oguiso-Shioda}
classification).  The Mordell-Weil lattice of a rational elliptic
surfrace is generated by the section of the form
$(c_{2}s^{2}+c_{1}s+c_{0},d_{3}s^{3}+d^{2}s^{2}+d_{1}s+d_{0})$.  An
easy calculation shows that the fact that our elliptic surface has a
singular fiber of type I${}^{*}_{0}$ at $s=\infty$ implies
$c_{2}=d_{3}=d_{2}=0$.  Some further calculations show that we have a
point $P_{0}$ given by
\[
P_{0}=\Bigl(-3s-9(uv+2)^{2},
(2\omega+1)^{3}\bigl(3(uv+2)s+4(u^3-1)(v^3-1)+36(u^2v^2+uv+1)\bigr)\Bigr).
\]
The actions of $\theta_{u}$ and $\theta_{v}$ do not leave
$M(s_{3,0})=R_s^{(3)}(\kbar(s))$ invariant, but
$\bar\theta_{u}:=\sigma^{4}\circ\theta_{u}$ and
$\bar\theta_{v}:=\sigma^{2}\circ\theta_{v}$ do. In fact, it is easily
verified that the action of $\bar\theta_u$ and $\bar\theta_v$ on
$F^{(3)}$ fix $s$. On the other hand, $\rho_{u}$ and $\rho_{v}$ leave
$M(s_{3,0})$ invariant; they take $s$ to $(2\omega+1)^2/(u-1)^2 s$ and
$(2\omega+1)^2/(v-1)^2$ respectively.  The actions of
$\bar\theta_{u}$, $\bar\theta_{v}$, $\rho_{u}$ and $\rho_{v}$ on
$M(s_{3,0})$ are given as follows.
\begin{align*}
\bar\theta_{u}\big(X(s,u,v),Y(s,u,v)\big)
&=\big(\omega^{2}X(s,\omega^{2} u,v),
Y(s,\omega^{2} u,v)\big)
\\
\bar\theta_{v}\big(X(s,u,v),Y(s,u,v)\big)
&=\big(\omega^{2}X(s,u,\omega^2 v),Y(s,u,\omega^2 v)\big)
\\
\rho_{u}\big(X(s,u,v),Y(s,u,v)\big)
&=\left(
\frac{(u-1)^{2}}{(2\omega+1)^{2}}
X\Bigl(\frac{(2\omega+1)^{2}}{(u-1)^{2}}s,\frac{u+2}{u-1},v\Bigr), \right. \\
& \qquad \left. -\frac{(u-1)^{3}}{(2\omega+1)^{3}}
Y\Bigl(\frac{(2\omega+1)^{2}}{(u-1)^{2}}s,\frac{u+2}{u-1},v\Bigr)
\right)
\\
\rho_{v}\big(X(s,u,v),Y(s,u,v)\big)
&=\left(
\frac{(v-1)^{2}}{(2\omega+1)^{2}}
X\Bigl(\frac{(2\omega+1)^{2}}{(v-1)^{2}}s,u,\frac{v+2}{v-1}\Bigr), \right. \\
& \qquad \left. -\frac{(v-1)^{3}}{(2\omega+1)^{3}} 
Y\Bigl(\frac{(2\omega+1)^{2}}{(v-1)^{2}}s,u,\frac{v+2}{v-1}\Bigr)
\right) 
\end{align*}

\begin{proposition}\label{prop:M(s30)}
The group of sections $M(s_{3,0})=R_s^{(3)}(\bar k(s))\subset
F^{(3)}_{E_{u},E_{v}}(\bar k(t_{3}))$ is invariant under the
automorphisms $\bar\theta_{u}$, $\bar\theta_v$, $\rho_{u}$, $\rho_v$.
The group $\<\bar\theta_{u},\bar\theta_v,\rho_{u},\rho_v\>$ acts on
the lattice $M(s_{3,0})$, and $M(s_{3,0})$ is generated by
\[
P_{0}, \quad \bar\theta_{u}(P_{0}), \quad  (\bar\theta_{u} \rho_{u} )(P_{0}), \text{ and }\
(\rho_{u}^{-1}\bar\theta_{u}\rho_{u})(P_{0}).
\]
Furthermore, the height pairing matrix with respect to these sections
is given by
\[
\frac{1}{2}\left(\begin{array}{*4r}
2 & -1 & -1 & -1 \\
-1 & 2 & 0 & 0 \\
-1 & 0 & 2 & 0 \\
-1 & 0 & 0 & 2
\end{array}\right).
\]
All the twenty-four sections of minimal height $1$ are obtained from
$P_{0}$ by the automorphism group $\<\bar\theta_{u},\rho_{u}\>$.
\end{proposition}

\begin{proof}
First of all, we verify that $s_{3,0}$ is invariant under the action
of the group $\<\bar\theta_{u},\rho_{u}\>$.  We then check that the
order of the group $\<\bar\theta_{u},\rho_{u}\>$ is~$24$.  Then, by
calculation, we verify that the orbit of $P_{0}$ contains $24$
different sections.  We then choose four sections whose height matrix
coincides with the desired form.
\end{proof}

\begin{remark}
The $X$-coordinates of four sections above are as follows.
\begin{align*}
X(P_{0}) &= -3s-9(uv+2)^2, \\
X(\bar\theta_{u}(P_{0})) &= -3\omega^2 s-9(uv+2\omega)^2 , \\
X\bigl((\bar\theta_{u} \rho_{u} )(P_{0})\bigr)
&=-3\omega^2 s +3(uv + 2u + 2\omega v - 2\omega)^{2} ,\\
X\bigl((\rho_{u}^{-1}\bar\theta_{u}\rho_{u})(P_{0})\bigr)
& =-3\omega^2 s +3(uv+ 2\omega^2 u + 2\omega^2 v -2\omega)^{2} .
\end{align*}
The $Y$-coordinates are a little more complicated and we omit them
here\footnote{Any explicitly omitted expressions may be found in the
  auxiliary computer files.}.
\end{remark}

\subsection{$F^{(3)}(\kbar(t_{3}))$ in the generic case}
If we consider $E_{u}$ and $E_{v}$ as universal curves with
independent variables $u$ and $v$, they are not isogenous curves.  Looking
at the diagram at the end of \S2, $F^{(3)}(\bar k(t_{3}))$ contains
four sublattices, $F^{(1)}(\bar k(t_{1}))$, $M(s_{3,0})$, $M(s_{3,1})$
and $M(s_{3,2})$.  In our case $F^{(1)}(\bar k(t_{1}))$ is trivial.

\begin{lemma}
The lattice $M(s_{3,1})$ is the image of $M(s_{3,0})$ by the
automorphism $\sigma^{2}$ of $F^{(3)}(\bar k(t_{3}))$, and
$M(s_{3,2})$ is the image of $M(s_{3,0})$ by $\sigma^{4}$.  In
particular $M(s_{3,0})$, $M(s_{3,1})$ and $M(s_{3,2})$ are all
isomorphic.
\end{lemma}

\begin{proof}
From the identity
\begin{align*}
s^3 - 3s &= \Bigl(t + \frac{1}{t} \Bigr)^3 - 3\Bigl(t + \frac{1}{t} \Bigr) = t^3 + \frac{1}{t^3} \\
 &= t^3 + \frac{1}{(\omega t)^3} = \Bigl(t + \frac{1}{\omega t} \Bigr)^3 - 3\omega^2\Bigl(t + \frac{1}{\omega t} \Bigr) = s_1^3 - 3\omega^2 s_1,
\end{align*}
we see that the lattice $M(s_{3,1})$ is the Mordell-Weil lattice of
the elliptic curve given by the equation obtained by replacing $s$ by
$\omega^{2}s_{3,1}$ in \eqref{E_s}.  The assertion is now clear from
\[
s_{3,1} = 4\omega\left( (u^3-1)\omega^2 t_{3}+ \frac{(v^3-1)}{\omega^2
  t_{3}} \right) = \omega\sigma^{2}(s_{3,0}).
\] 
Similarly for $M(s_{3,2})$. 
\end{proof}

\begin{theorem}\label{th:F^(3)}
Let $u,v \in \kbar$ be such that $E_u$ and $E_v$ are not
isogenous. The Mordell-Weil group $F^{(3)}_{E_{u},E_{v}}(\bar
k(t_{3}))$ for the elliptic curve $F^{(3)}_{E_{u},E_{v}}$ over
$k(u,v)$ defined by \eqref{F^(3)} is generated by $M(s_{3,0})$ and
$M(s_{3,1})$.  As a lattice, $F^{(3)}_{E_{u},E_{v}}(\bar k(t_{3}))$ is
generated by
\begin{gather*}
(\rho_{u}^{-1}\theta_{u}\rho_{u})(P_{0})
\quad  (\theta_{u}\rho_{u})(P_{0}), \quad \theta_{u}(P_{0}), 
\quad P_{0}, \\
\sigma^{2}(P_{0}), \quad (\sigma^{2}\theta_{u})(P_{0}), 
\quad (\sigma^{2}\theta_{u}\rho_{u})(P_{0}), \text{ and }\
(\sigma^{2}\rho_{u}^{-1}\theta_{u}\rho_{u})(P_{0})
\end{gather*}
with the height pairing matrix
\[
\frac{1}{2}
\left( 
\begin {array}{*8r} 
4&0&0&1&-2&0&0&-2\\ 
0&4&0&1&-2&0&-2&0\\
0&0&4&1&-2&-2&0&0\\
1&1&1&4&-2&1&1&1\\
-2&-2&-2&-2&4&1&1&1\\ 
0&0&-2&1&1&4&0&0\\
0&-2&0&1&1&0&4&0\\
-2&0&0&1&1&0&0&4
\end {array} \right).
\]
\end{theorem}

\begin{proof}
We calculate the height pairing matrix with respect to the basis of
$M(s_{3,0})$ in Proposition~\ref{prop:M(s30)} together with it image
under $\sigma^{2}$, and verify that its determinant equals
$3^{4}/2^{4}$, which coincides with the value given in
\cite{Shioda:sphere-packing}. 
\end{proof}
\begin{remark}
The $x$- and $y$-coordinates for this basis of sections may be easily
computed from the group action, starting with $P_0$. We do not list
them here for brevity, but they may be found in the auxiliary
files. The same holds for any complicated or lengthy expressions
suppressed in the body of the text.
\end{remark}

\begin{corollary}
Let $E_{1}$ and $E_{2}$ be elliptic curves over $k$.  Suppose $E_{1}$
and $E_{2}$ are not isogenous.  Then, the Mordell-Weil lattice
$F^{(3)}_{E_{1},E_{2}}(\bar k(t_{3}))$ is defined over
$k(E_{1}[3],E_{2}[3])$, the field over which all the $3$-torsion
points of $E_{1}$ and $E_{2}$ are defined.
\end{corollary}

\begin{proof}
Over the field $K = k(E_{1}[3],E_{2}[3])$, $E_{1}$ and $E_{2}$ are
isomorphic to $E_{u}$ and $E_{v}$ for suitable choices of $u$ and $v$,
respectively.  By Theorem~\ref{th:F^(3)}, the Mordell-Weil lattice
$F^{(3)}_{E_{u},E_{v}}(\bar k(t_{3}))$ is defined over $K(\omega)$,
but $\omega$ is contained in $K$ by the Weil pairing.  Thus,
$F^{(3)}_{E_{1},E_{2}}(\bar k(t_{3}))$ is defined over $K$.
\end{proof}

\section{Mordell-Weil group of $F^{(4)}$} \label{sec:F4}

\subsection{Elliptic modular surface associated with $\Gamma(4)$}
The elliptic modular surface over the modular curve $X(4)$ is given by
\begin{equation}\label{eq:X(4)}
E_{\sigma}:y^{2}=x\bigl(x+(\sigma^{2}+1)^{2}\bigr)\bigl(x+(\sigma^{2}-1)^{2}\bigr).
\end{equation}
(cf. \cite{Shioda:1973}).  The subgroup of $4$-torsion points are
generated by
\[
(-\sigma^4+1, 2\sigma^4-2) \text{ and }
(-(\sigma^{2}-1)(\sigma+i)^2, -2\sigma(\sigma^{2}-1)(\sigma+i)^2),
\]
where $i=\sqrt{-1}$.
The $j$-invariant of \eqref{eq:X(4)} is given by
\[
j=\frac{16(\sigma^8+14\sigma^4+1)^3}
{\sigma^{4}(\sigma^4-1)^{4}}.
\]
Let 
\[
\theta = \begin{pmatrix} 1 & 1 \\ 0 & 1 \end{pmatrix} \textrm{ and } \rho = \begin{pmatrix} 0 & 1 \\ -1 & 0 \end{pmatrix}.
\]
be generators for $G = \SL_2(\Z/4\Z)$. Consider the following
representation $\pi$ of $G$ on the projective line $\Proj^1_\sigma$ by
fractional linear transformations, which factors through
$\PSL_2(\Z/4\Z)$.
\[
\pi(\theta):\sigma \mapsto i \sigma ,\quad \pi(\rho):\sigma \mapsto \frac{\sigma+1}{\sigma-1}.
\]
The $j$-invariant is invariant under the action of $\pi(\SL_2(\Z/4\Z))$.

We take two copies of the elliptic modular surface \eqref{eq:X(4)}:
\begin{align*}
E_{u}:y_{1}^2 = x_{1}\bigl(x_{1}+(u^{2}+1)^{2}\bigr)
\bigl(x_{1}+(u^{2}-1)^{2}\bigr),
\\
E_{v}:y_{2}^2 = x_{2}\bigl(x_{2}+(v^{2}+1)^{2}\bigr)
\bigl(x_{2}+(v^{2}-1)^{2}\bigr).
\end{align*}
We then obtain $F^{(4)}_{E_{u},E_{v}}$:
\begin{multline*}
F^{(4)}_{E_{u},E_{v}}:Y^2 = X^3 - 27(u^{8}+14u^{4}+1)(v^{8}+14v^{4}+1)X
\\
+54\Bigl(54u^{4}(u^{4}-1)^{4}t^{4} 
+ (u^{12}-33u^8-33u^4+1)(v^{12}-33v^8-33v^4+1)
+\frac{54v^{4}(v^{4}-1)^{4}}{t^{4}}\Bigr)
\end{multline*}

\subsection{Rational elliptic surfaces}

As before we set 
\[
s_j = i^j u (u^4 - 1) t + \frac{ v (v^4 - 1)}{t}.
\]
where $i = \sqrt{-1}$. Then we may write the equation of $F^{(4)}$ as
\[
R_{s_j}^{(4)}: Y^2 = X^3 - AX + 4(s_j^4 - 4 i^j B s_j^2 + 2 i^{2j} B^2) + C,
\]
where 
\[
\left\{
\begin{aligned}
A &= 27(u^8+14u^4+1)(v^8+14v^4+1) \\
B &= uv(u^4-1)(v^4-1) \\
C &= 54(u^{12}-33u^8-33u^4+1)(v^{12}-33v^8-33v^4+1).
\end{aligned}
\right.
\]
This is a rational elliptic surface over $\Proj^1_{s_j}$, with a $\IV$
fiber over $s_j = \infty$. Generically, this is the only reducible
fiber, and the Mordell-Weil lattice is therefore $E_6^*$. To describe
the sections, note that it suffices to do so for $R_{s_0}^{(4)}$, since
$\big(X(u,v,s_0), Y(u,v,s_0)\big)$ is a section of $R_{s_0}^{(4)}$ if and only
if $\big(X(i^j u, v, s_i), Y(i^j u, v, s_i) \big)$ is a section of
$R_{s_j}^{(4)}$.

There are fifty-four sections of $R_{s_0}^{(4)}$ of minimal height, with
twenty-seven sections intersecting each non-identity component of the
$\IV$ fiber.

To solve for the sections intersecting one of these, we set
\[
X = x_0 + x_1 s, \qquad Y = y_0 + y_1 s + 54 s^2,
\]
(where we have written $s = s_0$ for ease of notation) and substitute
into the Weierstrass equation. It is then easy to solve for the
remaining coefficients. We obtain a basis of sections; we list the
$x$-coordinates here for brevity. (The full sections may be found in
the auxiliary files).
\begin{align*}
x(P_1) &= -18(1+i)(uv + i)s -3(5u^4v^4 - u^4 - v^4 + 24iu^3v^3 - 24u^2v^2 - 24iuv + 5), \\
x(P_2) &= 18(1+i)(uv + i)s -3(5u^4v^4 - u^4 - v^4 + 24iu^3v^3 - 24u^2v^2 - 24iuv + 5), \\
x(P_3) &= 18(1-i)(uv-i)s -3(5u^4v^4 - u^4 - v^4 - 24iu^3v^3 - 24u^2v^2 + 24iuv + 5), \\
x(P_4) &= 18(1-i)(u + iv)s + 3(u^4v^4 - 5u^4 - 5v^4 - 24iu^3v + 24u^2v^2 + 24iuv^3 + 1), \\
x(P_5) &= -18(1+i)(u - iv)s + 3(u^4v^4 - 5u^4 - 5v^4 + 24iu^3v + 24u^2v^2 - 24iuv^3 + 1), \\
x(P_6) &= 3(u^4v^4 + u^4 + v^4 + 6u^4v^2 + 6u^2v^4 - 12u^2v^2 + 6u^2 + 6v^2 + 1).
\end{align*}

\subsection{The Mordell-Weil group in the generic case}

Let $P_1, \dots, P_6$ be the sections obtained from $R_{s_0}^{(4)}$,
and $P'_1, \dots, P'_6$ the corresponding sections from
$R_{s_1}^{(4)}$ (obtained by substituting $i u$ for $u$ in
$P_i$). Together, they do not quite span the whole Mordell-Weil
lattice of $F^{(4)}$. We define
\begin{align*}
Q_1 &= -(P_1 - P_3 + P_4 - P_5 + P_6)/2, \\
Q'_1 &= -(P'_1 - P'_3 + P'_4 - P'_5 + P'_6)/2;
\end{align*}
expressions for these sections may be obtained from the computer
files. Note that they are well-defined by the equation above, since
$\MW(F^{(4)})$ is torsion-free by Theorem
\ref{structure-theorem1-shioda}.

\begin{theorem}\label{th:F4-RES}
Let $u, v \in \kbar$ be such that $E_u$ and $E_v$ are not
isogenous. The sections $P_1, \dots, P_5, Q_1, P'_1, \dots, P'_5,
Q'_1$ form a basis of the Mordell-Weil lattice
$F^{(4)}_{E_{u},E_{v}}(\kbar(t_{4}))$.
\end{theorem}

\begin{proof}
The height pairing matrix of these sections has discriminant
$4^4/3^2$, which is the discriminant of the Mordell-Weil lattice in
the generic case.
\end{proof}

\begin{corollary}
Let $E_{1}$ and $E_{2}$ be elliptic curves over $k$.  Suppose $E_{1}$
and $E_{2}$ are not isogenous.  Then, the Mordell-Weil lattice
$F^{(4)}_{E_{1},E_{2}}(\bar k(t_{4}))$ is defined over
$k(E_{1}[4],E_{2}[4])$, the field over which all the $4$-torsion
points of $E_{1}$ and $E_{2}$ are defined.
\end{corollary}

\subsection{$F^{(4)}$ as a quartic surface}

The minimal nonsingular model of $F^{(4)}_{E_{u},E_{v}}$ is
isomorphic to the quartic surface defined by
\begin{equation}\label{eq:quartic}
S_{4}:
ZW\bigl(Z+(v^{2}+1)^{2}W\bigr)\bigl(Z+(v^{2}-1)^{2}W\bigr)
=XY\bigl(X+(u^{2}+1)^{2}Y\bigr)\bigl(X+(u^{2}-1)^{2}Y\bigr).
\end{equation}
$F^{(4)}_{E_{u},E_{v}}$ corresponds to the elliptic fibration on
$S_{4}$ defined by the elliptic parameter $t_{4}=Y/W$.  Generically,
this quartic surface contains sixteen lines
(cf.~\cite{Segre,Inose:quartic,Kuwata:quartic}). They are obtained as
the intersection of one of the four planes
\[
X=0, \quad Y=0, \quad X+(u^{2}+1)^{2}Y=0, \quad X+(u^{2}-1)^{2}Y=0
\]
and one of the four planes
\[
Z=0, \quad W=0, \quad Z+(v^{2}+1)^{2}W=0, \quad Z+(v^{2}-1)^{2}W=0.
\]
We identify $S_{4}$ and $F^{(4)}$ by choosing $X=Z=0$ as the
$0$-section. Let $L_{1},\dots,L_{4}$ be four lines defined by
\begin{gather*}
L_{1}:X+(u^{2}+1)^{2}Y=Z+(v^{2}+1)^{2}W=0,\\
L_{2}:X+(u^{2}-1)^{2}Y=Z+(v^{2}-1)^{2}W=0,\\
L_{3}:X+(u^{2}+1)^{2}Y=Z+(v^{2}-1)^{2}W=0,\\
L_{4}:X+(u^{2}-1)^{2}Y=Z+(v^{2}+1)^{2}W=0.
\end{gather*}
By an abuse of notation, we also denote by $L_{i}$ the corresponding
section of $F^{(4)}_{E_{u},E_{v}}$.

$S_{4}$ may be considered as a family of the intersection of two
quadrics. Namely, consider the map $S_4 \to \Proj^1$ given by 
\[
\nu: (X:Y:Z:W) \mapsto \big(XY + (u^2+1)Y^2 : ZW \big)
\]
Over the point $(p:q) \in \Proj^1$, the fiber of $\nu$ is the
intersection of two quadrics
\[
\left\{
\begin{aligned}
&rZW =Y(X+(u^2+1)^2Y),\\
&(Z+(v^{2}+1)^{2}W)(Z+(v^2-1)^2W)=rX(X+(u^2-1)^2Y)\\
\end{aligned}
\right.
\]
where $r = p/q$. The intersection is a curve of genus~$1$ for each
$r$, except the following eight values:
\[
r=
\frac{\pm2i}{(u+1)^{2}}, \quad
\frac{\pm2i}{(u-1)^{2}}, \quad
\frac{\pm2iv^{2}}{(u+1)^{2}}, \quad
\frac{\pm2iv^{2}}{(u-1)^{2}}.
\]
At each of these values of $r$, the intersection degenerates and
becomes a union of two plane conics.

Let $R_{1},\dots,R_{4}$ be one of the plane conics at each of the
values 
\[r=2i/(u+1)^{2}, \quad -2i/(u+1)^{2}, \quad 2i/(u-1)^{2},\quad 2iv^{2}/(u-1)^{2},
\]
respectively.  Similarly, in the family
\[
\left\{
\begin{aligned}
&rZ(Z+(v^2-1)^2W) =(X+(u^2+1)^2Y)(X+(u^2-1)^2Y),\\
&W(Z+(v^{2}+1)^{2}W)=rXY,\\
\end{aligned}
\right.
\]
let $R_{5},R_{6}$ be one of the conics at the value
$r=\pm2i/(v+1)^{2}$, and let $R_{7}$ be one of the conics at
$r=(u+1)^2/(v+1)^2$ in the family
\[
\left\{
\begin{aligned}
&rZ(Z+(v^2-1)^2W) =X(X+(u^2-1)^2Y),\\
&W(Z+(v^{2}+1)^{2}W)=rY(X+(u^2+1)^2Y).\\
\end{aligned}
\right.
\]
Explicit choices for $R_1, \dots, R_7$ are made in the computer
files. Finally, let $R_{8}$ be the section obtained by letting
$u\mapsto iu$ in $R_{1}$.

\begin{theorem}\label{th:F4-quartic}
The sections $L_{1},\dots,L_{4},R_{1},\dots,R_{8}$ form a basis of the
Mordell-Weil lattice $F^{(4)}_{E_{u},E_{v}}(\kbar(t_{4}))$.
\end{theorem}

\begin{proof}
As in Theorem~\ref{th:F4-RES}, it suffices to verify that the height
pairing matrix of these sections has discriminant $4^4/3^2$.
\end{proof}

\begin{remark}
In \cite{Kuwata:quartic}, it is shown that the lines and conics
contained in $S_{4}$ generate a subgroup of finite index in
$\NS(S_{4})$.
\end{remark}

\begin{proposition}
The basis in Theorem~\ref{th:F4-RES} and that of
Theorem~\ref{th:F4-quartic} are related as follows.
\begin{align*}
L_{1} &= - P_{1} + P_{3} - P_{4} + P_{5} - 2 Q_{1},  \\
L_{2} &= - 2 P_{1} - P_{2} + P_{3} - P_{4} + P_{5} - 2 Q_{1}, \\
L_{3} &= - 2 P'_{1} - P'_{2} + P'_{3} - P'_{4} + P'_{5} - 2 Q'_{1}, \\
L_{4} &= - P'_{1} + P'_{3} - P'_{4} + P'_{5} - 2 Q'_{1},\\
R_{1} & = - P_{1} + P_{3} - Q_{1} - P'_{4} - Q'_{1}, \\
R_{2} & = P_{3} - P_{4} - Q_{1} - P'_{1} - P'_{4} + P'_{5} - Q'_{1}, \\
R_{3} & = - P_{1} + P_{3} - P_{4} + P_{5} - Q_{1} - P'_{1} - P'_{2} + P'_{5} - Q'_{1}, \\
R_{4} & = - P_{4} + P_{5} - Q_{1} - P'_{1} - P'_{2} + P'_{3} - Q'_{1}, \\
R_{5} & = P_{2} - P_{4} - Q_{1} - P'_{4} + P'_{5} - Q'_{1}, \\
R_{6} & = - P_{1} + P_{3} - Q_{1} + P'_{2} - P'_{4} - Q'_{1}, \\
R_{7} & = - P_{4} - Q_{1} - 2 P'_{1} - P'_{2} + P'_{3} - P'_{4} + 2P'_{5} - 2Q'_{1}, \\
R_{8} &= - P_{1} - Q_{1} + P'_{3} - P'_{4} - Q'_{1}.
\end{align*}
\end{proposition}

\section{Mordell-Weil group of $F^{(5)}$} \label{sec:F5}

\subsection{Elliptic modular surface associated with $\Gamma(5)$}
We begin with the elliptic modular surface over $\Gamma(5)$, which is
described in \cite{Rubin-Silverberg} for instance.

\begin{equation}\label{X(5)}
\begin{aligned}
y^2 &= x^3 - 27 (\mu^{20}-228 \mu^{15}+494 \mu^{10}+228 \mu^5+1) x \\
    & \quad   +54 (\mu^{30}+522 \mu^{25}-10005 \mu^{20}-10005\mu^{10}-522 \mu^5+1).
\end{aligned}
\end{equation} 

This elliptic curve has full $5$-torsion defined over
$\Q(\zeta)(\mu)$, where $\zeta$ is a fifth root of unity. Let $\eta =
\zeta + \zeta^{-1}$, which can be taken to be the golden ratio $(1 +
\sqrt{5})/2$. A basis for the $5$-torsion is given by
\begin{align*}
T_1 &= \big( 3(\mu^{10}+12\mu^8-12\mu^7+24\mu^6+30\mu^5+60\mu^4+36\mu^3+24\mu^2+12\mu+1), \\
& \qquad 108\mu(\mu^4-3\mu^3+4\mu^2-2\mu+1)(\mu^4+2\mu^3+4\mu^2+3\mu+1)^2  \big) \\
T_2 &= \Big( -3((12\eta+19)\mu^{10}+66(2\eta-1)\mu^5-12\eta+31)/5, \\
& \qquad 108(3\zeta^3-\zeta^2+2\zeta+1)\big(\mu^{15}+(-5\eta+19)\mu^{10}+(-55\eta+87)\mu^5+5\eta-8\big)/5 \Big).
\end{align*}
Let 
\[
\theta = \begin{pmatrix} 1 & 1 \\ 0 & 1 \end{pmatrix} \textrm{ and } \rho = \begin{pmatrix} 0 & 1 \\ -1 & 0 \end{pmatrix}.
\]
be generators for $G = \SL_2(\F_5)$, and let $\pi$ be the
representation of $G$ on $\Proj^1_\mu$ as follows.
\[
\pi(\theta):\mu \mapsto \zeta \mu ,\quad \pi(\rho):\mu \mapsto -1/\mu.
\]
The $j$-invariant of \eqref{X(3)} is given by
\[
j=-\frac{(\mu^{20}-228\mu^{15}+494\mu^{10}+228\mu^5+1)^3}{\mu^{5} (\mu^{10}+11\mu^5-1)^5}
\]
and is invariant under the action of the icosahedral group
$\pi(\SL_2(\F_{5})) \cong A_5$.

As before, there is a compatible action of $G$ on the universal
elliptic curve, given by
\[
\theta: (x,y,\mu) \to ( x,y,\zeta \mu) , \qquad \rho: (x,y,\mu) \to
(x/\mu^{10}, y/\mu^{15}, -1/\mu),
\]
and therefore an action on the Mordell-Weil group
$\mathcal{E}_{\Gamma(5)}(k(\mu))$ by $(\gamma P) (\mu) = \gamma \big( P
(\gamma^{-1} \mu) \big)$. On the $5$-torsion sections, we have 
\[
\theta:\left\{\begin{array}{ccl}
T_{1} & \mapsto & T_{1} + T_{2}\\
T_{2} & \mapsto & T_{2}
\end{array}\right.,
\qquad
\rho:\left\{\begin{array}{ccr}
T_{1} & \mapsto & 2T_1\\
T_{2} & \mapsto & -2T_2
\end{array}\right..
\]
This action is conjugate to the usual linear action of $\PSL_2(\F_5)$
(since $-1$ is a square modulo $5$, the usual action of $\rho$ is
diagonalizable).

\subsection{$F^{(5)}$ for universal families}
Now, take two copies of \eqref{X(5)}
\begin{align*}
E_{u} : y_{1}^{2} &= x_1^3 - 27 (u^{20}-228 u^{15}+494 u^{10}+228 u^5+1) x_1 \\
      & \qquad \qquad +54 (u^{30}+522 u^{25}-10005 u^{20}-10005u^{10}-522 u^5+1), \\
E_{v} : y_{2}^{2} &= x_2^3 - 27 (v^{20}-228 v^{15}+494 v^{10}+228 v^5+1) x_2 \\
     & \qquad \qquad +54 (v^{30}+522 v^{25}-10005 v^{20}-10005v^{10}-522 v^5+1).
\end{align*}
Computing $\Ino(E_{u},E_{v})$ by \eqref{eq:F1} and base changing, we get
the following Weierstrass equation for $F^{(5)}_{E_{u},E_{v}}$:
\begin{align}\label{F^(5)}
F^{(5)}_{E_{u},E_{v}}: Y^{2} &= X^{3} -27(u^{20}-228u^{15}+494u^{10}+228u^5+1) \\ \notag
& \qquad \qquad \quad \times (v^{20}-228v^{15}+494v^{10}+228v^5+1)X \\\notag
& \qquad -2^63^6u^5 (u^{10}+11u^5-1)^5 t_5^5  -\frac{2^63^6 v^5 (v^{10}+11v^5-1)^5}{t_5^5} \\\notag
& \qquad + 54(u^{30}+522u^{25}-10005u^{20}-10005u^{10}-522u^5+1) \\\notag
 &  \qquad \qquad \times (v^{30}+522v^{25}-10005v^{20}-10005v^{10}-522v^5+1).
\end{align}
(Here we have scaled $X$ and $Y$ by $9$ and $27$ respectively, to make
the coefficients smaller.)

As before, let $\mathcal{E}_{u}\to \P^{1}_{u}$ and $\mathcal{E}_{v}\to
\P^{1}_{v}$ be elliptic modular surfaces associated with $E_{u}$ and
$E_{v}$ respectively.  Also let $G_u = \< \theta_u, \rho_u \>$ and
$G_v = \< \theta_v, \rho_v \>$ be groups of automorphisms of
$\mathcal{E}_{u}$ and $\mathcal{E}_{v}$ described above, respectively.
We consider \eqref{F^(5)} as the family of elliptic surfaces
$\mathcal{F}^{(5)}_{E_{u},E_{v}}\to \P^{1}_{u}\times \P^{1}_{v}$
parametrized by $u$ and $v$. The total space is a fourfold.

Analogously to \ref{action:fourfold}, we have an action of $G_u \times
G_v$ on this variety, with the corresponding expressions being
somewhat simpler:
\begin{align*}
\theta_u & : (X,Y,t_5, u, v) \to (X,Y, t_5, \zeta u, v) \\
\theta_v & : (X,Y,t_5, u,v) \to (X, Y, t_5, u, \zeta v) \\
\rho_u &: (X, Y, t_5, u,v ) \to \left( \frac{X}{u^{10}}, \frac{Y}{u^{15}}, t_5 u^6, -\frac{1}{u}, v \right) \\
\rho_v &: (X, Y, t_5, u,v ) \to \left( \frac{X}{v^{10}}, \frac{Y}{v^{15}}, \frac{t_5}{v^6}, u, -\frac{1}{v} \right).
\end{align*}

There is also the action of the dihedral group $D_{10}$ on
$F^{(5)}$. Let $D_{10} = \< \sigma, \tau \>$, with $\sigma^5 = 1,
\tau^2 = 1$ and $\sigma \tau = \tau \sigma^{-1}$. Then $\sigma$ acts
by twisting $t_5$ by $\zeta$ and $\tau$ by taking it to $\delta/t$,
where 
\[ \delta = v(v^{10} + 11 v^5 - 1)/\big(u (u^{10} + 11 u^5 - 1) \big).
\] 
Consequently, we have the following result, which describes the action
of these groups on the sections.

\begin{proposition}
  The group $D_{10} \times G_{u}\times G_{v}$ acts on the group of sections $F^{(5)}_{E_{u},E_{v}}\bigl(\bar k(u,v)(t_{5})\bigr)$ as
  follows:
  \begin{align*}
    & \sigma\big(X(t_5,u,v),Y(t_5,u,v)\big)=\big(X(\zeta^{-1} t_5,u,v),Y(\zeta^{-1} t_5,u,v)\big)\\
    &  \tau\big(X(t_5,u,v),Y(t_5,u,v)\big)=\big(X(\delta/t_5,u,v),Y(\delta/t_5,u,v)\big)
    \\
    & \theta_{u}\big(X(t_5,u,v),Y(t_5,u,v)\big)  =\big( X(t_5,\zeta^{-1}u,v),Y(t_5,\zeta^{-1} u,v)\big)
    \\
    & \theta_{v}\big(X(t_5,u,v),Y(t_5,u,v)\big)  =\big( X(t_5,u,\zeta^{-1}v),Y(t_5, u,\zeta^{-1}v)\big)
    \\
    & \rho_{u}\big(X(t_5,u,v),Y(t_5,u,v)\big) = \big( u^{10} X(t_5 u^6, -1/u, v) ,-u^{15} Y( t_5 u^6, -1/u, v) \big)
    \\
    & \rho_{v}\big(X(t_5,u,v),Y(t_5,u,v)\big) = \big( v^{10} X(t_5/ v^6, u, -1/v) ,-v^{15} Y( t_5/ v^6, u, -1/v) \big)
    \\
  \end{align*}
\end{proposition}

\subsection{The rational elliptic surface}

Set 
\[
s = u(u^{10} + 11 u^5 - 1) t_5 + \frac{ v (v^{10} + 11 v^5 - 1)}{t_5}.
\]
Then the equation for $F^{(5)}$ transforms to 
\[
Y^{2} = X^{3} -27 A X  -2^6 3^6 (s^5 - 5 B s^3 + 5 B^2 s) + C, 
\]
where
\[
\left\{
\begin{aligned} A &= (u^{20}-228u^{15}+494u^{10}+228u^5+1)(v^{20}-228v^{15}+494v^{10}+228v^5+1), \\
B &= uv (u^{10}+11u^5-1)(v^{10}+11v^5-1),\\
C &=  54(u^{30}+522u^{25}-10005u^{20}-10005u^{10}-522u^5+1) \\
 & \qquad \qquad \quad  \times (v^{30}+522v^{25}-10005v^{20}-10005v^{10}-522v^5+1).
\end{aligned}
\right.
\]

This equation defines a rational elliptic surface $R^{(5)}_{s}$ over
$K = k(u,v)$, fibered over $\Proj^1_s$, with the property that its
base change to $\Proj^1_t$ is $F^{(5)}$. It has a singular fiber of
additive reduction (generically type $\II$) at $s = \infty$, so using
Shioda's specialization technique, we can readily determine an
equation of degree $240$ whose roots give all the specializations of
the minimal vectors of the Mordell-Weil lattice, which is generically
$E_8$. As before, we obtain an action of $\PSL_2(\F_5)$ on the
sections by the following lemma, whose (easy) proof is omitted.

\begin{lemma}
  The automorphisms $\theta_u \theta_v$ and $\rho_u \rho_v$ induce
  automorphisms of the Mordell-Weil group $R^{(s)}(K(s))$, given by
\begin{align*} 
\theta_u \theta_v& : \big( X(s,u,v), Y(s,u,v) \big)  \mapsto \big( X(\zeta^{-1} s, \zeta^{-1} u, \zeta^{-1} v), Y(\zeta^{-1} s, \zeta^{-1} u, \zeta^{-1} v) \big), \\
\rho_u \rho_v & : \big( X(s,u,v), Y(s,u,v) \big) \mapsto \left((uv)^{10} X \Big(\frac{s}{(uv)^6},-\frac{1}{u},-\frac{1}{v} \Big), (uv)^{15} Y \Big(\frac{s}{(uv)^6},-\frac{1}{u},-\frac{1}{v} \Big) \right).
\end{align*}
\end{lemma}

For the rational elliptic surface $R^{(5)}$, it follows from general
structural results (see \cite{Shioda:MWL}) that the Mordell-Weil
lattice is $E_8$, and is spanned by the $240$ smallest vectors. These
correspond to sections of the form
\[
\big(X(s), Y(s)\big) = ( x_2 s^2 + x_1 s + x_0, y_3 s^3 + y_2 s^2 +
y_1 s + y_0).
\]
The specialization of such a section at $s = \infty$ is $z = x_2/y_3
\in \mathbb{G}_a$. Substituting the above expression for $X$ and $Y$
into the Weierstrass equation, we obtain a system of equations, in
which we can eliminate all variables but $z$. The resulting equation
has degree $240$, and it splits over the field $k(u,v)$ into linear
factors. The variables $x_2, \dots, y_0$ are rational functions of $z$
with coefficients polynomials in the Weierstrass coefficients. As a
result, every section is defined over the field $k(u,v)$. In the
result below, we will just write down the specializations of the
relevant sections (the entire expression can be found in the auxiliary
source files).

\begin{proposition} \label{R5basis}
Let $P_0$ and $Q_0$ be sections whose specializations are given by 
\begin{align*}
z(P_0) &= (\zeta+1)\big(uv - (\zeta^2 + 1)u - \zeta^2(\zeta^2 + 1)v - \zeta^4\big)/6, \\
z(Q_0) &= \zeta^3(\zeta-1)(u - \zeta^3v)/6 .
\end{align*}
Then $P_0, \theta(P_0), \theta^{-1}(P_0), \rho(P_0), \theta^3(P_0),
\theta \rho(P_0), Q_0, \rho(Q_0)$ form a basis of the Mordell-Weil
group of $R^{(5)}_{s}$.
\end{proposition}

\begin{proof}
By direct calculation, the intersection matrix of these sections is 
\[
\left(\begin{array}{*8c}
2 & 1 & 1 & 0 & 0 & 1 & 1 & 0\\
1 & 2 & 0 & -1 & 0 & 0 & 1 & 0\\
1 & 0 & 2 & 1 & 1 & 1 & 1 & 0\\
0 & -1 & 1 & 2 & 1 & 1 & 0 & 1\\
0 & 0 & 1 & 1 & 2 & 0 & 1 & 0\\
1 & 0 & 1 & 1 & 0 & 2 & 0 & 1\\
1 & 1 & 1 & 0 & 1 & 0 & 2 & 0\\
0 & 0 & 0 & 1 & 0 & 1 & 0 & 2
\end{array}
\right)
\]
Therefore, they span an even unimodular eight-dimensional lattice,
which must be $E_8$. Hence, these sections are a basis of the entire
Mordell-Weil group.
\end{proof}

\subsection{$F^{(5)}(\kbar(t_5))$ in the generic case}

We now describe the Mordell-Weil group of $F^{(5)}$ in the generic
case, when $E_u$ and $E_v$ are not isogenous.

\begin{theorem} \label{th:F^5}
  Let $P_i, i = 1, \dots, 8$ be the basis of Proposition
  \ref{R5basis}. By abuse of notation, we let $P_i$ denote their base
  change to the elliptic K3 surface $F^{(5)}$. Let $u, v \in \kbar$ be
  such that $E_u$ and $E_v$ are not isogenous. Then the sixteen
  sections $P_i, \sigma(P_i)$ form a basis of the Mordell-Weil group
  $F^{(5)}_{E_{u},E_{v}}(\kbar(t_5))$.
\end{theorem}

\begin{proof}
  By explicit calculation, the intersection matrix (listed in the
  auxiliary files) has determinant $625 = 5^4$. Since this number
  agrees with the discriminant of the full N\'eron-Severi lattice as
  computed by Shioda, the listed sections form a basis.
\end{proof}

\begin{corollary}
Let $E_{1}$ and $E_{2}$ be elliptic curves over $k$.  Suppose $E_{1}$
and $E_{2}$ are not isogenous.  Then, the Mordell-Weil lattice
$F^{(5)}_{E_{1},E_{2}}(\bar k(t_{5}))$ is defined over
$k(E_{1}[5],E_{2}[5])$, the field over which all the $5$-torsion
points of $E_{1}$ and $E_{2}$ are defined.
\end{corollary}

\section{Mordell-Weil group of $F^{(6)}$} \label{sec:F6}

In this section we determine the generators of the Mordell-Weil group
$F^{(6)}(\bar k(t_{6}))$ in two different ways.  Both methods use
rational elliptic surfaces arising as the quotient of $F^{(6)}$ by
some involutions.

\subsection{Elliptic modular surface associated with $\Gamma(6)$}
The modular curve $X(6)$ associated with the congruence subgroup
$\Gamma(6)$ is known to be a curve of genus~$1$ and its affine model
is given by $r^{2}=f^{3}+1$, and the elliptic modular surface
associated with it is given by
\begin{equation}\label{eq:X(6)}
y^2 = x^3-2(f^6+20f^3-8)x^2+f^3(f^3-8)^3x.
\end{equation}
(cf. \cite{RS:mod6}).  The subgroup of $3$-torsion points are generated by
\begin{align*}
& \bigl(f^2(f^2+2f+4)^2,4f^2(f^2-f+1)(f^2+2f+4)^2\bigr) \quad \text{ and } \\
& \bigl(-(f^3-8)^2/3,4(2\omega+1)(f^3+1)(f^3-8)^2/9\bigr),
\end{align*}
and the points of order~$2$ are given by
\[
(0,0), \quad \bigl((r+1)(r-3)^3,0\bigr), \quad \bigl((r+1)(r-3)^3,0\bigr).
\]
The $j$-invariant of \eqref{eq:X(6)} is given by
\[
j=\frac{(f^3+4)^3(f^9+228f^6+48f^3+64)^3}
{f^6(f^{3}+1)^3(f^{3}-8)^6}.
\]
Note that the curve \eqref{eq:X(6)} can be written in terms of $r$:
\begin{equation}\label{eq:X(6)var}
y^{2} = x\bigl(x-(r+1)(r-3)^3\bigr)\bigl(x-(r-1)(r+3)^3\bigr),
\end{equation}
and after scaling $x$ and $y$ it is transformed to
\[
(r-1)(r+3)y^{2} = x(x-1)(x-\lambda), \quad\text{where}\quad
\lambda=\frac{(r+1)(r-3)^3}{(r-1)(r+3)^3}.
\]
If we view \eqref{eq:X(6)} as an elliptic surface over $\P^{1}_{f}$,
it is an elliptic modular surface corresponding to
$\Gamma(3)\cap\Gamma_{1}(2)$, whereas \eqref{eq:X(6)var} as an
elliptic surface over $\P^{1}_{r}$ is an elliptic modular surface
corresponding to $\Gamma(2)\cap\Gamma_{1}(3)$. Note that the map to
$X(2)$ is just $(r,f) \to \lambda$, whereas the map to $X(3)$ is
$(r,f) \to \mu = (f^3+4)/3f^2$.

\subsection{$F^{(6)}$ for universal family}
We take two copies of the modular curve $X(6)$
\[
r^{2} = f^{3} + 1, \quad q^{2} = g^{3} + 1,
\]
and the elliptic modular surface \eqref{eq:X(6)}:
\begin{align*}
E_{r,u}:y_{1}^2 &= x_{1}^3-2(f^6+20f^3-8)x_{1}^2+f^3(f^3-8)^3x_{1}.
\\
E_{q,v}:y_{2}^2 &= x_{2}^3-2(g^6+20g^3-8)x_{2}^2+g^3(g^3-8)^3x_{2}.
\end{align*}
We then obtain $F^{(6)}_{E_{r,f},E_{q,g}}$
\begin{equation}\label{eq:generic-F6}
\begin{aligned}
Y^2 &= X^3 - 27(f^{12}+232f^9+960f^6+256f^3+256) \\
& \qquad \qquad \quad \times (g^{12}+232g^9+960g^6+256g^3+256)X \\
& \quad +54\Bigl(864f^6(f^3+1)^3(f^3-8)^6t^6 +\frac{864g^6(g^3+1)^3(g^3-8)^6}{t^6} \\
& \qquad \qquad +(f^{18}-516f^{15}-12072f^{12}-24640f^9-30720f^6+6144f^3+4096)\\
& \qquad \qquad \quad \times(g^{18}-516g^{15}-12072g^{12}-24640g^9-30720g^6+6144g^3+4096) \Big).
\end{aligned}
\end{equation}

\subsection{Rational elliptic surfaces with parameter $s_{6,i}$}
As before, we have several rational elliptic surfaces arising as
quotients of $F^{(6)}$. Namely, for $0 \leq i \leq 6$, let 
\[
s_{6,i}=s_i = \zeta_6 rf(f^3-8) t + \frac{qg(g^3-8)}{t} ,
\]
where $\zeta_6 = -\omega$ is a primitive sixth root of unity.
Then the equation tranforms to the rational elliptic surface

\[
R_s^{(6)}: Y^2 = X^3 - A X + 6^6 (s^6-6 B s^4+9 B^2 s^2-2 B^3) + C
\]
where $s=s_{0}=s_{6,0}$ and 
\[
\left\{
\begin{aligned}
A &= 27 (f^{12}+232 f^9+960 f^6+256 f^3+256) (g^{12}+232 g^9+960 g^6+256 g^3+256) \\
B &= q r f g (f^3-8) (g^3-8) \\
C &= 54 (f^{18}-516 f^{15}-12072 f^{12}-24640 f^9-30720 f^6+6144 f^3+4096) \\
& \qquad \times (g^{18}-516 g^{15}-12072 g^{12}-24640 g^9-30720 g^6+6144 g^3+4096).
\end{aligned}
\right.
\]

The Mordell-Weil lattice of this elliptic surface is generically
$E_8$. It has generically twelve $\I_1$ fibers, none of them defined
over the ground field $k(u,v)$. Therefore, one has to proceed by brute
force in order to enumerate these. Taking $X$ to be a quadratic
polynomial in $s$, and $Y$ cubic, with undetermined coefficients, we
obtain a system of equations for the coefficients. The $240$ solutions
give the sections of minimal height. Since these are complicated to
write down, we will not do so here. The formulas may be found in the
auxiliary files. Instead, we will give a more conceptual description
below, in terms of sections arising from $F^{(3)}$ and its twist, the
cubic surface.

We may also form the rational elliptic surface in terms of $s_1$: the
equation becomes (with the same values of $A, B, C$ as above):
\[
R_{s_1}^{(6)}: Y^2 = X^3 - A X + 6^6 (s_1^6-6 \zeta B s_1^4+9 \zeta^2 B^2 s_1^2-2
\zeta^3 B^3) + C.
\]
A similar calculation gives the $240$ minimal height sections for this
elliptic surface. Let $P_1, \dots, P_8$ be the sections coming from
$R_s^{(6)}$ and $P'_1, \dots, P'_8$ those from $R_{s_1}^{(6)}$.

\begin{theorem} \label{th:F^6-RES}
Let $(r,f), (q,g) \in X(6)(\kbar)$ be such that $E_{r,f}$ and
$E_{q,g}$ are not isogenous. The sections $P_1, \dots, P_8, P'_1,
\dots, P'_8$ form a basis for the Mordell-Weil group of
$F^{(6)}(\kbar(t_{6}))$.
\end{theorem}
\begin{proof}
By direct calculation of the height pairing, we find that the
discriminant of the sublattice of the Mordell-Weil group spanned by
these sections is $6^4$. Therefore it must be the full group.
\end{proof}

\begin{corollary}
Let $E_{1}$ and $E_{2}$ be elliptic curves over $k$.  Suppose $E_{1}$
and $E_{2}$ are not isogenous.  Then, the Mordell-Weil lattice
$F^{(6)}_{E_{1},E_{2}}(\bar k(t_{6}))$ is defined over
$k(E_{1}[6],E_{2}[6])$, the field over which all the $6$-torsion
points of $E_{1}$ and $E_{2}$ are defined.
\end{corollary}

\subsection{$F^{(6)}$ as a double cover of a cubic surface}

Now we describe a different method to compute the Mordell-Weil group
of $F^{(6)}$, going through its quotient $F^{(3)}$ and a quadratic
twist of this quotient, which is a rational surface. In the remainder
of this subsection, we will let $\lambda$ and $\mu$ be parameters on
$X(2)$. For an elliptic curve $E$ over $k(t)$, we denote by
${}^{t}\!E$ its quadratic twist.

\begin{lemma}
Let $E_{1}$ and $E_{2}$ be given as in \eqref{ell}.  The
Kodaira-N\'eron model of the quadratic twist
${}^{t_{3}}\!F^{(3)}_{E_{1},E_{2}}$ is birationally equivalent to the
cubic surface given by
\[
Z^{3}+cZW^{2}+dW^{3}=X^{3}+aXY^{2}+bY^{3}.
\]
\end{lemma}

\begin{proof}
The equation of ${}^{t_{3}}\!F^{(3)}$ is given by 
\begin{equation}\label{F3twist}
t_{3}Y^2 = X^3 - 3\,ac\,X +
\frac{1}{64}\Bigl(\Delta_{E_{1}}t_{3}^{3}+864\,bd+\frac{\Delta_{E_{2}}}{t_{3}^{3}}\Bigr).
\end{equation}
Since $t_{3}=t_{6}^{2}$, this equation can be written as
\begin{equation}\label{twist2}
(t_{6}Y)^2 = X^3 - 3\,ac\,X +
  \frac{1}{64}\Bigl(\Delta_{E_{1}}t_{6}^{6}+864\,bd+\frac{\Delta_{E_{2}}}{t_{6}^{6}}\Bigr).
\end{equation}
Rewriting the change of coordinates \eqref{change-of-var} using $t_{6}^{2}=t_{3}$, we see that $X$ and $t_{6}Y$ are written in terms of $t_{3}$:
\[
\left\{
\begin{aligned}
&X=\frac{-t_{3}(2at_{3}^2-c)x_{1} - 3(bt_{3}^3-d) - (at_{3}^2-2c)x_{2}}
{t_{3}(t_{3}x_{1}-x_{2})},
\\
&t_{6}Y=\frac{6(at_{3}^2-c)(bt_{3}^3-d)+6(at_{3}^2-c)(at_{3}^3x_{1}-cx_{2})
-9(bt_{3}^3-d)(t_{3}^2x_{1}^2-x_{2}^2)}
{2t_{3}(t_{3}x_{1}-x_{2})^2}.
\end{aligned}
\right.
\]
Plugging these back into \eqref{twist2}, we obtain the equation
\[
x_{2}^{3}+cx_{2}+d = t_{3}^{3}(x_{1}^{3}+ax_{1}+b).
\]
Now, if we let $x_{1}=X/Y$, $x_{2}=Z/W$ and $t_{3}=Y/W$, we obtain the
desired homogenous cubic equation.
\end{proof}

\begin{lemma}\label{q-twist}
Let $E$ be an elliptic curve over $k(T)$, and ${}^{T}\!E$ its
quadratic twist by $T$.  Then $E(k(T))\oplus{}^{T}\!E(k(T))$ is a
subgroup of finite index of $E(k(\sqrt{T}))$.
\end{lemma}

We will use the following well-known lemma to put together the
sections from the quotients $F^{(3)}$ and ${}^{t_{3}}\!F^{(3)}$.
\begin{proof}
Let $\iota$ be the automorphism $\sqrt{T}\mapsto -\sqrt{T}$.  Then the
composition of the maps
\[\renewcommand{\arraystretch}{1.414}
\begin{array}{*5c}
E(k(\sqrt{T})) & \longrightarrow & E(k(T))\oplus{}^{T}\!E(k(T)) 
& \longrightarrow & E(k(\sqrt{T})) \\
P & \longmapsto & (P+\iota(P),P-\iota(P)) &  & \\
& & (Q,R) & \longmapsto & Q+R
\end{array}
\]
is the multiplication-by-$2$ map $[2]$.  Since the image of $[2]$ is a
subgroup of finite index, the assertion follows.
\end{proof}

In order to compute ${}^{t_{3}}\!F^{(3)}(\bar k(t_{3}))$, we find the
twenty-seven lines contained in the cubic surface.  To state the
results clearly, we take $E_{1}$ and $E_{2}$ to be in Legendre form as
in \eqref{legendre}, and consider the cubic surface
\begin{equation}\label{cubic-legendre}
Z(Z-W)(Z-\mu W) = X(X-Y)(X-\lambda Y).
\end{equation}

As is well-known, the group generated by $\lambda\mapsto 1-\lambda$
and $\lambda\mapsto 1/\lambda$ leave the $j$-invariant of
$y^{2}=x(x-1)(x-\lambda)$.  As in Proposition~\ref{action:fourfold},
this action lifts to the family of cubic surfaces.
\begin{proposition}\label{prop:action-cubic-surface}
There are automorphisms acting on the family of cubic surfaces
\eqref{cubic-legendre} parametrized by $(\lambda, \mu)$:
\begin{align*}
&\bigl((X:Y:Z:W),\lambda,\mu\bigr) \mapsto \bigl((X-Y:-Y:Z:W),1-\lambda,\mu\bigr), \\
&\bigl((X:Y:Z:W),\lambda,\mu\bigr) \mapsto \bigl((X:\lambda Y:Z:W),1/\lambda,\mu\bigr), \\
&\bigl((X:Y:Z:W),\lambda,\mu\bigr) \mapsto \bigl((X:Y:Z-W:-W),\lambda,1-\mu\bigr), \\
&\bigl((X:Y:Z:W),\lambda,\mu\bigr) \mapsto \bigl((X:Y:Z:\mu W),\lambda,1/\mu\bigr). \\
\end{align*}
These automorphisms act on the set of twenty-seven lines contained in
\eqref{cubic-legendre}
\end{proposition}

Let
$\delta=(\Delta_{E_{2}}/\Delta_{E_{1}})^{1/6}=\bigl(\mu(\mu-1)\bigr)^{1/3}/\bigl(\lambda(\lambda-1)\bigr)^{1/3}$. Note
$\delta^3$ is invariant under $\lambda \to 1- \lambda$ and $\mu \to 1-
\mu$, whereas under $\lambda \to 1/\lambda$, it is taken to
$-\delta^3/\lambda^3 = (-\delta/\lambda)^3$. Therefore, we may extend
the action of the group of automorphisms to $\delta$ in a natural
way. The next result shows that all the twenty-seven lines are defined
over the cubic extension field of $k(\lambda, \mu)$ defined by
$\delta$.

\begin{proposition}\label{prop:27-lines}
The twenty-seven lines contained in the cubic surface
\eqref{cubic-legendre} are given as follows.  In
terms of homogeneous parameters $(\alpha: \beta)$ on $\Proj^1$, the lines
$(X,Y,Z,W)$ belong to the following list:
\begin{small}
\begin{gather*}
\begin{array}{ccc}
(0:\alpha:0:\beta), \quad &
(\alpha:\alpha:0:\beta), \quad &
(\lambda \alpha:\alpha:0:\beta), \\[\smallskipamount]
(0:\alpha:\beta:\beta), \quad &
(\alpha:\alpha:\beta:\beta), \quad &
(\lambda \alpha:\alpha:\beta:\beta), \\[\smallskipamount]
(0:\alpha:\mu \beta:\beta), \quad &
(\alpha:\alpha:\mu \beta:\beta), \quad &
(\lambda \alpha:\alpha:\mu \beta:\beta),
\end{array}
\\[\smallskipamount]
{\setlength{\arraycolsep}{2pt}
\begin{array}{rcccccccl}
\bigl(&\lambda (\mu-1) \alpha - \delta^2 \lambda (\lambda-1) \beta &:& (\mu -\lambda) \alpha &:& \delta \lambda (\lambda-1) \alpha - \mu (\lambda-1) \beta &:& (\mu -\lambda) \beta&\bigr),
\\[\smallskipamount]
\bigl(&\lambda \alpha + \delta^2 \lambda (\lambda-1) \beta &:& (\lambda+\mu-\lambda \mu) \alpha &:& \delta \lambda (\lambda-1) \alpha + \mu \beta &:& (\lambda+\mu-\lambda \mu) \beta&\bigr),
\\[\smallskipamount]
\bigl(&\lambda (\mu-1) \alpha - \omega^2 \delta^2 \lambda (\lambda-1) \beta &:& (\mu-\lambda) \alpha &:& \omega \delta \lambda (\lambda-1) \alpha - \mu (\lambda-1) \beta &:& (\mu-\lambda) \beta&\bigr),
\\[\smallskipamount]
\bigl(&\lambda \alpha + \omega^2 \delta^2 \lambda (\lambda-1) \beta &:& (\mu + \lambda -\lambda \mu) \alpha &:& \omega \delta \lambda (\lambda-1) \alpha + \mu \beta &:& (\lambda+\mu-\lambda \mu) \beta&\bigr),
\\[\smallskipamount]
\bigl(&\lambda (\mu-1) \alpha - \omega \delta^2 \lambda (\lambda-1) \beta &:& (\mu-\lambda) \alpha &:&  \omega^2 \delta \lambda (\lambda-1) \alpha - \mu (\lambda-1) \beta &:& (\mu-\lambda) \beta&\bigr),
\\[\smallskipamount]
\bigl(&\lambda \alpha + \omega \delta^2 \lambda (\lambda-1) \beta &:& (\lambda+\mu-\lambda \mu) \alpha &:& \omega^2 \delta \lambda (\lambda-1) \alpha + \mu \beta &:& (\lambda+\mu-\lambda \mu) \beta&\bigr),
\\[\smallskipamount]
\bigl(&\lambda (\mu-1) \alpha + \delta^2 \lambda (\lambda-1) \beta &:& (\lambda \mu-1) \alpha &:& \delta \lambda (\lambda-1) \alpha + \mu (\lambda-1) \beta &:& (\lambda \mu-1) \beta&\bigr),
\\[\smallskipamount]
\bigl(&\lambda (\mu-1) \alpha + \omega^2 \delta^2 \lambda (\lambda-1) \beta &:& (\lambda \mu-1) \alpha &:& \omega \delta \lambda (\lambda-1) \alpha + \mu (\lambda-1) \beta &:& (\lambda \mu-1) \beta&\bigr),
\\[\smallskipamount]
\bigl(&\lambda (\mu-1) \alpha + \omega \delta^2 \lambda (\lambda-1) \beta &:& (\lambda \mu-1) \alpha &:& \omega^2 \delta \lambda (\lambda-1) \alpha + \mu (\lambda-1) \beta &:& (\lambda \mu-1) \beta&\bigr),
\\[\smallskipamount]
 \bigl(&\lambda \mu \alpha - \delta^2 \lambda (\lambda-1) \beta &:& (\lambda+\mu-1) \alpha &:& -\delta \lambda (\lambda-1) \alpha + \lambda \mu \beta &:& (\lambda+\mu-1) \beta&\bigr),
\\[\smallskipamount]
\bigl(&\lambda \mu \alpha + \delta^2 \lambda (\lambda-1) \beta &:& (\lambda \mu-\lambda+1) \alpha &:& -\delta \lambda (\lambda-1) \alpha + \mu \beta &:& (\lambda \mu-\lambda+1) \beta&\bigr),
\\[\smallskipamount]
\bigl(&\lambda \mu \alpha - \omega^2 \delta^2 \lambda (\lambda-1) \beta &:& (\lambda+\mu-1) \alpha &:& -\omega \delta \lambda (\lambda-1) \alpha + \lambda \mu \beta &:& (\lambda+\mu-1) \beta&\bigr),
\\[\smallskipamount]
\bigl(&\lambda \mu \alpha + \omega^2 \delta^2 \lambda (\lambda-1) \beta &:& (\lambda \mu-\lambda+1) \alpha &:& -\omega \delta \lambda (\lambda-1) \alpha + \mu \beta &:& (\lambda \mu-\lambda+1) \beta&\bigr),
\\[\smallskipamount]
\bigl(&\lambda \mu \alpha - \omega \delta^2 \lambda (\lambda-1) \beta &:& (\lambda+\mu-1) \alpha &:& -\omega^2 \delta \lambda (\lambda-1) \alpha + \lambda \mu \beta &:& (\lambda+\mu-1) \beta&\bigr),
\\[\smallskipamount]
\bigl(&\lambda \mu \alpha + \omega \delta^2 \lambda (\lambda-1) \beta &:& (\lambda \mu-\lambda+1) \alpha &:& -\omega^2 \delta \lambda (\lambda-1) \alpha + \mu \beta &:& (\lambda \mu-\lambda+1) \beta&\bigr),
\\[\smallskipamount]
\bigl(&\lambda \alpha - \delta^2 \lambda (\lambda-1) \beta &:& (\lambda \mu-\mu+1) \alpha &:& \delta \lambda (\lambda-1) \alpha + \lambda \mu \beta &:& (\lambda \mu-\mu+1) \beta&\bigr),
\\[\smallskipamount]
\bigl(&\lambda \alpha - \omega^2 \delta^2 \lambda (\lambda-1) \beta &:& (\lambda \mu-\mu+1) \alpha &:& \omega \delta \lambda (\lambda-1) \alpha + \lambda \mu \beta &:& (\lambda \mu-\mu+1) \beta&\bigr),
\\[\smallskipamount]
\bigl(&\lambda \alpha - \omega \delta^2 \lambda (\lambda-1) \beta &:& (\lambda \mu-\mu+1) \alpha &:& \omega^2 \delta \lambda (\lambda-1) \alpha + \lambda \mu \beta &:& (\lambda \mu-\mu+1) \beta&\bigr).
\end{array}
}
\end{gather*}
\end{small}
They may be generated by taking the orbits of the first and the last
three lines in the list, under the action of the automorphism group
defined by Proposition \ref{prop:action-cubic-surface}.
\end{proposition}

\begin{proof}
The first nine lines are obvious ones; they are obtained by letting
one of the factors of the left hand side of the
equation \eqref{cubic-legendre} equal $0$ and one of the right hand
side equal $0$.  The other eighteen lines are obtained as follows.
Take a factor from the left hand side and another from the right hand
side, say $Z-Y$ and $X-\lambda W$.  Take a parameter $m$ and let
$Z-Y=m(X-\lambda W)$. We will take the intersections of the cubic
surface with this family of planes. By construction, they always
contain the line $Z - Y = X - \lambda W = 0$. The family of residual
conics will degenerate to pairs of lines at suitable values of $m$.

Concretely, we replace $Z$ by $m(X-Y)+W$ in the equation of the
surface, and we obtain a family of conics in $X,Y,W$:
\[
(m^{3}-1)X^2 - (\mu-1)mW^2 + m^3Y^2 
- (\mu-2)m^{2}XW -(2m^3-\lambda)XY + (\mu-2)m^{2}YW = 0.
\]
Writing this equation in matrix form:
\[\renewcommand{\arraystretch}{1.2}
\begin{pmatrix} X & Y & W\end{pmatrix}
\left(\begin{array}{*3c}
m^{3} - 1 & -m^{3}+\lambda/2   & -(\mu-1)m^{2}/2\\
-m^{3}+\lambda/2 & m^{3} & (\mu-2)m^{2}/2 \\
-(\mu-1)m^{2}/2 & (\mu-2)m^{2}/2 & - (\mu-1)m 
\end{array}\right)
\begin{pmatrix} X \\ Y \\ W \end{pmatrix} = 0.
\]
We then calculate the determinant of this matrix:
\[
-(\lambda-1)m(\mu m -\lambda \delta)
(\mu m - \omega\lambda\delta) (\mu m -\omega^{2}\lambda\delta)/(4\mu).
\]
So, at
$m=0,\lambda\delta/\mu,\omega\lambda\delta/\mu,\omega^{2}\lambda\delta/\mu$,
the conic becomes a pair of lines. We repeat this process, and
eliminate the duplicates to obtain the list of all the twenty-seven
lines.
\end{proof}

The elliptic surface ${}^{t_{3}}F^{(3)}_{E_{\lambda},E_{\mu}}$ is
obtained from the family of plane cubic curves
\[
x_{2}(x_{2}-z)(x_{2}-\mu z) = t_{3}^{3} x_{1}(x_{1}-z)(x_{1}-\lambda z)
\]
and the rational point $(x_{1}:x_{2}:z)=(1:t_{3}:0)$. Let us recall
its equation (a specialization of \eqref{F3twist}, where the elliptic
curves are given by \eqref{cubic-legendre}):
\begin{align*}
t y^2 &= x^3  -27 (\lambda^2-\lambda+1) (\mu^2-\mu+1) x \\
      & \qquad + 729 \lambda^2 (\lambda-1)^2 t^3/4    + 729 \mu^2 (\mu-1)^2/(4 t^3) \\
      & \qquad + 27 (\lambda-2) (\lambda+1) (2 \lambda-1) (\mu-2) (\mu+1) (2 \mu-1)/2
\end{align*}

There are two
other rational points $(1:\omega t_{3}:0)$ and $(1:\omega^{2}
t_{3}:0)$, which become sections of
${}^{t_{3}}F^{(3)}_{E_{\lambda},E_{\mu}}$ given by
\begin{align*}
R_{1}&=\Bigl(-3\bigl(\omega^2(\lambda^2-\lambda+1)t^2+\omega(\mu^2-\mu+1)\bigl),
\\
& \qquad \qquad \qquad (2\omega+1)\bigl((\lambda+1)(\lambda-2)(2\lambda-1)t^3-(\mu+1)(\mu-2)(2\mu-1)\bigr)/2 \Bigr),
\\
R_{2}&=\Bigl(-3\bigl(\omega(\lambda^2-\lambda+1)t^2+\omega^{2}(\mu^2-\mu+1)\bigl),
\\
& \qquad \qquad \qquad  (2\omega+1)\bigl((\lambda+1)(\lambda-2)(2\lambda-1)t^3-(\mu+1)(\mu-2)(2\mu-1)\bigr)/2 \Bigr).
\end{align*}

The elliptic surface ${}^{t_{3}}F^{(3)}_{E_{\lambda},E_{\mu}}$ can be
obtained by blowing up the cubic surface \eqref{cubic-legendre} at
three points corresponding to these, namely, $(X:Y:Z:W)=(1:0:1:0)$,
$(1:0:\omega:0)$, and $(1:0:\omega^{2}:0)$.

\begin{theorem}
Let $\lambda, \mu$ be such that $E_{\lambda}$ and $E_{\mu}$ are not
isomorphic over $\bar k$.  The Mordell-Weil lattice
${}^{t_{3}}F^{(3)}_{E_{\lambda},E_{\mu}}(\kbar(t_{3}))$
is of type $E_{8}$, and generated by $R_{1}$, $R_{2}$ above and the
sections coming from the the twenty-seven lines in the cubic surface.
\end{theorem}

\begin{proof}
For generic $E_{\lambda}$ and $E_{\mu}$,
${}^{t_{3}}F^{(3)}_{E_{\lambda},E_{\mu}}$ is a rational elliptic
surface with only irreducible singular fibers.  Thus, its Mordell-Weil
lattice is isomorphic to $E_{8}$. The N\'eron-Severi group of the
cubic surface is generated by the classes of the twenty-seven lines,
which form a lattice isometric to $E_6$ (see Manin
\cite{Manin}). Therefore, the N\'eron-Severi group of the rational
surface ${}^{t_{3}}F^{(3)}_{E_{\lambda},E_{\mu}}$ is generated by the
exceptional divisors of blow-ups and the twenty-seven lines.
Transforming to the elliptic model, we see that the Mordell-Weil
lattice ${}^{t_{3}}F^{(3)}_{E_{\lambda},E_{\mu}}(\bar
k(\lambda,\mu)(t_{3}))$ is generated by $R_{1}$, $R_{2}$ and sections
coming from the twenty-seven lines.
\end{proof}

\begin{remark}
The sections $R_1$ and $R_2$ above, along with the sections coming
from lines $1$,$2$,$4$,$5$,$10$ and $12$, form a basis of the
Mordell-Weil lattice. Below, we display some of these sections; the
remaining ones are omitted for lack of space. The formulas for the
full basis may be obtained from the auxiliary files.
\end{remark}

\begin{align*}
R_3 &:= L_1 = \Big( 3\big( 3lt^2 + (l+1)(m+1)t + 3m \big)/t,  \\
& \qquad \qquad \quad -27\big( l(l+1)t^3 + 2l(m+1)t^2 \\
& \qquad \qquad \quad + 2(l+1)mt + m(m+1) \big)/(2t^2) \Big) \\
R_4 &:= L_2 = \Big( -3\big( 3(l-1)t^2 - (l-2)(m+1)t - 3m \big)/t, \\
&  \qquad \qquad \quad 27\big( (l-1)(l-2)t^3 + 2(l-1)(m+1)t^2 \\
& \qquad \qquad \quad -2(l-2)mt -m(m+1) \big)/(2t^2) \Big) \\
R_5 &:= L_4 = \Big( 3\big( 3lt^2 + (l+1)(m-2)t - 3(m-1)\big)/t,  \\
&  \qquad \qquad \quad -27\big( l(l+1)t^3 + 2l(m-2)t^2 - 2(m-1)(l+1)t \\
& \qquad \qquad \quad - (m-1)(m-2)\big)/(2t^2) \Big) \\
R_6 &:= L_5 = \Big(-3\big( 3(l-1)t^2 - (l-2)(m-2)t + 3(m-1)\big)/t,  \\
&  \qquad \qquad \quad 27\big( (l-1)(l-2)t^3 + 2(l-1)(m-2)t^2 \\
& \qquad \qquad \quad + 2(l-2)(m-1)t + (m-1)(m-2)\big)/(2t^2) \Big)
\end{align*}

\begin{remark}
The sections coming from the twenty-seven lines form a sublattice of
type $E_{7}$ in the Mordell-Weil lattice, and the first nine lines in
Proposition~\ref{prop:27-lines}, which are defined over~$k$, form a
sublattice of type $A_{5}$. These nine lines together with the
sections $R_1$ and $R_2$ from the blowup generate a sublattice of type
$E_7$.
\end{remark}

\begin{corollary}
Let $\lambda, \mu \in \kbar$ be such that $E_\lambda$ and $E_\mu$ are
not isomorphic over $\kbar$. The field of definition of
${}^{t_{3}}F^{(3)}_{E_{\lambda},E_{\mu}}(\bar k(t_{3}))$ is
$k(\lambda,\mu,\delta,\omega)$.  If $E_{1}$ and $E_{2}$ are not
isomorphic over $\kbar$, then the field of definition of
${}^{t_{3}}F^{(3)}_{E_{1},E_{2}}(\bar k(t_{3}))$ is
$k\bigl(E_{1}[2],E_{2}[2],(\Delta_{E_{2}}/\Delta_{E_{1}})^{1/6},\omega\bigr)$.
\end{corollary}

\subsection{$F^{(6)}(\kbar(t_{6}))$ in the generic case} \label{rationalF3twist}

Let $Q_1, \dots, Q_8$ be the basis of the Mordell-Weil group for
$F^{(3)}_{E_{u},E_{v}}$ described in Theorem~\ref{th:F^(3)}. By abuse
of notation, let $Q_1, \dots, Q_8$ be their base change (pullback) to
$F^{(6)}_{E_{r,f},E_{q,g}}$ defined by \eqref{eq:generic-F6}, by the
map $\lambda = (r+1)(r-3)^3/\big((r-1)(r+3)^3\big)$ and $\mu =
(q+1)(q-3)^3/\big((q-1)(q+3)^3\big)$. Similarly, let $R_1, \dots,
R_8$ be the base change of the basis of
${}^{t_{3}}F^{(3)}_{E_{\lambda},E_{\mu}}(\bar k(t_{3}))$.  ($R_7$ and
$R_8$ are shown only in the auxiliary files.)  Define
\begin{align*}
S_1 &= (Q_2 + Q_3 + Q_6 + Q_8 + R_1)/2 \\
S_2 &= (Q_1 + Q_3 + Q_7 + Q_8 - R_2)/2 \\
S_3 &= (Q_1 + Q_3 + R_1 - R_7 + R_8)/2 \\
S_4 &= (Q_2 + Q_3 + R_4 + R_5 - R_8)/2.
\end{align*}
These are sections of $F^{(6)}_{E_{r,f},E_{q,g}}$, i.e. the
expressions in parentheses are (uniquely) divisibe by $2$ in the
Mordell-Weil group.  Explicit formulas are also given in auxiliary
files.

\begin{theorem} \label{th:F^6-cubic}
Let $(r,f), (q,g) \in X(6)(\kbar)$ be such that $E_{r,f}$ and
$E_{q,g}$ are not isogenous. The sections $Q_1, \dots, Q_8, R_3, R_4,
R_5, R_6, S_1, S_2, S_3, S_4$ form a basis of the Mordell-Weil group
of $F^{(6)}_{E_{r,f},E_{q,g}}(\kbar (t_{6}))$.
\end{theorem}

\begin{proof}
By construction and base change, the lattice spanned by the $Q_i$'s
and the $R_i$'s has discriminant $(3^4/2^4) \cdot 2^8 \cdot 1 \cdot
2^8 = 6^4 \cdot 2^8$. Since the lattice spanned by the new basis is an
overlattice of index $16$, it has discriminant $6^4$, which matches
the discriminant of the Mordell-Weil lattice of $F^{(6)}$, as computed
by Shioda~\cite{Shioda:sphere-packing}. Therefore it must be the full
Mordell-Weil group.
\end{proof}

\begin{remark}
We have described two different bases for the Mordell-Weil lattice
$F^{(6)}(\kbar(t_{6}))$, obtained through two different methods: first
by using rational elliptic surfaces parametrized by $s_{6,i}$, and
second by using $F^{(3)}$ and the cubic surface that is a twist of
$F^{(3)}$. The first method, though similar in spirit to that for
$F^{(4)}$ and $F^{(5)}$, is significantly more difficult to carry out
computationally. The change of basis matrix for these two bases is
also given in the auxiliary files.
\end{remark}

\section{Singular K3 surfaces} \label{sec:singular}

In this section we consider $K3$ surfaces with Picard number~$20$.
These surfaces are called \emph{singular $K3$ surfaces} because they
do not involve any moduli. We are interested in elliptic $K3$ surfaces
defined over $\Q$ whose Mordell-Weil rank (over $\Qbar$) is
maximal~$18$. If the Mordell-Weil rank of an elliptic $K3$ surface
is~$18$, the underlying $K3$ surface must be a singular $K3$
surface. Our goal in this section is to construct as many such
elliptic $K3$ surfaces as possible.

Singular $K3$ surfaces are closely related to elliptic curves with
complex multiplication.  We use work of
Shioda-Mitani~\cite{Shioda-Mitani}, Shioda-Inose~\cite{Shioda-Inose},
Inose~\cite{Inose:singular-K3}, and the theory of complex
multiplication (see for example \cite{Cox-primes}).  Shioda and Inose
\cite{Shioda-Inose} show that a complex singular $K3$ surface $X$ is
what we call the Inose surface $\Ino(E_{1},E_{2})$ for some elliptic
curves $E_{1}$ and $E_{2}$ that have complex multiplication and are
isogenous to each other.  More specifically, we have

\begin{theorem}[Shioda-Inose~\cite{Shioda-Inose}]\label{th:singular-K3}
There is a one-to-one correspondence between the set of isomorphism
classes of complex singular $K3$ surfaces and the set of equivalence
classes of even positive definite Euclidean lattices, or equivalently,
positive definite integral binary quadratic forms, with respect to
$\SL_{2}(\Z)$\textup{:}
\begin{align*}
&\{\textup{singular $K3$ surfaces over $\C$}\}/\textup{$\C$-isomorphisms}
\\
&\overset{1:1}{\longleftrightarrow}
\left\{\left.\left.
\begin{pmatrix} 2a & b \\ b & 2c\end{pmatrix}\,\right|\,
a,b,c\in \Z,\quad a, c>0, \quad b^{2}-4ac<0\right\}\right/\SL_{2}(\Z)
\\
&\overset{1:1}{\longleftrightarrow}
\{ ax^{2}+bxy+cy^{2}\mid a,b,c\in \Z, \quad a, c>0, \quad b^{2}-4ac<0\}/\SL_{2}(\Z)
\end{align*}
\end{theorem} 

In fact, Shioda-Inose~\cite{Shioda-Inose} construct a singular $K3$
surface $X$ corresponding to the lattice $Q=\begin{pmatrix} 2a & b
\\ b & 2c\end{pmatrix}$, or the quadratic form $ax^{2}+bxy+cy^{2}$, as
follows.  First, let $\tau_{1}$ and $\tau_{2}$ be the points on the
upper half plane $\H$ given by
\begin{equation}\label{eq:tau}
\tau_{1}= \frac{-b+\sqrt{D}}{2a}, \quad
\tau_{2}= \frac{b+\sqrt{D}}{2}, \quad \text{where} \quad D = b^{2}-4ac.
\end{equation}
Let $j(\tau)$ be the elliptic modular function defined on $\H$, and
let $E_{1}$ and $E_{2}$ be elliptic curves whose $j$-invariants are
$j(\tau_{1})$ and $j(\tau_{2})$ respectively.  For example, $E_{i}$
can be given by
\[
E_{i}:y^2 = x^3-\frac{3j(\tau_{i})}{j(\tau_{i})-1728}\,x+\frac{2j(\tau_{i})}{j(\tau_{i})-1728}, \quad i=1,2.
\]
Then, the Inose surface $\Ino(E_{1},E_{2})$ is a singular $K3$ surface
corresponding to~$Q$.

First, consider the case where $ax^{2}+bxy+cy^{2}$ is primitive, that
is, $\gcd(a,b,c)=1$.  Since $b$ and $D$ have the same parity, on the
upper half plane~$\H$, $\tau_{2}=(b+\sqrt{D})/2$ represents the same
point as the root $\sqrt{D}/2$ or $(-1+\sqrt{D})/2$ of the trivial
form
\[
\begin{cases}
x^{2}+\dfrac{-D}{4}y^{2} & \text{if $D\equiv 0 \bmod 4$,}
\\[1ex]
x^{2}+xy+\dfrac{1-D}{4}y^{2} & \text{if $D\equiv 1 \bmod 4$}
\end{cases}
\]
modulo the action of the modular group $\SL_{2}(\Z)$.  The lattice
$\sO=\<1,\tau_{2}\>$ spanned by $1$ and $\tau_{2}$ is an order in the
imaginary quadratic field $K=\Q(\sqrt{D})$ (in fact, it is the unique
order of discriminant $D$), and the lattice
$\mathfrak{a}=\<1,\tau_{1}\>$ is a proper ideal of an order.  It is
well-known from the theory of complex multiplication that
$j(\sO)=j(\tau_{1})$ and $j(\mathfrak{a})=j(\tau_{2})$ are conjugate
roots of the class equation $H_{\sO}(X)=0$ (see
\cite[\S13]{Cox-primes}).  The degree of the class equation is the
class number $h(\sO)=h_{D}$.

\begin{theorem}\label{th:classno=2}
Let $D$ be a negative integer $\equiv0$ or $1\bmod 4$.  Suppose that
its class number $h_{D}$ equals~$2$, and let $ax^{2}+bxy+cy^{2}$ be
the nontrivial element of the class group $\textit{Cl}(D)$.  Then, the
Inose surface $\Ino(E_{1},E_{2})$ corresponding to $ax^{2}+bxy+cy^{2}$
has a model defined over $\Q$.  Furthermore, the Mordell-Weil lattices
$F^{(n)}_{E_{1},E_{2}}(\Qbar(t_{n}))$, $n=5,6$, constructed from
$\Ino(E_{1},E_{2})$ have rank~$18$.
\end{theorem}

\begin{proof}
Since $h_{D}=2$, both $j(\tau_{1})$ and $j(\tau_{2})$ are conjugate
elements of a quadratic extension of $\Q$.  If we take $E_{1}$ and
$E_{2}$ to be conjugate to each other, then Lemma~\ref{lem:quadratic}
below assures that $\Ino(E_{1},E_{2})$ and all $F^{(N)}_{E_{1},E_{2}}$
has a model over~$\Q$.  Since $ax^{2}+bxy+cy^{2}$ corresponds to the
nontrivial element, $j(\tau_{2})\neq j(\tau_{1})$, and $E_{1}$ and
$E_{2}$ are not isomorphic. But they are isogenous, as they come from
ideals in the same quadratic field. Thus,
$F^{(n)}_{E_{1},E_{2}}(\Qbar(t_{m}))$, $n=5,6$ have rank~$18$ by
Proposition~\ref{prop:Fn-MWrank}.
\end{proof}

\begin{lemma}\label{lem:quadratic}
Let $E_{1}$ be an elliptic curve defined over an quadratic field
$\Q(\sqrt{d})$ and let $E_{2}$ be its conjugate.  Then, the elliptic
fibration $F^{(n)}_{E_{1},E_{2}}$ has a model defined over $\Q$.
\end{lemma}

\begin{proof}
Let $E_{1}$ be given by $y^{2} = x^{3} + ax + b$, $a,b \in
\Q(\sqrt{d})$.  $E_{2}$ is given by $y^{2} = x^{3} + {\bar a}x + {\bar
  b}$, where $\bar\quad$ stands for the conjugate
$\sqrt{d}\mapsto-\sqrt{d}$.  Then the equation \eqref{eq:F6} can be
written as
\[
Y^2 = 
X^3 - 3\,a{\bar a}\,X  
+ \frac{1}{64}\Bigl(\Delta_{E_{1}}t_{6}^{6}
+864\,b{\bar b}
+\frac{\Delta_{E_{1}}\overline{\Delta_{E_{1}}}}{\Delta_{E_{1}}{t_{6}^{6}}}\Bigr).
\]
Thus, if we let $T=\Delta_{E_{1}}^{1/6}t_{6}$, the equation of
$F^{(6)}_{E_{1},E_{2}}$ is given by
\[
Y^2 = X^3 - 3\,N(a)\,X  
+ \frac{1}{64}\Bigl(T^{6}
+864\,N(b)+\frac{N(\Delta_{E_{1}})}{T^{6}}\Bigr),
\]
where $N:\Q(\sqrt{d})\to \Q$ is the norm.
\end{proof}

\begin{remark}
The lemma also follows from the proof of Proposition 8.1 in
\cite{Schuett}.
\end{remark}

\begin{remark}
An elliptic curve $E$ over the Hilbert class field $H$ of an imaginary
quadratic field $K$ with complex multiplication by $K$ is called a
\emph{$\Q$-curve} (in the original sense) if $E$ is isogenous over $H$
to all its Galois conjugates.  Theorem~\ref{th:classno=2} shows that
we can obtain elliptic $K3$ surfaces defined over $\Q$ with
Mordell-Weil rank~$18$ from a $\Q$-curve defined over a quadratic
Hilbert class field.  However, for our purpose, we do not need the
isogeny between $E$ and its Galois conjugate to be defined over $H$.
In fact, we will see some examples of elliptic curves with complex
multiplication by $K$ such that they are isogenous to their Galois
conjugates only over some extension of $H$.
\end{remark}

Next, consider the case where $ax^{2}+bxy+cy^{2}$ is not primitive.
Write $ax^{2}+bxy+cy^{2}=m(a'x^{2}+b'xy+c'y^{2})$, where $m>1$ and
$\gcd(a',b',c')=1$.  Define
\[
\tau_{1}'=\frac{-b'+\sqrt{D'}}{2a'}, \text{ and } 
\tau_{2}'=\frac{b'+\sqrt{D'}}{2}, \text{ where }
D'=b'^{2}-4a'c'.
\]
Then, we have $\tau_{1}=\tau_{1}'$ and $\tau_{2}=m\tau_{2}'$.  Now,
suppose that $h_{D'}=2$.  Then, $j(\tau'_{1})$ belongs to some
quadratic extension and $j(\tau'_{2})$ is its conjugate.  Since
$j(\tau'_{1})=j(\tau_{1})$, in order for our method to work, we need
that $j(\tau_{2})=j(m\tau'_{2})$ is conjugate to $j(\tau_{1})$.  But,
this implies $j(\tau'_{2})=j(\tau_{2})$ and thus the Inose surface
corresponding to $ax^{2}+bxy+cy^{2}$ and that to
$a'x^{2}+b'xy+c'y^{2}$ are isomorphic. So, in the nonprimitive case we
can reduce to the case of the primitive discriminant $D$,
corresponding to the latter quadratic form.

\subsection{Class number 1 case}
It is well-known that there are thirteen discriminants of class number~$1$. 
\[
h_{D}=1 \Longleftrightarrow D=-3,-4,-7,-8, -11,-12,-16,-19,-27,-28,-43,-67,-163.
\]
(See for example \cite[Theorem~7.30]{Cox-primes} and its references.)
From these we see that only the fields $\Q(\sqrt{-1})$,
$\Q(\sqrt{-3})$ and $\Q(\sqrt{-7})$ possess non-maximal order of class
number~$1$. In the following table, $D$ is a discriminant with
$h_{D}=1$, $f$ is the conductor of the order, $\tau$ is the root of
the quadratic form, and $E$ is an example of $E$ having $j(\tau)$ as
its $j$-invariant.
\begin{footnotesize}
\[
\renewcommand{\arraystretch}{1.4}\setlength{\arraycolsep}{6pt}
\begin{array}{c|c|c|c|c|c|l}
K & D & f & \textup{quadratic form} &\tau & j(\tau) & \hfil\textup{example of $E$} \\
\hline
\multirow{3}{*}{$\Q(\sqrt{-3})$} 
& -3 & 1 & x^{2}+xy+y^{2} &(-1+\sqrt{-3})/2 & 0 & 
y^{2} = x^{3} + 1 \\
& -12 & 2 & x^{2} + 3x^{2}& \sqrt{-3} & 54000 & 
y^{2} = x^{3} -15x +22 \\
& -27 & 3 & x^{2} +xy + 7y^{2}& (-1+3\sqrt{-3})/2 & -12288000 & 
y^{2} = x^{3} + 18x^{2} - 12x +2 \\\hline 
\multirow{2}{*}{$\Q(\sqrt{-1})$} 
& -4 & 1 & x^{2}+y^{2} & \sqrt{-1} & 1728 & 
y^{2} = x^{3} - x \\ 
& -16 & 2 & x^{2} + 4y^{2}& 2\sqrt{-1} &  287496 & 
y^{2} = x^{3} - 11 x - 14 \\\hline 
\multirow{2}{*}{$\Q(\sqrt{-7})$}
& -7 & 1 & x^{2} + xy + 2y^{2} & (-1+\sqrt{-7})/2 & -3375 & 
y^{2} = x^{3} - 21 x^{2} + 112 x \\
& -28 & 2 & x^{2}+ 7y^{2} & \sqrt{-7} & 16581375 & 
y^{2} = x^{3} + 42 x^{2} - 7x \\
\hline
\end{array}
\]
\end{footnotesize}
\begin{theorem}\label{th:classno1}
Let $E_{1}$ and $E_{2}$ be the pair in the table below.  Then, the
Mordell-Weil lattice $F^{(n)}_{E_{1},E_{2}}(\Qbar(t_{n}))$ for $n=5,6$
has rank~$18$.
\begin{footnotesize}
\[
\renewcommand{\arraystretch}{1.6}\setlength{\arraycolsep}{0.5em}
\begin{array}{c|l|l|c}
K & \hfil E_{1}, \quad E_{2} & \hfil F^{(n)}_{E_{1},E_{2}} &
T_{\Ino(E_{1},E_{2})} \\\hline
\multirow{6}{*}{$\Q(\sqrt{-3})$} 
& y_{1}^{2} = x_{1}^{3} + 1 
& \multirow{2}{*}{$Y^{2} = X^3 + \dfrac{t_{n}^n}{4} - 11 - \dfrac{4}{t_{n}^n}$} 
& \multirow{2}{*}{$\renewcommand{\arraystretch}{1}\setlength{\arraycolsep}{4pt}\begin{pmatrix} 4 & 2 \\ 2 & 4\end{pmatrix}$} 
\\[-4pt]
& y_{2}^{2} = x_{2}^{3} -15 x_{2} + 22 & &
\\\cline{2-4}
& y_{1}^{2} = x_{1}^{3} - 432 
& \multirow{2}{*}{$Y^{2} = X^3 +t_{n}^{n} -506 + \dfrac{9}{t_{n}^{n}}$}
& \multirow{2}{*}{$\renewcommand{\arraystretch}{1}\setlength{\arraycolsep}{4pt}\begin{pmatrix} 6 & 3 \\ 3 & 6\end{pmatrix}$} 
\\[-4pt]
& y_{2}^{2} = x_{2}^{3} - 108 x_{2}^{2} - 432x_{2} - 432 & &
\\\cline{2-4}
& y_{1}^{2} = x_{1}^{3} -15 x_{1} + 22
& \multirow{2}{*}{$Y^{2} = X^3 - 600X +9t_{n}^{n} -5566 -\dfrac{4}{t_{n}^{n}}$}
& \multirow{2}{*}{$\renewcommand{\arraystretch}{1}\setlength{\arraycolsep}{4pt}\begin{pmatrix} 12 & 6 \\ 6 & 12 \end{pmatrix}$} 
\\[-4pt]
& y_{2}^{2} = x_{2}^{3} +18 x_{2}^{2} - 12x_{2} + 2 &
\\\hline
\multirow{2}{*}{$\Q(\sqrt{-1})$}
& y_{1}^{2} = x_{1}^{3} - x_{1} 
& \multirow{2}{*}{$Y^{2} = X^3-33X+t_{n}^{n}+\dfrac{8}{t_{n}^{n}}$}
& \multirow{2}{*}{$\renewcommand{\arraystretch}{1}\setlength{\arraycolsep}{4pt}\begin{pmatrix} 4 & 0 \\ 0 & 4\end{pmatrix}$} 
\\[-4pt]
& y_{2}^{2} = x_{2}^{3} - 11 x_{2} - 14 & & 
\\\hline
\multirow{2}{*}{$\Q(\sqrt{-7})$}
& y_{1}^{2} = x_{1}^3-21 x_{1}^2 + 112x_{1}
& \multirow{2}{*}{$Y^{2} = X^3 -1275 X + 64t_{n}^{n} -21546 -\dfrac{64}{t_{n}^{n}}$}
& \multirow{2}{*}{$\renewcommand{\arraystretch}{1}\setlength{\arraycolsep}{4pt}\begin{pmatrix} 4 & 2 \\ 2 & 8\end{pmatrix}$} 
\\[-4pt]
& y_{2}^{2} = x_{2}^{3} + 42x_{2}^2 -7x_{2} & &
\\\hline
\end{array}
\]
\end{footnotesize}
\end{theorem}

\begin{proof}
Elliptic curves belonging to the same $K$ in the previous table are
isogenous to each other with complex multiplication in some order in
$K$.  They are not isomorphic since $j$-invariants are different.
Thus, the Mordell-Weil rank of $F^{(5)}$ and $F^{(6)}$ are~$18$.  Note
that the choices of $E_{1}$ and $E_{2}$ are made so that the field of
definition of isogeny is as small as possible.
\end{proof}

\subsection{Class number 2 case}

It is also known that there are only finitely many negative
discriminants $D$ whose class number equals~$2$ (see for example
\cite[pp.~358--361]{Ireland-Rosen}).  Table~\ref{tbl:j(tau_1)} at the
end shows twenty-nine such $D$, together with $j(\tau_{1})$.
Table~\ref{tbl:EandF} shows an example of $E_{1}$ for
each~$j(\tau_{1})$, together with $F^{(n)}_{E_{1},E_{2}}$ where
$E_{2}$ is the Galois conjugate of $E_{1}$\footnote{Equations have
  also been given in \cite{Robert}.}. In these equations $X$ and $Y$
are rescaled so that the coefficients become simpler.  However, $t$ is
the original parameter $t_{n}$, and thus some equation contains a
variable $\varepsilon$, which indicates the fundamental unit of the
real quadratic field $\Q(j(\tau_{1}))$.  By rescaling the elliptic
parameter $t$ suitably as in Lemma~\ref{lem:quadratic}, we obtain
elliptic $K3$ surfaces defined over~$\Q$.

\begin{theorem}\label{th:rank18}
The twenty-nine Inose surfaces shown in Table~\ref{tbl:EandF} are
defined over $\Q$ and geometrically non-isomorphic. The elliptic
fibrations $F^{(5)}$ and $F^{(6)}$ constructed from them have
Mordell-Weil rank~$18$ over~$\Qbar$.
\end{theorem}

\begin{remark}
As remarked earlier, the equation of the Inose surface may be given in
the form
\[
Y^2=X^3-3\root 3\of {J_{1}J_{2}}\,X
+t+\frac{1}{t}-2\sqrt{(1-J_{1})(1-J_{2})},
\]
where $J_{i}=j(\tau_{i})/1728$.  Quite often $\root
3\of{j(\tau_{1})j(\tau_{2})}$ becomes a rational integer (see
\cite[\S12]{Cox-primes}), and so does
$\sqrt{(1728-j(\tau_{1}))(1728-j(\tau_{2}))}$.  In that case the above
equation gives a model of the Inose surface over~$\Q$.  However, since
that is not always the case, we uniformly started from \eqref{eq:F6}.
By doing so, we can also keep track of the field of definition of
$F^{(n)}_{E_{1},E_{2}}(\Qbar(t_{n}))$.
  
\end{remark}

\begin{table}[h]
\caption{Discriminants with class number $2$.}
\label{tbl:j(tau_1)}
\begin{sideways}
\begin{minipage}{\textheight}
\begin{footnotesize}
\begin{tabular}{ 
>{$}l<{$}   >{$}c<{$}   >{$}c<{$}   >{$}c<{$} >{$}c<{$} >{$}c<{$} >{$}r<{$} >{$}c<{$}}
\toprule
\hfil D &	 K &	 d_K &	 f &	 \textrm{quadratic form} &	 H_{D} &	j(\tau_{1})	\\
\midrule							
\\							
-15 = -3 \times 5 &	 \Q(\sqrt{-15}) &	-15&	1&	 2x^{2}+xy+2y^{2}  &	 K(\sqrt{5}) &	-191025/2+85995 \sqrt{5}/2	\\
-20 = -2^2 \times 5 &	 \Q(\sqrt{-5 })  &	-20&	1&	 2x^{2}+2xy+3y^{2} &	 K(\sqrt{5}) &	632000-282880 \sqrt{5}	\\
-24 = -2^3 \times 3 &	 \Q(\sqrt{-6 })  &	-24&	1&	 2x^{2}+3y^{2} &	 K(\sqrt{2}) &	2417472-1707264 \sqrt{2}	\\
-32 = -2^5&	 \Q(\sqrt{-2 })  &	-8&	2&	 3x^{2}+2xy+3y^{2} &	 K(\sqrt{2}) &	26125000 - 18473000 \sqrt{2}	\\
-35 = -5 \times 7 &	 \Q(\sqrt{-35 })  &	-35&	1&	 3x^{2}+xy+3y^{2} &	 K(\sqrt{5}) &	-58982400 + 26378240 \sqrt{5}	\\
-36 = -2^2 \times 3^2 &	 \Q(\sqrt{-1 })  &	-4&	3&	 2x^{2}+2xy+5y^{2} &	 K(\sqrt{3}) &	76771008-44330496 \sqrt{3}	\\
-40 = -2^3 \times 5 &	 \Q(\sqrt{-10 })  &	-40&	1&	 2x^{2}+5y^{2} &	 K(\sqrt{5}) &	212846400 - 95178240 \sqrt{5}	\\
-48 = -2^4 \times 3 &	 \Q(\sqrt{-3 })  &	-3&	4&	 3x^{2}+4y^{2} &	 K(\sqrt{3}) &	1417905000-818626500 \sqrt{3}	\\
-51 = -3 \times 17 &	 \Q(\sqrt{-51 })  &	-51&	1&	 3x^{2}+3xy+5y^{2} &	 K(\sqrt{17}) &	-2770550784+671956992 \sqrt{17}	\\
-52 = -2^2 \times 13 &	 \Q(\sqrt{-13 })  &	-52&	1&	 2x^{2}+2xy+7y^{2} &	 K(\sqrt{13}) &	3448440000-956448000 \sqrt{13}	\\
-60 = -2^2 \times 3 \times 5 &	 \Q(\sqrt{-15 })  &	-15&	2&	 3x^{2}+5y^{2} &	 K(\sqrt{5}) &	37018076625/2 - 16554983445 \sqrt{5}/2	\\
-64 = -2^6&	 \Q(\sqrt{-1 })  &	-4&	4&	 4x^{2}+4xy+5y^{2} &	 K(\sqrt{2}) &	41113158120-29071392966 \sqrt{2}	\\
-72 = -2^3 \times 3^2  &	 \Q(\sqrt{-2 })  &	-8&	3&	 2x^{2}+9y^{2} &	 K(\sqrt{6}) &	188837384000+77092288000 \sqrt{6}	\\
-75 = -3 \times 5^2 &	 \Q(\sqrt{-3 })  &	-3&	5&	 3x^{2}+3xy+7y^{2} &	 K(\sqrt{5}) &	-327201914880 + 146329141248 \sqrt{5}	\\
-88 = -2^3 \times 11 &	 \Q(\sqrt{-22 })  &	-88&	1&	 2x^{2}+11y^{2} &	 K(\sqrt{2}) &	3147421320000-2225561184000 \sqrt{2}	\\
-91 = -7 \times 11 &	 \Q(\sqrt{-91 })  &	-91&	1&	 5x^{2}+3xy+5y^{2} &	 K(\sqrt{13}) &	-5179536506880+1436544958464 \sqrt{13}	\\
-99 = -3^2 \times 11 &	 \Q(\sqrt{-11 })  &	-11&	3&	 5x^{2}+xy+5y^{2} &	 K(\sqrt{33}) &	-18808030478336+3274057859072 \sqrt{33}	\\
-100 = -2^2 \times 5^2 &	 \Q(\sqrt{-1 })  &	-4&	5&	 2x^{2}+2xy+13y^{2} &	 K(\sqrt{5}) &	22015749613248-9845745509376 \sqrt{5}	\\
-112 = -2^4 \times 7 &	 \Q(\sqrt{-7 })  &	-7&	4&	 4x^{2}+7y^{2} &	 K(\sqrt{7}) &	137458661985000-51954490735875 \sqrt{7}	\\
-115 = -5 \times 23 &	 \Q(\sqrt{-115 })  &	-115&	1&	 5x^{2}+5xy+7y^{2} &	 K(\sqrt{5}) &	-213932305612800+95673435586560 \sqrt{5}	\\
-123 = -3 \times 41 &	 \Q(\sqrt{-123 })  &	-123&	1&	 3x^{2}+3xy+11y^{2} &	 K(\sqrt{41}) &	-677073420288000+105741103104000 \sqrt{41}	\\
-147 = -3^2 \times 7  &	 \Q(\sqrt{-7 })  &	-7&	3&	 3x^{2}+3xy+13y^{2} &	 K(\sqrt{21}) &	-17424252776448000 + 3802283679744000 \sqrt{21}	\\
-148 = -2^2 \times 37 &	 \Q(\sqrt{-37 })  &	-148&	1&	 2x^{2}+2xy+19y^{2} &	 K(\sqrt{37}) &	19830091900536000-3260047059360000 \sqrt{37}	\\
-187 = -11 \times 17 &	 \Q(\sqrt{-187 })  &	-187&	1&	 7x^{2}+3xy+7y^{2} &	 K(\sqrt{17}) &	-2272668190894080000+551203000178688000 \sqrt{17}	\\
-232 = -2^3 \times 29 &	 \Q(\sqrt{-58 })  &	-232&	1&	 2x^{2}+29y^{2} &	 K(\sqrt{29}) &	302364978924945672000-56147767009798464000 \sqrt{29}	\\
-235 = -5 \times 47 &	 \Q(\sqrt{-235 })  &	-235&	1&	 5x^{2}+5xy+13y^{2} &	 K(\sqrt{5}) &	-411588709724712960000 + 184068066743177379840 \sqrt{5}	\\
-267 = -3 \times 89 &	 \Q(\sqrt{-267 })  &	-267&	1&	 3x^{2}+3xy+23y^{2} &	 K(\sqrt{89}) &	-9841545927039744000000 + 1043201781864732672000 \sqrt{89}	\\
-403 = -13 \times 31 &	 \Q(\sqrt{-403 })  &	-403&	1&	 11x^{2}+9xy+11y^{2} &	 K(\sqrt{31}) &	-1226405694614665695989760000+340143739727246741938176000 \sqrt{13}	\\
-427 = -7 \times 61 &	 \Q(\sqrt{-427 })  &	-427&	1&	 7x^{2}+7xy+17y^{2} &	 K(\sqrt{61}) &	-7805727756261891959906304000 + 999421027517377348595712000 \sqrt{61}	\\
\bottomrule
\end{tabular}
\end{footnotesize}
\end{minipage}
\end{sideways}
\end{table}

\clearpage

\begin{table}[h]
\caption{$\Q$-curves over quadratic fields and associated K3 surfaces.}
\label{tbl:EandF}
\begin{sideways}
\begin{minipage}{\textheight}
\begin{footnotesize}
\begin{tabular}{ >{$}l<{$}   >{$}l<{$}   >{$}l<{$}   >{$}c<{$} }
\toprule
D &	\hfil E_{1}&	\hfil F^{(n)}_{E_{1},E_{2}} &		\\		
\midrule 	
\\		
-15&	y^2 = x^3-3 (-3+2 \sqrt{5}) x^2-24 (-3+\sqrt{5}) x&	Y^2 = X^3+1485 X -1728 \varepsilon^2 t^{n}-29106-1728/(\varepsilon^2 t^{n})&	\\				
-20&	y^2 = x^3-12 x^2+9 (\sqrt{5}+2) x&	Y^2 = X^3-1485 X -729 \varepsilon^3 t^{n}-45144+729/(\varepsilon^3 t^{n})&		\\			
-24&	y^2 = x^3-12 x^2+6 (3+2 \sqrt{2}) x&	Y^2 = X^3-459 X -27 \varepsilon^2 t^{n}+2484-27/(\varepsilon^2 t^{n})&		\\			
-32&	y^2 = x^3-6 (-3+\sqrt{2}) x^2+9 (3+2 \sqrt{2}) x&	Y^2 = X^3+15525 X -5832 \varepsilon^3 t^{n}-1886598+5832/(\varepsilon^3 t^{n})&	\\				
-35&	y^2 = x^3+84 (-15+7 \sqrt{5}) x-98 (115 \sqrt{5}-256)&	Y^2 = X^3+423360 X -250047 t^{n}-76366206-250047/t^{n}&	\\				
-36&	y^2 = x^3+12 (-1+\sqrt{3}) x^2+(3 (3+2 \sqrt{3})) x&	Y^2 = X^3-6831 X -81 \sqrt{3} t^{n}/\varepsilon+232848+81 \varepsilon \sqrt{3}/t^{n}&	\\				
-40&	y^2 = x^3-6 (\sqrt{5}+3) x^2+(55 \sqrt{5}+123) x&	Y^2 = X^3-435 X -t^{n}+3348-1/t^{n}&		\\			
-48&	y^2 = x^3-6 (3 \sqrt{3}-7) x^2+6 (2+\sqrt{3}) x&	Y^2 = X^3-9900 X +128 t^{n}+190256+128/t^{n}&	\\				
-51&	y^2 = x^3+12 (-5+\sqrt{17}) x+14 (-15+4 \sqrt{17})&	Y^2 = X^3-384 X +\varepsilon^2 t^{n}+4606+1/(\varepsilon^2 t^{n})&	\\				
-52&	y^2 = x^3-12 x^2+(18+5 \sqrt{13}) x&	Y^2 = X^3-1725 X -\varepsilon^3 t^{n}-27864+1/(\varepsilon^3 t^{n})&\\			
-60&	y^2 = x^3-6 (5 \sqrt{5}-13) x^2-6 (7+3 \sqrt{5}) x&	Y^2 = X^3-17835 X +32 \varepsilon^4 t^{n}-699622+32/(\varepsilon^4 t^{n})&	\\				
-64&	y^2 = x^3-(6 (3 \sqrt{2}-5)) x^2+(3+2 \sqrt{2}) x&	Y^2 = X^3-3243 X -16 \sqrt{2} t^{n}-320166+16 \sqrt{2}/t^{n}&	\\				
-72&	y^2 = x^3-(42 (2+\sqrt{6})) x^2-9 (-5+2 \sqrt{6}) x&	Y^2 = X^3-115275 X -27 \varepsilon t^{n}-15043196-27/(\varepsilon t^{n})&	\\				
-75&	y^2 = x^3+12 (69 \sqrt{5}-155) x-2 (-21760+9729 \sqrt{5})&	Y^2 = X^3-10560 X +5 t^{n}-460790+5/t^{n}&\\			
-88&	y^2 = x^3-132 (-1+\sqrt{2}) x^2+22 (17+12 \sqrt{2}) x&	Y^2 = X^3-52275 X -t^{n}-4598748-1/t^{n}&	\\			
-91&	y^2 = x^3+28 (-227+63 \sqrt{13}) x+1078 (-256+71 \sqrt{13})&	Y^2 = X^3+3264 X +t^{n}-137214+1/t^{n}&\\			
-99&	y^2 = x^3+108 (-3751+653 \sqrt{33}) x+7182 (-19543+3402 \sqrt{33})&	Y^2 = X^3+71808 X +27 t^{n}/\varepsilon+2936374+27 \varepsilon/t^{n}&	\\				
-100&	y^2 = x^3-(12 (3 \sqrt{5}-5)) x^2+(1/2) (105+47 \sqrt{5}) x&	Y^2 = X^3-691185 X -5 \sqrt{5} t^{n}+221205600+5 \sqrt{5}/t^{n}&	\\				
-112&	y^2 = x^3+6 (\sqrt{7}-8) x^2+(127+48 \sqrt{7}) x&	Y^2 = X^3-367275 X +64 \varepsilon^3 t^{n}-68796378+64/(\varepsilon^3 t^{n})&	\\				
-115&	y^2 = x^3+92 (-785+351 \sqrt{5}) x+1058 (-9984+4465 \sqrt{5})&	Y^2 = X^3-10560 X +t^{n}+1079298+1/t^{n}&	\\			
-123&	y^2 = x^3+480 (-8+\sqrt{41}) x-112 (-951+160 \sqrt{41})&	Y^2 = X^3-110400 X +\varepsilon^2 t^{n}+14229502+1/(\varepsilon^2 t^{n})&	\\				
-147&	y^2 = x^3+360 (142 \sqrt{21}-651) x-66 (-935424+204125 \sqrt{21})&	Y^2 = X^3-1713600 X -7 t^{n}+865686514-7/t^{n}&	\\				
-148&	y^2 = x^3-84 x^2+(882+145 \sqrt{37}) x&	Y^2 = X^3-4148925 X -\varepsilon^3 t^{n}-3252770136+1/(\varepsilon^3 t^{n})&	\\				
-187&	y^2 = x^3+220 (-51+13 \sqrt{17}) x+242 (-2765+676 \sqrt{17})&	Y^2 = X^3+326400 X +t^{n}+73279998+1/t^{n}&	\\			
-232&	y^2 = x^3-198 (\sqrt{29}-5) x^2+(135 \sqrt{29}+727) x&	Y^2 = X^3-512317875 X -t^{n}-4463313183252-1/t^{n}&	\\				
-235&	y^2 = x^3+2068 (3969 \sqrt{5}-8875) x+92778 (-461312+206305 \sqrt{5})&	Y^2 = X^3-4762560 X +t^{n}+4231914498+1/t^{n}&	\\				
-267&	y^2 = x^3+60 (-625+53 \sqrt{89}) x-14 (-232143+26500 \sqrt{89})&	Y^2 = X^3-168748800 X +\varepsilon^2 t^{n}+843767999998+1/(\varepsilon^2 t^{n})&	\\				
-403&	y^2 = x^3+5580 (-2809615+779247 \sqrt{13}) x-363258 (-2941504000+815826423 \sqrt{13})&	Y^2 = X^3+99470400 X +t^{n}+2352019840002+1/t^{n}&	\\				
-427&	y^2 = x^3+3080 (-236674+30303 \sqrt{61}) x-2254 (608549875 \sqrt{61}-4752926464)&	Y^2 = X^3-1119201600 X +t^{n}-15615066773502+1/t^{n}&	\\\bottomrule
\end{tabular}
\end{footnotesize}
\end{minipage}
\end{sideways}
\end{table}
\clearpage

\section{Examples} \label{sec:examples}

In this section, we illustrate the techniques of this paper by working
out the Mordell-Weil group explicitly for a few singular K3 surfaces
defined over $\Q$. We will use the set-up of the previous section,
choosing a few small discriminants. The basic idea is to obtain a
finite index sublattice of the Mordell-Weil lattice by combining the
formulas in the generic case (when the elliptic curves are
non-isogenous) with the extra sections coming from isogenies between
the pair of CM curves. To obtain the full Mordell-Weil group, we
saturate this sublattice. In practice, this will be most convenient
for the surface $F^{(6)}$, for which the Mordell-Weil lattice has
sublattices induced from $F^{(3)}$ and its twist, enabling us to
proceed in stages. One can also apply our methods to $F^{(5)}$, which
has the disadvantage that saturating the corresponding sublattice is
computationally more expensive (but that may be offset by the fact
that it is somewhat easier to calculate sections directly, owing to
convenient specialization maps at $t = 0$ and $\infty$). For
simplicity we restrict ourselves to $F^{(6)}$ here.

\begin{example}\label{eg:disc-16}
(cf. \cite[Example~4.5]{Kuwata:MW-rank})
Let $E_{1}$ and $E_{2}$ be given by
\begin{align*}
& E_{1}:y_{1}^{2} = x_{1}^{3} - x_{1},\\
& E_{2}:y_{2}^{2} = x_{2}^{3} - 11 x_{2} - 14.
\end{align*}
In this case $F^{(n)}_{E_{1},E_{2}}$ is given by
\[
F^{(n)}_{E_{1},E_{2}}:Y^{2} = X^3-33X+t_{n}^{n}+\frac{8}{t_{n}^{n}}.
\]
The matrix of the quadratic form associated with the Inose surface
$F^{(1)}_{E_1, E_2}$ is
\[
\begin{pmatrix}
4 & 0 \\
0 & 4
\end{pmatrix},
\]
so we first identify two $2$-isogenies from $E_1$ to $E_2$. Then using
the method described in \S3.2, we obtain a basis of
$F^{(1)}(\Qbar(t_{1}))$:
\begin{align*}
P_{1} & =\left(\frac{1}{144} ( s_{1,+}^2 
-72 s_{1,+}+ 400),
\frac{1}{1728}
s_{1,-}(s_{1,+}^{2}-108 s_{1,+} + 2560 \right),
\\
P_{2} &=\left(-\frac{1}{144}  (s_{1,+}^{2}
+ 72 s_{1,+} + 400),
\frac{i}{1728} s_{1,-}
( s_{1,+}^{2} + 108 s_{1,+}+ 2560) \right).
\end{align*}
where $s_{1,+} = t_1 + 8/t_1$ and $s_{1,-} = t_1 - 8/t_1$. Note that
we have to extend the base field to $k = \Q(i)$ to define the
isogenies.

Next, we study $F^{(3)}$ in this example.  The images of the above
sections in $F^{(3)}(\Qbar(t_{3}))$ are of height~$6$. The splitting
field of the $3$-torsion points of $E_{1}$ and $E_{2}$ is
\[
k(E_{1}[3],E_{2}[3])=\Q(i, \omega, 12^{1/4}), \quad \text{ where } \omega=(-1+\sqrt{-3})/2.
\]
Let $\alpha = \sqrt{3} = -i(2\omega+1)$, $\beta = \sqrt{2}$ and
$\gamma = 12^{1/4}$. The sections described in Theorem~\ref{th:F^(3)}
are given by
\begin{align*}
&P_{3} = \Bigl(-s_{3,+}-9, 3 i \alpha (s_{3,+}+4)\Bigr), \qquad
P_{4} = \Bigl(-s_{3,+}+9, -3\alpha (s_{3,+}-4)\Bigr),
\\
&P_{5} = \Bigl(-s_{3,+}+ i \alpha, \gamma \alpha (1+i)/2 (-s_{3,+}+2i \alpha) \Bigr),
\\
&\begin{aligned}
&P_{6} = 
\Bigl(\omega\big(-s_{3,+} + 3 \omega (1-i) \gamma/2 + 2\alpha + 3 i \big), 
\\
&\qquad\qquad
 \big( (\omega -1)i\gamma/2 + 3(\omega-1)(i+1)/2 \big) s_{3,+} + (7i\alpha -3)\gamma /2 + 3(i+1)\alpha \Bigr),
\end{aligned}
\\
&\sigma^{2}(P_{6}), \quad \sigma^{2}(P_{5}), \quad
\sigma^{2}(P_{4}), \quad \sigma^{2}(P_{3}),
\end{align*}
where $s_{3,+} = t_{3} + 2/t_{3}$ and $s_{3,-} = t_{3} - 2/t_{3}$.  The height paring matrix
with respect to the above eight sections coincides with the one in
Theorem~\eqref{th:F^(3)}.  Knowing that there are sections of height
smaller than~$6$ independent of the above sections, we search for
sections and find the following:
\begin{gather*}
P_{7}=\left(3-s_{3,+}, -3 s_{3,-}\right)
\quad
P_{8}=\left(-3-s_{3,+}, -3is_{3,-} \right)
\end{gather*}
With respect to the basis
$P_{3},P_{4},P_{5},P_{6},\sigma^{2}(P_{6})\sigma^{2}(P_{5}),
\sigma^{2}(P_{4}), \sigma^{2}(P_{3}), P_{7},P_{8}$ the height matrix
is
\[\renewcommand{\arraystretch}{1}
\left( 
\begin{array}{*{10}r} 
4&0&0&-2&1&0&0&-2&0&0\\  
0&4&0&-2&1&0&-2&0&0&0\\  
0&0&4&-2&1&-2&0&0&0&0\\  
-2&-2&-2&4&-2&1&1&1&0&0\\  
1&1&1&-2&4&-2&-2&-2&-1&-1\\  0&0&-2&1&-2&4&0&0&0&0\\  
0&-2&0&1&-2&0&4&0&0&2\\  
-2&0&0&1&-2&0&0&4&2&0\\  
0&0&0&0&-1&0&0&2&4&0\\  
0&0&0&0&-1&0&2&0&0&4\end {array} \right).
\]
Note that the Mordell-Weil lattice $F^{(3)}(\bar k(t_{3}))$ is generated
by sections of height~$4$ in this case.

The rational elliptic surface ${}^{t_{3}}F^{(3)}_{E_{1},E_{2}}$ is
\[
{}^{t_{3}}F^{(3)}_{E_{1},E_{2}}:t_{3}Y^2 = X^3-33 X+(t_{3}^3+8/t_{3}^3)
\]
In this case $\Q(E_{1}[2])=\Q$ and $\Q(E_{2}[2])=\Q(\sqrt{2})$, and
$\Delta_{E_{2}}/\Delta_{E_{1}}=2^{9}/2^{6}=(\sqrt{2})^{6}$.  So, the
field of definition of ${}^{t_{3}}F^{(3)}_{E_{1},E_{2}}(\Qbar(t_{3}))$ is
$\Q(\sqrt{2},\omega)$.  Following the recipe of Section
\ref{rationalF3twist}, we can obtain a basis of the Mordell-Weil
lattice.  A small modification of the basis there gives the following
simpler basis:
\begin{align*}
Q_1 &= \big( -\omega^2(11t_{3}^2 + 4\omega^2)/(2t_{3}), -21\beta(2\omega+1)t_{3}/4\big) \\
Q_2 &= \big( 2\omega(4-3\beta)t_{3} -3(1-2\beta) + 4\omega^2/t_{3}, \\
& \qquad 3(11-8\beta)t_{3} - 12\omega^2(3-2\beta) + 6\omega(4-\beta)/t_{3} + 6\beta/t_{3}^2 \big) \\
Q_3 &= \big( (2t_{3}^2 - 6t_{3} + 1)/t_{3}, -3(t_{3}^3-4t_{3}^2+t_{3}-1)/t_{3}^2 \big) \\
Q_4 &= \big( 2t_{3} - 3(1-2\beta) + 4(4-3\beta)/t_{3}, \\
& \qquad -3t_{3} + 6(1-2\beta) - 12(4-3\beta)/t_{3} + (96-66\beta)/t_{3}^2\big) \\
Q_5 &= \big( -\omega t_{3} + 4\omega^2(4-3\beta)/t_{3}, -3\omega^2(1-2\beta) + (96-66\beta)/t_{3}^2\big) \\
Q_6 &= \big( 2t_{3} + 3(1-2\beta) + 4(4-3\beta)/t_{3}, 3t_{3} + 6(1-2\beta) + 12(4-3\beta)/t_{3} + (96-66\beta)/t_{3}^2\big) \\
Q_7 &= \big( (2t_{3}^2 + 6t_{3} + 1)/t_{3}, 3(t_{3}^3 + 4t_{3}^2 + t_{3} + 1)/t_{3}^2 \big) \\
Q_8 &= \big( -t_{3} - 11/t_{3}, 21(1 + 2\omega)/t_{3}^2\big) \\
\end{align*}
which gives the height pairing matrix
\[
\left( 
\begin{array}{*{8}r} 
2&-1&0&0&0&0&0&0 \\
-1&2&-1&0&0&0&0&0\\
0&-1&2&-1&0&0&0&0\\
0&0&-1&2&-1&0&0&0\\
0&0&0&-1&2&-1&0&-1\\
0&0&0&0&-1&2&-1&0\\
0&0&0&0&0&-1&2&0\\
0&0&0&0&-1&0&0&2
\end{array}
\right).
\]
In order to find a basis for 
\[
F^{(6)}_{E_{1},E_{2}}:Y^2 = X^3-33t_{6}^4X+(t_{6}^{12}+8),
\]
we have to fill the gap between $L' = F^{(3)}_{E_{1},E_{2}}(\bar
k(t_{3}))\oplus {}^{t_{3}}F^{(3)}_{E_{1},E_{2}}(\bar k(t_{3}))$ and $L
= F^{(6)}_{E_{1},E_{2}}(\bar k(t_{6}))$. The lattices differ by a
power-of-two index. In principle, this is a routine calculation; we
simply have to check whether non-trivial coset of $L'$ modulo $2L'$ is
twice an element of the Mordell-Weil group $L$, and this boils down to
checking whether a suitable equation has a root in $\bar k
(t_{6})$. However, there are $2^{18} - 1$ such cosets, and the number
field involved is quite large, so we need to speed up the process by
reducing the large number of possible candidate cosets.

To do so, we apply the following useful procedure (a similar trick was
used in \cite{Shioda:2007}). First reduce modulo a suitable prime $p$
such that all the $6$-torsion points are defined over $\F_{p}$.  In
our case $p=193$ will do, as $x^4+1$ and $x^4-3$ split into linear
factors. Since arithmetic modulo $p$ is cheap, we can do the search
mentioned in the previous paragraph fairly quickly. This pins down the
likely candidates, and now we may solve the equation back in the
original number field $k$ for each of them.  We find new sections
$R_1, \dots, R_5$ which satisfy the relations

\begin{align*}
2R_1 &= P_3 - P_7 - Q_3 - 2Q_4 - 2Q_5 - Q_8 \\
2R_2 &= -P_3 - P_5 - 2P_6 - \sigma^2(P_3) + \sigma^2(P_4) + P_7 - P_8 \\
& \qquad + Q_3 + 2Q_4 + 2Q_5 + 2Q_6 + Q_7 + Q_8 \\
2R_3 &= -P_4 - P_5 - \sigma^2(P_4) - \sigma^2(P_5) - P_7 + P_8 + Q_3 + Q_4 + Q_5 + Q_6 + Q_7 \\
2R_4 &= P_3 - P_4 -Q_1 - Q_2 - 2Q_3 - 3Q_4 - 4Q_5 - 4Q_6 - 2Q_7 -2Q_8 \\
2R_5 &= \sigma^2(P_3) - \sigma^2(P_4) + Q_1 + 2Q_2 + 2Q_3 + 3Q_4 + 3Q_5 + Q_6 + Q_8.
\end{align*}
Recall that $\sigma = \sigma_{6}$ is the map
$\big(X(t_{6}),Y(t_{6})\big)\mapsto
\big(X(t_{6}/\zeta_{6}),Y(t_{6}/\zeta_{6})\big)$. The expressions for the $x$-
and $y$-coordinates of these new sections may be found in the
auxiliary files; apart from $R_1$, which is shown below, they are
quite complicated.

\[
R_{1}=\bigl(-\omega t_{6}^2+3+6 \omega\beta/t_{6} +4\omega^2/t_{6}^2, -3\omega t_{6}^2 -3 \omega^2\beta t_{6} + 18 \omega \beta/t_{6} + 24 \omega^2/t_{6}^2 + 6 \beta/t_{6}^3  \bigr),
\]

The sections
\begin{gather*}
P_{4},P_{5},P_{6},\sigma^{2}(P_{6}),\sigma^{2}(P_{5}), P_{7},P_{8},
Q_{1},Q_{3},Q_{4}, Q_{6}, Q_{7}, R_{1},R_{2},\sigma^{-1}(R_{2}),R_{3},R_{4},R_{5},
\end{gather*}
form a basis of the Mordell-Weil lattice $F^{(6)}(\Qbar(t_{6}))$. The height
paring matrix with respect to this basis is

\begin{small}
\[
\left(\begin {array}{*{18}r} 
 4 & 0 & -2 & 1 & 0 & 0 & 0 & 0 & 0 & 0 & 0 & 0 & 0 & 1 & -1 & -1 & -2 & 1 \\
 0 & 4 & -2 & 1 & -2 & 0 & 0 & 0 & 0 & 0 & 0 & 0 & 0 & 0 & 0 & -1 & 0 & 0 \\
 -2 & -2 & 4 & -2 & 1 & 0 & 0 & 0 & 0 & 0 & 0 & 0 & -1 & -2 & 1 & 1 & 0 & 0 \\
 1 & 1 & -2 & 4 & -2 & -1 & -1 & 0 & 0 & 0 & 0 & 0 & 1 & 1 & -2 & 1 & 0 & 0 \\
 0 & -2 & 1 & -2 & 4 & 0 & 0 & 0 & 0 & 0 & 0 & 0 & 0 & 0 & 0 & -1 & 0 & 0 \\
 0 & 0 & 0 & -1 & 0 & 4 & 0 & 0 & 0 & 0 & 0 & 0 & -2 & 1 & 2 & -2 & 0 & 1 \\
 0 & 0 & 0 & -1 & 0 & 0 & 4 & 0 & 0 & 0 & 0 & 0 & 0 & -1 & -1 & 1 & 0 & -1 \\
 0 & 0 & 0 & 0 & 0 & 0 & 0 & 4 & 0 & 0 & 0 & 0 & 0 & 0 & 0 & 0 & -1 & 0 \\
 0 & 0 & 0 & 0 & 0 & 0 & 0 & 0 & 4 & -2 & 0 & 0 & 0 & 0 & -1 & 1 & 0 & -1 \\
 0 & 0 & 0 & 0 & 0 & 0 & 0 & 0 & -2 & 4 & 0 & 0 & -1 & 1 & 0 & 0 & 0 & 1 \\
 0 & 0 & 0 & 0 & 0 & 0 & 0 & 0 & 0 & 0 & 4 & -2 & 2 & 1 & 0 & 0 & -2 & -1 \\
 0 & 0 & 0 & 0 & 0 & 0 & 0 & 0 & 0 & 0 & -2 & 4 & 0 & 0 & -1 & 1 & 0 & -1 \\
 0 & 0 & -1 & 1 & 0 & -2 & 0 & 0 & 0 & -1 & 2 & 0 & 4 & 0 & -1 & 1 & 0 & -2 \\
 1 & 0 & -2 & 1 & 0 & 1 & -1 & 0 & 0 & 1 & 1 & 0 & 0 & 4 & 0 & -1 & -1 & -1 \\
 -1 & 0 & 1 & -2 & 0 & 2 & -1 & 0 & -1 & 0 & 0 & -1 & -1 & 0 & 4 & -2 & 1 & 1 \\
 -1 & -1 & 1 & 1 & -1 & -2 & 1 & 0 & 1 & 0 & 0 & 1 & 1 & -1 & -2 & 4 & 0 & -1 \\
 -2 & 0 & 0 & 0 & 0 & 0 & 0 & -1 & 0 & 0 & -2 & 0 & 0 & -1 & 1 & 0 & 4 & 0 \\
 1 & 0 & 0 & 0 & 0 & 1 & -1 & 0 & -1 & 1 & -1 & -1 & -2 & -1 & 1 & -1 & 0 & 4 \\
\end {array}\right).
\] 
\end{small}

Its determinant equals $576=2^{6}3^{2}$, as expected.  Note that this
Mordell-Weil lattice is generated by sections of height~$4$. The field
of definition for the Mordell-Weil lattice of $F^{(6)}$ is
$\Q(i,\omega,12^{1/4},\sqrt{2}) = \Q( \sqrt{-1}, \sqrt{2}, 3^{1/4})$.

\end{example}

\begin{example}\label{eg:disc-28}
Let $E_{1}$ and $E_{2}$ be given by
\begin{align*}
& E_{1}:y_{1}^{2} = x_{1}^{3} - 21 x_{1}^{2} + 112x_{1},\\
& E_{2}:y_{2}^{2} = x_{2}^{3} + 42 x_{2}^{2} - 7 x_{2}.
\end{align*}
They have complex multiplication by $\alpha = \sqrt{-7}$. The matrix of the
quadratic form associated with the Inose surface $F^{(1)}_{E_1, E_2}$ is
\[
\begin{pmatrix}
4 & 2 \\
2 & 8
\end{pmatrix},
\]
so there are $2$- and $4$-isogenies from $E_1$ to $E_2$. The
$2$-isogeny $\ph:E_{1}\to E_{2}$ is given by
\[
\ph(x_{1},y_{1}) = \left(\frac{x_{1}^2-21x_{1}+112}{x_{1}}, 
\frac{y_{1}(x_{1}^2-112)}{x_{1}^2}\right)
\]
Also, there is a $4$-isogeny $\psi:E_{1}\to E_{2}$ given by
$\psi(x_{1},y_{1}) = (x_{2},y_{2})$, where 
the $x$-coordinate is given by
\begin{align*}
&x_{2}=\frac{-(3+\alpha)x_{1}(2x_{1}-21-\alpha)(x_{1}-7-\alpha)^2}
{8(x_{1}-14+2\alpha)^2(2x_{1}-21+\alpha)}.
\end{align*}
For the $2$-isogeny $\ph$, the intersection \eqref{intersection}
consists of the image of $2$-torsion points and two effective divisors
$D^{\pm}$ of degree~$2$ defined over $\Q(t_{2})$.  Write
$D^{\pm}_{\ph}=P^{\pm}_{1}+P^{\pm}_{2}$, where $P^{\pm}_{1}$ and
$P_{2}^{\pm}$ are points on the cubic curve defined over
$\Qbar(t_{2})$.  Let $L^{\pm}$ be the line passing through
$P^{\pm}_{1}$ and $P^{\pm}_{2}$.  Then, each third point of
intersection $P^{\pm}_{3}$ is a $\Q(t_{2})$ rational point, given by
\[
(x^{\pm}_{1},x^{\pm}_{2}) = \left(\frac{84}{t_{2}^2\pm4t_{2}+8}, -\frac{21t_{2}^2}{t_{2}^2\pm4t_{2}+8}\right).
\]
For the $4$-isogeny $\psi$, the intersection \eqref{intersection}
consists of the image of $2$-torsion points and two effective divisors
$D^{\pm}_{\psi}$ of degree~$5$ defined over $\Q(\sqrt{-7})(t_{2})$.
Write $D^{\pm}_{\psi}=Q^{\pm}_{1}+\cdots+Q^{\pm}_{5}$, where
$Q^{\pm}_{1},\dots,Q^{\pm}_{5}$ are points on the cubic curve defined
over $\Qbar(t_{2})$.  There exists conics $C^\pm$ passing through
these five points $Q^{\pm}_{1},\dots,Q^{\pm}_{5}$.  Then, the sixth
point of intersection $Q^{\pm}_{6}$ of $C^\pm$ with the cubic curve is
a $\Q(\sqrt{-7})(t_{2})$-rational point.  The problem of finding the
coordinates $(x^{\pm}_{1},x^{\pm}_{2})$ of $Q^{\pm}_{6}$ can be
reduced to a linear algebra problem of determining the coefficients
of~$C^\pm$.  Namely, let $p(x_{1})$ be the quintic equation satisfied
by the $x_{1}$-coordinates of the five points $Q_i^\pm$, with
coefficients in $K = \Q(\sqrt{-7})(t_{2})$. We work in the field $L =
K[x_{1}]/(p(x_{1}))$. Then, we compute $x_{2}$ in terms of $x_{1}$,
which follows from the quartic equation $\psi_{y}(x_{1})=\pm t_{2}$,
and it is an element of $L$.  So, now if we make the $5$ by $6$ matrix
whose columns are the coordinates of
$1,x_{1},x_{2},x_{1}x_{2},x_{1}^2,x_{2}^2$ in terms of the basis
$(1,x_{1},x_{1}^2,x_{1}^3,x_{1}^4)$ of $L$ as a $K$-vector space, we
just need to take the kernel of this matrix (which has coefficients in
$K$).  The $1$-dimensional kernel gives us the coefficients of the
conic. From there, by taking resultants and factoring it, we obtain
the sixth point.  It is given by
$(x^{\pm}_{1},x^{\pm}_{2})=(x^{\pm}_{1,n}/d^{\pm},x^{\pm}_{2,n}/d^{\pm})$,
where
\begin{align*}
x^{\pm}_{1,n} &=
(\pm 2t+5+\alpha)\bigl((21-\alpha)t^4\pm2(91-15\alpha)t^3 \\
& \qquad  +2(259 -11\alpha)t^2\pm2(63+29\alpha)t-7+3\alpha\bigr), \\
x^{\pm}_{2,n} &= -3(\pm7t-\alpha)(\pm2t-5-\alpha)(\pm2t+5+\alpha)t^2,
\\
d^{\pm} &= 2\bigl(\pm2t^5+(23-\alpha)t^4\pm(117-19\alpha)t^3
+(109+29\alpha)t^2\mp2(1-9\alpha)t-5+\alpha\bigr).
\end{align*}
Over $\Q(\sqrt{-7})$, the Weierstrass equation of the elliptic curve $F^{(n)}_{E_{1},E_{2}}$ is given by
\[
F^{(n)}_{E_{1},E_{2}}: Y^{2} = X^3-1275 X  + 64\,t_{n}^{n} - 21546 - \frac{64}{t_{n}^{n}}.
\]
Using the construction of Proposition \ref{prop:sectionsfromisogeny},
the points obtained above yield points in $F^{(1)}(\bar k(t_{1}))$
\[
(X,Y) =\left(-\frac{1}{63}\Bigl(4 s_{1,-}^{2} -252 s_{1,-}+1339\Bigr), 
\frac{4}{1323}\alpha s_{1,+}
\Bigl(2 s_{1,-}^{2} -189 s_{1,-} +3977\Bigr)\right),
\]
and $(X,Y)=(X_{0}/d^{2},Y_{0}/d^{3})$, where
\begin{align*}
X_{0} &=
\Bigl(\frac{1-\alpha}{2}\Bigr)^{4} s_{1,-}^{4}
+16(33-4\alpha)\Bigl(\frac{1+\alpha}{2}\Bigr)^{4}
s_{1,-}^{3} + 36(1783+324\alpha) s_{1,-}^{2}
\\
& \qquad
+ 16(100923+3701\alpha) \Bigl(\frac{1+\alpha}{2}\Bigr)^{4} s_{1,-}
 + 64(12531+3413\alpha)\Bigl(\frac{1+\alpha}{2}\Bigr)^{6}, \\
Y_{0}&=
s_{1,+} \left(-\Bigl(\frac{1-\alpha}{2}\Bigr)^{6} s_{1,-}^{5}
-96(33-4\alpha)\Bigl(\frac{1+\alpha}{2}\Bigr)^{2} s_{1,-}^{4} +2^{5}(2531+1260\alpha)\Bigl(\frac{1+\alpha}{2}\Bigr)^{6}s_{1,-}^{3} \right.
\\
& \qquad  -2^{7}(18898-98053\alpha)
\Bigl(\frac{1+\alpha}{2}\Bigr) s_{1,-}^{2} 
-2^{10}(1735197-165636\alpha)\Bigl(\frac{1+\alpha}{2}\Bigr)^{2} s_{1,-} \\
& \qquad \left. -2^{14}(58959-85160\alpha)\Bigl(\frac{1+\alpha}{2}\Bigr)^{6} \right) \\
d &= 6\Bigl((11+4\alpha)s_{1,-}-4(1-11\alpha)\Bigr).
\end{align*}
Here we have used $s_{1,+} = t_{1} + 1/t_{1}$ and $s_{1,-} = t_{1} - 1/t_{1}$ to
condense the above formulae. These points form a basis of
$F^{(1)}(\bar k(t_{1}))$, and the height matrix is given by
\[
\left(\begin{array}{rr}
2 & -1 \\ -1 & 4
\end{array}\right).
\]
One can proceed to compute the Mordell-Weil group of $F^{(6)}$; this
is done in the auxiliary files. The field of definition is
$\Q\Big(\sqrt{-1},\sqrt{3},\sqrt{(3 + \sqrt{21})/2}\Big)$.
\end{example}

\begin{example}\label{eg:disc-15}
Discriminant $\Delta=-15$.

The class number of $\Delta=-15$ is $2$, and the non trivial quadratic
form is represented by
\[
\begin{pmatrix}  4 & 1 \\ 1 & 4 \end{pmatrix},
\quad \text{ or } \quad
2x^{2} + xy + 2 y^{2}.
\]
The Hilbert class field of $K=\Q(\sqrt{-15})$ equals $H=K(\sqrt{5})$.
The value of $j(\tau)$ for $\tau_{1}=(-1+\sqrt{-15})/4$ and
$\tau_{2}=(1+\sqrt{-15})/2$ are given by
\begin{align*}
j(\tau_{1})
&=\frac{-191025+85995\sqrt{5}}{2}
={\bar \eta}^{5}(3\pi_{5}\bar \pi_{11})^{3},
\\
j(\tau_{2})
&=\frac{-191025-85995\sqrt{5}}{2}
=-\eta^{5}(3\pi_{5}\pi_{11})^{3},
\end{align*}
where $\eta = (1+\sqrt{5})/2$ (the Golden ratio) is the fundamental
unit of $\Q(\sqrt{5})$, $\pi_{5}=\sqrt{5}$, $\pi_{11}=(1+3\sqrt{5})/2$
is the generator of a prime ideal above~$11$, and $\bar{\quad}$
indicates the conjugate $\sqrt{5}\mapsto -\sqrt{5}$.  We also have
\begin{align*}
j(\tau_{1}) - 1728 &=-3^{3}(\bar{\eta}^{3}7\pi_{11})^{2},\\
j(\tau_{2}) - 1728 &=-3^{3}(\eta^{3}7\bar \pi_{11})^{2}.
\end{align*}
We remark that in this case $j(\tau_{i})$ are not perfect cubes in
$H$.  While $j(\tau_{i})-1728$ are not perfect squares over $K$, they
are perfect squares in $H$ as $H$ contains
$\sqrt{-3}=\sqrt{-15}/\sqrt{5}$.

One of the elliptic curves whose $j$-invariant equals $j(\tau_{1})$ and $j(\tau_{2})$ are
\begin{align*}
&E_{1}: y^{2} = x^3-3\bigl(3-2\sqrt{5}\bigr)x^2+24\bigl(3-\sqrt{5}\bigr)x, \\
&E_{2}: y^{2} = x^3-3\bigl(3+2\sqrt{5}\bigr)x^2+24\bigl(3+\sqrt{5}\bigr)x.\\
\end{align*}
There are two $2$-isogenies between them, one defined over
$\Q(\sqrt{5})$ and the other defined over $H$.  So, $E_{1}$ and
$E_{2}$ are so-called $\Q$-curves. Over $H$, the Inose
surface can be transformed to
\[
F^{(1)}_{E_{1},E_{2}}:Y^{2} = X^{3} + 165 X +64 t_{1} +1078 + \frac{64}{t_{1}}.
\]

The Mordell-Weil lattice $F^{(6)}_{E_{1},E_{2}}(\Q(t_{1}))$ has rank~$18$,
and is defined over $H(\root 3\of \eta,E_{1}[6],E_{2}[6])$. The auxiliary
files contain an explicit basis. Note that $H(\root 3\of \eta
,E_{1}[6],E_{2}[6])=H(\root 3\of \eta) = \Q(\sqrt{-3}, \root 3\of \eta)$.
\end{example}

\begin{example}\label{eg:disc-20}
Discriminant $\Delta=-20$.

The class number of $\Delta=-20$ is $2$, and the non trivial quadratic
form is represented by
\[
\begin{pmatrix}  4 & 2 \\ 2 & 6 \end{pmatrix},
\quad \text{ or } \quad
2x^{2} + 2xy + 3y^{2}.
\]
The Hilbert class field of $K=\Q(\sqrt{-5})$ equals $H=K(\sqrt{5})$.
Using the same notation as in Example~\ref{eg:disc-15}, the value of
$j(\tau)$ for $\tau_{1}=(-1+\sqrt{-15})/4$ and
$\tau_{2}=(1+\sqrt{-15})/2$ are written as
\begin{align*}
j(\tau_{1})
&=632000-282880\sqrt{5}=
({\bar \eta}^{3}2^{2}\pi_{5}\bar \pi_{11})^{3},\\
j(\tau_{2})
&=632000+282880\sqrt{5}=
(\eta^{3}2^{2}\bar \pi_{5}\pi_{11})^{3},
\end{align*}
and we also have
\begin{align*}
j(\tau_{1}) - 1728
&=\eta^{3}(2^{4}\bar \pi_{11}\pi_{19})^{2},
\\
j(\tau_{2}) - 1728
&={\bar \eta}^{3}(2^{4}\pi_{11}\bar \pi_{19})^{2},
\end{align*}
where $\pi_{19}=1+2\sqrt{5}$.  As we can see in this example,
$j(\tau_{i})$ are perfect cubes, but $j(\tau_{i})-1728$ are not
perfect squares.

One of the elliptic curves whose $j$-invariant equals $j(\tau_{1})$ is
given by
\[
E_{1}:y^2 = x^3-4 x^2+\bigl(2+\sqrt{5}\bigr) x.
\]
This elliptic curve is $2$- and $3$-isogenous to its Galois conjugate
\[
E_{2}:y^2 = x^3-4 x^2+\bigl(2-\sqrt{5}\bigr) x,
\]
not over the field $\Q(\sqrt{5},\sqrt{-1})$ but over
$\Q(\sqrt{5},\sqrt{-1},\sqrt{2})$.  In fact, no twist of $E_{1}$ over
$\Q(\sqrt{5})$ is isogenous over $\Q(\sqrt{5},\sqrt{-1})$ to the
corresponding twist of $E_{2}$.  So, $E_{1}$ is not a $\Q$-curve in a
narrow sense.  Nevertheless, the Inose surface $F^{(1)}_{E_1,E_2}$ is
isomorphic over $\Q(\sqrt{5})$ to
\[
Y^{2}= X^3- \frac{55}{3} X - t_{1} - \frac{1672}{27} + \frac{1}{t_{1}}.
\]
Considering the fact that the $E_{1}$ and $E_{2}$ are at the same time
$2$- and $3$-isogenous over $\Q(\sqrt{5},\sqrt{-1},\sqrt{2})$, we
expect that the splitting field of $6$-torsion points of $E_{1}$ and
$E_{2}$ is a relatively small extension of
$\Q(\sqrt{5},\sqrt{-1},\sqrt{2})$.  Indeed, on the modular curve
$X(6)$ we discussed in \S4, the elliptic curves corresponding to the
point $\bigl((\eta-1)(1+\sqrt{-1}),\big(
(1+2\sqrt{-1})\eta-(1+3\sqrt{-1}) \big)\sqrt{\eta}\bigr)$ and its
conjugate under $\sqrt{5}\mapsto -\sqrt{5}$ are given by
\begin{align*}
&E'_{1}: 3\bigl(2+3\sqrt{-1}+\sqrt{5}\sqrt{-1}\bigr)\,y^{2} 
= x^3-4 x^2+\bigl(2+\sqrt{5}\bigr)x,
\\
&E'_{2}: 3\bigl(2+3\sqrt{-1}-\sqrt{5}\sqrt{-1}\bigr)\,y^{2} 
= x^3-4 x^2+\bigl(2-\sqrt{5}\bigr)x.
\end{align*}
So, we see that all the $6$-torsion points of $E_{1}$ and $E_{2}$ are
defined over $\Q(\sqrt{-1},\sqrt{2},\sqrt{3}, \sqrt{\eta} )$, and the
Mordell-Weil lattice $F^{(6)}_{E_{1},E_{2}}(\Qbar(t_{6}))$ is defined over
this field. An explicit basis is given in the computer files.

\end{example}


\begin{thebibliography}{PTvL}
\def\MR#1{\relax}

\bibitem[AS-D]{Artin-Swinnerton-Dyer} M.~Artin and
  H.~P.~F.~Swinnerton-Dyer, \textit{The Shafarevich-Tate conjecture
    for pencils of elliptic curves on K3 surfaces\/},
  Invent.\ Math.\ {\bf 20} (1973), 249--266.

\bibitem[BT]{Bogomolov-Tschinkel} F.~A.~Bogomolov, and Y.~Tschinkel,
  \textit{Density of rational points on elliptic K3 surfaces\/}, Asian
  J.\ Math.\ {\bf 4} (2000), no.\ 2, 351--368, \arXiv{math/9902092}.

\bibitem[CMT]{Chahal-Meijer-Top} J.~Chahal, M.~Meijer, and J.~Top,
  \textit{Sections on certain {$j=0$} elliptic surfaces},
  Comment. Math. Univ. St. Paul. \textbf{49} (2000), 79--89,
  \arXiv{math/9911274}.
  
\bibitem[Ch]{Charles} F.~Charles, \textit{On the Picard number of K3
  surfaces over number fields}, Algebra Number Theory {\bf 8} (2014),
  no.\ 1, 1--17, \arXiv{1111.4117}.

\bibitem[Co1]{Cox} D.~A. Cox, \textit{{Mordell-Weil} groups of ellipitc
  curves over {$\C(t)$} with {$p_g=0$ or~$1$}}, Duke
  Math. J. \textbf{49} (1982), 677--689.

\bibitem[Co2]{Cox-primes} D.~A. Cox, \textit{Primes of the form
  $x^{2}+ny^{2}$.  Fermat, class field theory and complex
  multiplication}, John Wiley \& Sons, Inc., New York, 1989.

\bibitem[E]{Elkies} N.~ D.~ Elkies, \textit{Elliptic curves of high
  rank over $\Q$ and $\Q(t)$\/}, in preparation.

\bibitem[ES]{Elkies-Schuett} N.~D.~Elkies and M.~Sch\"utt,
  \textit{Modular forms and K3 surfaces\/}, Adv.\ Math.\ {\bf 240}
  (2013), 106--131, \arXiv{0809.0830}.

\bibitem[G]{Gross} B.~H.~Gross, \textit{Arithmetic on elliptic curves
  with complex multiplication\/}, Lecture Notes in Mathematics {\bf
  776}, Springer, Berlin, 1980, With an appendix by B. Mazur.

\bibitem[GS]{Gross-Zagier} B.~Gross, D.~Zagier, \textit{On singular
  moduli}, J. Reine Angew. Math. \textbf{355} (1985), 191--220.

\bibitem[I1]{Inose:quartic} H. Inose, \textit{On certain {K}ummer
  surfaces which can be realized as non-singular quartic surfaces in
  {$P^{3}$}}, J. Fac. Sci. Univ. Tokyo Sect.  IA Math. \textbf{23}
  (1976), no.~3, 545--560. \MR{0429915 (55 \#2924)}

\bibitem[I2]{Inose:singular-K3} H. Inose, \textit{Defining equations of
  singular {$K3$} surfaces and a notion of isogeny}, Proceedings of
  the {I}nternational {S}ymposium on {A}lgebraic {G}eometry ({K}yoto
  {U}niv., {K}yoto, 1977) (Tokyo), Kinokuniya Book Store, 1978,
  495--502. \MR{578868 (81h:14021)}
  
\bibitem[IR]{Ireland-Rosen} K.~Ireland and M.~Rosen, \textit{A Classical
  Introduction to Modern Number Theory}, Springer-Verlag, New York,
  1993.

\bibitem[Kl1]{Kloosterman-rank15} R.~Kloosterman, \textit{Elliptic K3
  surfaces with geometric Mordell-Weil rank 15\/},
  Canad.\ Math.\ Bull.\ {\bf 50} (2007), 215--226,
  \arXiv{math/0502439}.

\bibitem[Kl2]{Kloosterman-explicit} R.~Kloosterman, \textit{Explicit sections
  on Kuwata's elliptic surfaces\/}, Comment.\ Math.\ Univ.\ St.\ Pauli
  {\bf 54} (2005), no.\ 1, 69--86, \arXiv{math/0502017}.

\bibitem[Km]{Kumar} A.~ Kumar, \textit{$K3$ surfaces associated with
  curves of genus two\/}, Int.\ Math.\ Res.\ Not.\ IMRN \textbf{2008},
  no.~6, Art.\ ID rnm165, 26 pp, \arXiv{math/0701669}.

\bibitem[KK]{rank15} A.~Kumar and M.~Kuwata, \textit{Inose's
  construction and elliptic K3 surfaces with Mordell-Weil rank 15
  revisited\/}, Contemp.\ Math.\ (to appear), \arXiv{1604.00738}.

\bibitem[KS]{Kuwata-Shioda} M.~Kuwata and T.~Shioda, \textit{Elliptic
  parameters and defining equations for elliptic fibrations on a
  Kummer surface\/}, Algebraic geometry in East Asia---Hanoi 2005,
  177--215, Adv.\ Stud.\ Pure Math.\ {\bf 50}, Math.\ Soc.\ Japan,
  Tokyo, 2008, \arXiv{math/0609473}.

\bibitem[Kw1]{Kuwata:quartic} M.~Kuwata, \textit{Elliptic fibrations
  on quartic K3 surfaces with large Picard numbers}, Pacific
  J. Math. \textbf{171} (1995), no. 1, 231--243.

\bibitem[Kw2]{Kuwata:MW-rank} M.~Kuwata, \textit{Elliptic {$K3$} surfaces
  with given {M}ordell-{W}eil rank},
  Comment. Math. Univ. St. Paul. \textbf{49} (2000), 91--100.

\bibitem[KwW]{Kuwata-Wang} M.~Kuwata and L.~Wang, \textit{ Topology of
  rational points on isotrivial elliptic surfaces\/},
  Internat.\ Math.\ Res.\ Notices {\bf 1993}, no.\ 4, 113--123.

\bibitem[Ma]{Manin} Manin, Yu. I., \textit{Cubic forms. Algebra,
  geometry, arithmetic}. Translated from the Russian by
  M. Hazewinkel. Second edition. North-Holland Mathematical Library,
  4. North-Holland Publishing Co., Amsterdam, 1986.

\bibitem[Me]{Mestre:given-j-invariant} Jean-Fran{\c{c}}ois Mestre,
  \textit{Rang de courbes elliptiques d'invariant donn\'e},
  C. R. Acad. Sci. Paris S\'er. I Math. \textbf{314} (1992), no.~12,
  919--922. \MR{1168325 (93e:11075)}

\bibitem[Mo]{Morrison} D.~R.~Morrison, \textit{On K3 surfaces with
  large Picard number\/}, Invent.\ Math.\ {\bf 75} (1984) no.\ 1,
  105--121.

\bibitem[O]{Oguiso} K.~Oguiso, \textit{On Jacobian fibrations on the
  Kummer surface of the product of non-isogenous elliptic curves\/},
  J.\ Math.\ Soc.\ Japan {\bf 41} (1989), 651--680.

\bibitem[OS]{Oguiso-Shioda} K.~Oguiso and T.~Shioda, \textit{The
  {Mordell-Weil} lattice of a rational elliptic surface},
  Comment. Math. Univ. St. Paul. \textbf{40} (1991), 83--99.

\bibitem[P-SS]{P-SS} I.~Piatetski-Shapiro and I.~R.~Shafarevich,
  \textit{A Torelli theorem for algebraic surfaces of type K3\/},
  Math.\ USSR Izv.\ {\bf 5} (1971), 547--587.

\bibitem[PTvL]{Poonen-et-al} B.~Poonen, D.~Testa and R.~van Luijk,
  \textit{Computing N\'eron-Severi groups and cycle class groups\/},
  Compos.\ Math.\ {\bf 151} (2015) no.\ 4, 713--734,
  \arXiv{1210.3720}.

\bibitem[Ro]{Robert} G.~Robert, \textit{Nombres de Hurwitz et unit\'es
  elliptiques. Un crit\`ere de r\'egularit\'e pour les extensions
  ab\'eliennes d'un corps quadratique imaginaire\/}, Annales
  scientifiques de l'\'Ecole Normale Sup\'erieure {\bf 11}
  no.\ 3. Soci\'et\'e math\'ematique de France, 1978.

\bibitem[RS1]{Rubin-Silverberg} K.~Rubin and A.~Silverberg,
  \textit{Families of elliptic curves with constant mod p
    representations. Elliptic curves, modular forms, \& Fermat's last
    theorem\/}, 148--161, Ser.\ Number Theory, I, Int.\ Press,
  Cambridge, MA, 1995.

\bibitem[RS2]{RS:mod6} K.~Rubin and A.~Silverberg, \textit{Mod 6 representations of
  elliptic curves\/}, Automorphic forms, automorphic representations,
  and arithmetic (Fort Worth, TX, 1996), 213--220, Proc. Sympos. Pure
  Math., 66, Part 1, Amer. Math. Soc., Providence, RI, 1999.

\bibitem[Sch]{Schuett} M.~Sch\"utt, \textit{Fields of definition of
  singular K3 surfaces}, Communications in Number Theory and Physics
  {\bf 1} no.\ 2 (2007), 307-321, \arXiv{math/0612396}.

\bibitem[Se]{Segre} B.~Segre, \textit{On arithmetical properties of
  quartic surfaces}, Proc. London Math. Soc. (2) \textbf{49},
  (1947), 353--395.

\bibitem[Sh1]{Shioda:1973} T.~Shioda, \textit{On rational points of the
  generic elliptic curve with level $N$ structure over the field of
  modular functions of level $N$}, J. Math. Soc. Japan \textbf{25}
  (1973), 1--167.
    
\bibitem[Sh2]{Shioda:MWL} T.~Shioda, \textit{On the Mordell-Weil
  lattices}, Comment. Math. Univ. St. Pauli \textbf{39} (1990),
  211--240.

\bibitem[Sh3]{Shioda:2000} T.~Shioda, \textit{A note on {$K3$} surfaces
  and sphere packings}.  Proc. Japan Acad.  \textbf{76, {\rm Ser. A}}
  (2000), 51--82.
    
\bibitem[Sh4]{Shioda:Kummer-sandwich} T.~Shioda, \textit{Kummer sandwich
  theorem of certain elliptic {$K3$} surfaces}, Proc. Japan
  Acad. \textbf{82, {\rm Ser. A}} (2006), 137--140.

\bibitem[Sh5]{Shioda:correspondence} T.~Shioda, \textit{Correspondence of
  elliptic curves and {M}ordell-{W}eil lattices of certain elliptic
  {$K3$} surfaces} , in Algebraic Cycles and Motives, vol. 2,
  319--339, Cambridge Univ. Press (2007).

\bibitem[Sh6]{Shioda:2007} T.~Shioda, \textit{The Mordell-Weil lattice of
  {$y^2=x^3 + t^5 - 1/t^5 -11$}},
  Comment. Math. Univ. St. Paul. \textbf{56} (2007), 45--70.

\bibitem[Sh7]{Shioda:sphere-packing} T.~Shioda, \textit{{$K3$} surfaces
  and sphere packings}, J. Math. Soc. Japan \textbf{60} (2008),
  1083--1105.

\bibitem[SI]{Shioda-Inose} T.~Shioda and H.~Inose, \textit{On singular
  {$K3$} surfaces}, Complex analysis and algebraic geometry, Iwanami
  Shoten, Tokyo, 1977, 119--136.  \MR{0441982 (56 \#371)}

\bibitem[SM]{Shioda-Mitani} T.~Shioda and N.~Mitani, \textit{Singular
  abelian surfaces and binary quadratic forms}, in ``Classification of
  Algebraic Varieties and Compact Complex Manifolds,'' Lecture Notes
  in Mathematics \textbf{412}, Springer, Berlin, 1974, 259--287.

\bibitem[TdZ]{Top-de Zeeuw} J.~Top and F.~de Zeeuw, \textit{Explicit
  elliptic K3 surfaces with rank $15$}, Rocky Mountain J. of
  Math. {\bf 39} (2009), 1689--1698.
\end{thebibliography}
\end{document}